\documentclass[11pt]{article}
\usepackage{fullpage,amssymb}
\usepackage{caption}
\usepackage{subcaption}
\usepackage{graphicx}
\usepackage{amsmath}
\usepackage{float}
\usepackage{listings}
\usepackage{xcolor}
\usepackage{amsthm}
\usepackage{algorithmic}
\usepackage{algorithm}
\usepackage[colorlinks,
linkcolor=blue,
anchorcolor=blue,
citecolor=blue]{hyperref}
\usepackage{natbib}
\usepackage{comment}
\usepackage{bbm}
\usepackage{multirow}
\usepackage{array}
\usepackage[shortlabels]{enumitem}
\usepackage{romannum}
\AtBeginDocument{\pagenumbering{arabic}}
\usepackage{multirow}
\usepackage{stmaryrd}

\setcitestyle{authoryear,round}

\newtheorem{theorem}{Theorem}
\newtheorem{lemma}{Lemma}

\newtheorem{proposition}{Proposition}
\newtheorem{corollary}{Corollary}
\newtheorem{assumption}{Assumption}
\newtheorem{remark}{Remark}

\newcommand\ltwo[1]{\| #1 \|_{2}}

\newcommand{\inp}[2]{\langle #1,#2\rangle}

\def\calD{{\mathcal D}}
\def\calE{{\mathcal E}}

\def\calP{{\mathcal P}}

\def\calS{{\mathcal S}}

\def\bSigma{{\boldsymbol{\Sigma}}}

\def\bcalE{{\boldsymbol{\mathcal E}}}

\def\EE{{\mathbb E}}

\def\II{{\mathbb I}}

\def\PP{{\mathbb P}}

\def\RR{{\mathbb R}}

\def\TT{{\mathbb T}}

\def\ZZ{{\mathbb Z}}

\def\e{{\mathbf e}}

\def\g{{\mathbf g}}

\def\G{{\mathbf G}}

\def\I{{\mathbf I}}

\def\M{{\mathbf M}}

\def\X{{\mathbf X}}

\def\Z{{\mathbf Z}}

\def\bSigma{{\boldsymbol{\Sigma}}}

\def\Bbeta{{\boldsymbol{\beta}}}

\def\ku{C_{{\tiny u}}}
\def\kl{C_{{\tiny l}}}

\def\ca{C_{{\tiny a}}}
\def\cb{C_{{\tiny b}}}

\allowdisplaybreaks

\begin{document}
	\title{Online Quantile Regression}
	\author{Yinan Shen$^1$, Dong Xia$^2$, Wen-Xin Zhou$^3$\\
		~\\
		$^1$Department of Mathematics, University of Southern California\\
		$^2$Department of Mathematics, Hong Kong University of Science and Technology\\
		$^3$Department of Information and Decision Sciences, University of Illinois at Chicago}
	
	\date{}

	\maketitle
	\begin{abstract}
		This paper addresses the challenge of integrating sequentially arriving data into the quantile regression framework, where the number of features may increase with the number of observations, the time horizon is unknown, and memory resources are limited. Unlike least squares and robust regression methods, quantile regression models different segments of the conditional distribution, thereby capturing heterogeneous relationships between predictors and responses and providing a more comprehensive view of the underlying stochastic structure. We employ stochastic sub-gradient descent to minimize the empirical check loss and analyze its statistical properties and regret behavior. Our analysis reveals a subtle interplay between updating iterates based on individual observations and on batches of observations, highlighting distinct regularity characteristics in each setting. The proposed method guarantees long-term optimal estimation performance regardless of the chosen update strategy. Our contributions extend existing literature by establishing exponential-type concentration inequalities and by achieving optimal regret and error rates that exhibit only \emph{short-term} sensitivity to initialization. A key insight from our study lies in the refined statistical analysis showing that properly chosen stepsize schemes substantially mitigate the influence of initial errors on subsequent estimation and regret. This result underscores the robustness of stochastic sub-gradient descent in managing initial uncertainties and affirms its effectiveness in sequential learning settings with unknown horizons and data-dependent sample sizes. Furthermore, when the initial estimation error is well-controlled, our analysis reveals a trade-off between short-term error reduction and long-term optimality. For completeness, we also discuss the squared loss case and outline appropriate update schemes, whose analysis requires additional care. Extensive simulation studies corroborate our theoretical findings.
	\end{abstract}

	\section{Introduction}
	\label{sec:1}
	
	Online learning aims to efficiently incorporate sequentially arriving data and make timely predictions. In contrast to offline learning, where all data are collected and stored prior to analysis, online learning processes data in a sequential manner without requiring access to the entire sample at once, thereby alleviating both storage and computational burdens. As a result, online methods are particularly well-suited for large-scale datasets \citep{lecun1989backpropagation, rajalakshmi2019pattern, finn2019online}, streaming asset pricing data \citep{soleymani2020financial}, and the increasingly prominent domain of reinforcement learning \citep{gao2019batched, han2022online, ren2023dynamic}. Broader treatments and further applications of online learning can be found in \citet{bottou1999line}, \citet{cesa2006prediction}, \citet{hoffman2010online}, \citet{hazan2016introduction}, and \citet{orabona2019modern}. 
	
	In classical offline linear regression, estimation and inference are based on a pre-collected sample of independent observations $\{(\mathbf{X}_i,Y_i)\}_{i=1}^n$ satisfying $Y_{i}=\mathbf{X}_i^{\top}\boldsymbol{\beta}^* +\xi_{i}$, where $\boldsymbol{\beta}^* \in\mathbb{R}^{d}$ is the unknown parameter of interest, and $\xi_i$ denotes the random noise variable, satisfying $\mathbb{E}(\xi_{i}|\mathbf{X}_{i}) = 0$. It is well-known that the ordinary least squares estimator $\hat{\boldsymbol{\beta}}$ achieves the minimax rate of convergence, i.e., $\mathbb{E} \| \hat{\boldsymbol{\beta}} - \boldsymbol{\beta}^* \|_{2}^{2}=O(\sigma^2 d /n)$, where $\sigma^2$ is such that $\mathbb{E}(\xi_{i}^{2}|\mathbf{X}_i) \leq \sigma^2$ almost surely. In contrast, the entire sample is inaccessible in the online setting, and the total number of observations may even be unknown. Specifically, at time $t$, only $n_t$ pairs of observations can be used: 
	\begin{align}
		Y_i^{(t)}=\X_i^{(t)\top}\Bbeta^*+\xi_i^{(t)},\quad i=1,\dots,n_t. \label{eq:model}
	\end{align}
	While the data are received sequentially, online learning refrains from repeatedly using the data for model refitting. Instead, it sequentially updates the estimator. In contrast, offline methods leverage the entire sample, and based on the available data, the offline method \citep{he2000parameters,koenker2005quantile} leads to a statistically optimal estimator with mean squared error of order $O(\sigma^2 \cdot d/\sum_{t=0}^{T}n_t)$, where $T$ is called the \emph{horizon}. 
	This prompts the question of whether online methods can achieve error rates comparable to their offline counterparts. Must we compromise accuracy for more efficient computation and reduced storage in online learning? Furthermore, with the accumulation of more information over time, the theoretical offline error rates decline. Is it feasible to sustain this downward trend in optimal error rates over time? Particularly when the horizon (maximum time $T$) is unknown, isolated error rates at specific times lack persuasiveness.

	In online learning, a learner must generate timely predictions as data arrive sequentially, either individually or in batches. The concept of \textit{regret}, first introduced by \cite{savage1951theory} and later popularized in the online learning literature \citep[e.g.,][]{goldenshluger2013linear, han2020sequential, ren2023dynamic}, quantifies cumulative prediction error over time and serves as a key measure of learning performance. The regret at horizon $T$ is defined as
	\begin{align}
		\textsf{Regret}_T:= \EE\Bigg\{\sum_{t=0}^{T} f_{t}(\Bbeta_{t})-f_t(\Bbeta^*) \Bigg\}, 
		\label{eq:regret}
	\end{align}
	where $f_t(\cdot)$ and $\Bbeta_{t}$ denote the loss function and estimator at time $t$, respectively. The expectation
	is taken over the data from $t=0$ to $T$. Regret serves as a metric to evaluate the efficacy of a sequential estimation scheme. Optimal regret occurs when error rates exhibit a scaling behavior of $1/t$ as time $t$ becomes sufficiently large \citep{hazan2016introduction}. In contrast, \cite{cai2023online} assumes a known horizon and focuses more on the ultimate accuracy of estimation. However, this setting bears resemblance to the sample splitting technique \citep{xia2021statistical, zou2020consistent}.

	Online estimation under the squared loss in equation~\eqref{eq:regret} has been extensively studied over the past two decades; see, for example, \cite{bottou1999line}, \cite{zinkevich2003online}, \cite{hazan2007adaptive}, and \cite{langford2009sparse}, among others.
	Stochastic gradient descent stands out as a natural and elegant methodology for processing sequentially arriving data, integrating streaming information, and minimizing the regret in \eqref{eq:regret}. \citet{bottou1999line} lays the foundation for the general theoretical framework of online learning, establishing the asymptotic convergence of gradient descent toward stationary points. Under suitable conditions, the seminal work of \citet{zinkevich2003online} proves the iterative convergence of gradient descent with a stepsize of order $O(t^{-1/2})$. Subsequently, \citet{hazan2007logarithmic} demonstrates that using a stepsize of order $O(t^{-1})$ yields an optimal regret bound of order $O(\log T)$. The pioneering work of \citet{duchi2009efficient} further examined the performance of iterative sub-gradient descent applied to non-differentiable loss functions, offering insights into convergence dynamics and establishing regret bounds under various stepsize schemes. We refer to \cite{hazan2007adaptive}, \cite{langford2009sparse}, \cite{streeter2010less}, \cite{mcmahan2010adaptive}, \cite{duchi2011adaptive}, \cite{bach2013non}, and the references therein \cite{orabona2019modern} for additional gradient descent-based algorithms in online learning. In the offline learning context, stochastic (sub-)gradient descent serves as an effective approach to mitigate computational burdens \citep{zhang2004solving}. This method updates the current iterate using only a single, randomly selected observation pair, rather than the entire sample. \citet{rakhlin2011making} introduces a stepsize scheme of order $O(1/t)$, yielding a sub-linear convergence rate comparable to that of its online counterpart under a strongly convex objective. In offline stochastic optimization, the total sample size is typically known, allowing such information to be incorporated into the stepsize design and ensuring that the statistical optimality remains stable over time. Examples of constant stepsize schemes that explicitly account for the learning horizon can be found in \citet{duchi2009efficient}, \citet{cai2023online}, and \citet{puchkin2023breaking}. Other studies, such as \citet{delyon1993accelerated}, \citet{roux2012stochastic}, and \citet{johnson2013accelerating}, focus on accelerating the sub-linear convergence rate of stochastic descent methods. While both offline stochastic gradient descent and online learning alleviate computational pressures, they differ fundamentally in nature. In online learning, the time horizon and the total number of observations are unknown, and may even be infinite, necessitating procedures that achieve dynamic and statistically optimal error rates over time. Moreover, an online learner often must generate predictions for newly arriving covariate-only inputs, with predictive accuracy evaluated through the notion of regret.
	
	The statistical analysis of the squared loss remains a delicate yet crucial area of study. Although online learning and (stochastic) gradient descent methods have been extensively investigated from an optimization perspective for both squared and non-differentiable loss functions, most of the existing literature assumes that the empirical loss possesses certain regularity properties, such as strong convexity and/or smoothness. Moreover, the feasible domain is often restricted to a bounded convex set \citep{hazan2007logarithmic, duchi2009efficient, langford2009sparse, rakhlin2011making}. Comprehensive overviews of online learning algorithms from an optimization standpoint can be found in \citet{hazan2016introduction} and \citet{orabona2019modern}. However, empirical strong convexity may fail to hold even under the squared loss when the available storage size is limited. Addressing this issue calls for refined statistical analysis that explicitly accounts for the interplay between data dependence, sample size, and memory constraints; see, for example, \citet{chen2019quantile}, \citet{han2023online}, and the references therein. \citet{fan2018statistical} further demonstrates that an expected error rate of order $O(\log(T)/T)$ is attainable for sparse online regression under a stepsize scheme of order $O(1/t)$, yielding an optimal $O(\log T)$ regret.

	The preceding statistical online frameworks exhibit notable limitations when applied to quantile regression, often yielding suboptimal error rates and weakened guarantees on success probability, particularly in the presence of heavy-tailed noise. Existing studies on renewable or online quantile regression primarily analyze asymptotic performance within batch learning frameworks, where the batch size diverges over time. A common feature of these methods is their reliance on storing historical summary statistics, which incurs a memory cost of order $O(d^2)$ \citep{jiang2022renewable, sun2023online, chen2024renewable}. In contrast, the present work complements prior studies by developing online quantile regression methods that eliminate the need to store either historical data or summary statistics. For comparison, Section~\ref{sec:simulation} provides a parallel discussion of online least squares methods and their corresponding stepsize schemes. To estimate conditional quantile models sequentially, we adopt a simple yet effective sub-gradient descent approach. Through refined statistical analysis, we design a stepsize sequence that either exhibits geometric decay, remains constant, or decays at the rate of $1/t$, depending on the proximity of the current iterate to the oracle solution. This design departs from the scheme proposed by \citet{duchi2009efficient}; a detailed comparison between the two approaches is provided in Section~\ref{sec:simulation}, highlighting the advantages of the proposed statistically informed stepsize rule. Under heavy-tailed noise, we further establish exponential-type tail bounds for the proposed online estimators, improving upon the polynomial-type guarantees in \citet{han2022online}. This enhancement in tail behavior underscores the robustness and reliability of our estimators when applied to heavy-tailed data environments.

	In this work, we refer to online learning as a setting in which data arrive sequentially over time. The complete dataset and the total number of observations are never simultaneously available, regardless of how many observations can be processed at each step. Under the linear model \eqref{eq:model}, we consider three representative scenarios: (i) the newly arriving dataset $\mathcal{D}_t := \{ (\mathbf{X}_i^{(t)}, Y_i^{(t)}) \}_{i=1}^{n_t}$ contains only a single observation pair, requiring $O(d)$ storage; (ii) $\mathcal{D}_t$ contains at least $O(d)$ samples; and (iii) the server can store an unbounded number of samples. The first setting has been extensively examined in the literature \citep{duchi2009efficient, hazan2007logarithmic, hazan2016introduction, cai2023online, langford2009sparse, bach2013non, mhammedi2019lipschitz, han2022online}, yet it continues to demand careful statistical treatment, as discussed earlier. \citet{do2009proximal} further studies the case where data arrive in batches at each iteration. From an optimization standpoint, however, there is little distinction between single-sample and batch arrivals, since regularity conditions such as smoothness or strong convexity are typically assumed. From a statistical perspective, by contrast, we will demonstrate that these two settings exhibit markedly different behaviors. Specifically, this paper seeks to address the following fundamental questions: Is it possible to achieve a continuously decreasing error rate as new data arrive under limited storage? Can statistical optimality be maintained along the learning trajectory? How much information is lost relative to the offline benchmark? In online learning, how do failure probabilities accumulate across iterations, and how do they determine the maximal horizon for guaranteed convergence? What are the implications for per-iteration success probability, especially under heavy-tailed noise and single-sample updates? Finally, what role does initialization play? Does the initial error necessarily have a long-term influence on subsequent errors and regret, or can arbitrary initialization be tolerated?
	
	We summarize our main contributions addressing these questions as follows.

	\begin{enumerate}[1.]
		\item This paper develops a statistical analysis framework for online quantile regression, addressing the challenges arising from the empirical loss function's lack of strong convexity and smoothness. To achieve long-term statistical optimality under heavy-tailed noise, we design a stepsize scheme informed by the statistical regularity properties of the quantile loss. Unlike prior studies, the derived error rates explicitly scale with the noise level, capturing the intrinsic stochastic complexity of the problem. Furthermore, the proposed online approach deliberately trades a marginal loss in statistical efficiency for substantial gains in computational scalability and memory efficiency relative to offline methods. As demonstrated in Section~\ref{sec:simulation}, this slight reduction in accuracy becomes asymptotically negligible as the time horizon grows. 
		
		\item Our algorithm attains statistically optimal rates even under heavy-tailed noise and admits a high-probability guarantee, with the failure probability decaying exponentially fast in the dimension. This property ensures the validity of our theoretical results even when the unknown horizon grows exponentially with dimension. For instance, when only a single datum arrives at each iteration, the failure probability associated with our established error rate is $O(\exp(-c_0 d))$, where $c_0 > 0$ is a universal constant. Consequently, the horizon can be as large as $T = O(\exp(c_0 d))$ without compromising the theoretical guarantees. In sharp contrast, the results in \citet{han2022online} and \citet{cai2023online} hold with probability $1 - d^{-O(1)}$ under sub-Gaussian noise, implying that their admissible horizon is limited to $T = O(d^{O(1)})$.
		
		\item Our analysis demonstrates that the influence of the initial error on both estimation accuracy and regret is merely transient. Consequently, the derived statistical error bounds remain valid for arbitrary initializations, without requiring the initial estimator to lie within a compact region. This robustness is achieved through refined statistical analyses and the design of a novel stepsize scheme. Stochastic (sub-)gradient descent methods for both quantile and squared losses benefit substantially from stepsize schedules informed by underlying statistical properties. For instance, in the setting where $\xi_i^{(t)} \sim N(0, \sigma^2)$ and $n_t=1$ in \eqref{eq:model}, we establish the regret bound $\sum_{t=0}^T\EE\|\Bbeta_t-\Bbeta^*\|^2\leq O(\|\Bbeta_{0}-\Bbeta^*\|_2^2 + \sigma^2\log T)$.
	\end{enumerate}

	The remainder of this paper is organized as follows. Section~\ref{sec:linear_regression} introduces the quantile (check) loss function and the online sub-gradient descent algorithm. It further examines the convergence dynamics and regret bounds, with detailed discussions of the three distinct settings considered. Section~\ref{sec:simulation} presents extensive numerical experiments, including evaluations of the proposed stepsize scheme, comparisons of accuracy with offline regression methods, and analyses of convergence behavior. A brief discussion of online least squares from a statistical perspective is also provided. The proof of Theorem~\ref{thm:one_sample} appears in Section~\ref{sec:proof_online}, while the proofs of the remaining theorems are deferred to the supplementary material.

	\section{Online quantile regression via adaptive sub-gradient descent}
	\label{sec:linear_regression}
	
	Throughout this section, we consider the conditional quantile model \eqref{eq:model}, that is, the conditional $\tau$-th quantile of $\xi_i^{(t)}$ given $\X_i^{(t)}$ is zero for some predetermined $\tau \in (0, 1)$. Quantile regression (QR) plays a crucial role in unraveling pathways of dependence between the outcome and a collection of features, which remain elusive through conditional mean regression analysis, such as the least squares estimation. Since its introduction by \cite{koenker1978regression}, QR as undergone extensive study, encompassing theoretical understanding, methodological development for various models and data types, and practical applications in a wide range of fields. The main focus of existing literature centers around the formulation of methodologies and theories for QR utilizing static data, where a complete dataset is available. This is referred to as the offline setting in online learning literature. The most fundamental and well-studied method for estimating $\Bbeta^*$ involves empirical risk minimization, or statistical $M$-estimation. The associated loss function $\rho_{Q,\tau}(x):=\tau x\mathbb{I}(x\geq 0)+(\tau-1)x\mathbb{I}(x<0)$ is known as the check loss or quantile loss. We refer to \cite{koenker2005quantile} and \cite{koenker2017handbook} for a comprehensive exploration of offline quantile regression, covering statistical theory, computational methods, as well as diverse extensions and applications.
	
	In contrast to the least squares method, the non-smooth nature of the quantile loss introduces additional challenges to QR, particularly in the era of big data. Established algorithms employed for this purpose include the simplex algorithm \citep{koenker1987algorithm}, interior point methods with preprocessing \citep{portnoy1997gaussian}, alternating direction method of multipliers, and proximal methods \citep{parikh2014proximal}. More recently, \cite{FGH2021} and \cite{he2021smoothed} have shown that convolution-smoothed quantile regression attains desired asymptotic and non-asymptotic properties, provided that the smoothing parameter (bandwidth) is appropriately selected as a function of the sample size and dimensionality.
	When addressing low-rank matrix regression or completion problems under either absolute or quantile loss, recent studies, including \citet{cambier2016robust}, \citet{li2020nonconvex}, \citet{charisopoulos2021low}, \citet{tong2021low}, \citet{ma2023global}, and \citet{shen2023computationally}, have conducted extensive investigations of sub-gradient descent methods from both computational and statistical perspectives. The offline methods mentioned above crucially hinge on specific regularity properties inherent in the empirical loss function. However, in the online setting, characterized by sequential data arrival and an unknown total number of observations, the assurance of offline regularity properties becomes untenable with only a limited number of available samples. Consequently, neither the offline results nor the corresponding proof techniques are applicable in this online context.
	As an illustration, in an offline setting based on model (\ref{eq:model}), \cite{shen2023computationally} demonstrated the existence of certain parameters $\mu_1, \mu_2>0$ such that with high probability, 
	$$
	\frac{1}{n}\sum_{i=1}^n ( \vert Y_i-\inp{\X_i}{\Bbeta} \vert- \vert Y_i-\inp{\X_i}{\Bbeta^*} \vert )\geq \max\{\mu_1\ltwo{\Bbeta-\Bbeta^*}-\gamma, \mu_2\ltwo{\Bbeta-\Bbeta^*}^2\}
	$$
	holds for all $\Bbeta$ under the sample size requirement $n\geq Cd$, $\gamma=\EE|\xi|$ is the first absolute moment of the noise variable. However, in the context of online learning, where the stored data is limited, such as when $n=1$, the above inequality does not hold in general.

	The foundational contributions to online learning by \cite{duchi2009efficient}, \cite{duchi2011adaptive} and \cite{johnson2013accelerating} employ sub-gradient descent to address scenarios characterized by the lack of differentiability in the loss function at specific points. These studies establish the properties of iterates under various stepsize schemes, with a primary emphasis on the excess risk, leaving the analysis of estimation error  $\ltwo{\Bbeta_t-\Bbeta^*}$ unaddressed. In the absence of strong convexity and/or certain smoothness, the upper bounds on the excess risk cannot be straightforwardly extended to those on the estimation error.
	From the statistical viewpoint, \cite{jiang2022renewable} and \cite{wang2022renewable} consider online quantile regression within the batch learning framework in which the data arrive in batches, and establish their asymptotic normality as the batch size goes to infinity. It is worth noticing that the requirement for diverging batch sizes renders these methods impractical in scenarios where only one observation or a few observations arrive at a time.

	This paper is dedicated to the exploration of online quantile regression with unknown horizon and sample sizes (batch or total). We aim to provide a non-asymptotic analysis of online sub-gradient descent. By using a customized stepsize scheme with explicit dependencies on dimensionality, sample size, and noise scale, we establish optimal rates of convergence for online QR estimators. Additionally, we seek to demonstrate that near-optimal regret performance can be achieved. Let $\calS_t$ be the set of observations acquired at time $t$, which will be used to update the \emph{current} iterate $\Bbeta_{t}$. Consequently, the loss function at $t$ is given by
	$$
	f_t(\Bbeta):=\sum_{(\X_i^{(t)},Y_i^{(t)})\in\calS_t}\rho_{Q,\tau}(Y_i^{(t)}- \langle \X_i^{(t)}, \Bbeta \rangle).
	$$ 
	We will explore three settings: (i) $\calS_t=\calD_t$, containing only one data vector, referred to as the {\it online learning} setting, (ii) $\calS_t=\calD_t$, comprising a set of observations with a size of at least $O(d)$, termed as the {\it online batch learning} setting, and (iii) $\calS_t$ containing all the data vectors accumulated up to time $t$, expressed as $\calS_t=\cup_{l=0}^t\calD_l$, recognized as the {\it sequential learning with infinite storage} setting.

	\emph{Online sub-gradient descent.} Initiated at some $\Bbeta_0$, the online QR estimates are iteratively defined via sub-gradient descent as 
	$$
	\Bbeta_{t+1}=\Bbeta_{t}-\eta_t\g_t, \quad t= 0, 1, \ldots, 
	$$
	where $\g_t\in\partial f_t(\Bbeta_{t})$ is the sub-gradient of $f_t$ at $\Bbeta_{t}$, and $\{ \eta_t \}_{t=0, 1, \ldots }$ constitutes a sequence of stepsize parameters (learning rates). The loss function varies across different settings, exhibiting distinct regularity properties. Consequently, tailored stepsize schemes are imperative in the three settings to attain the desired convergence properties. We will show that these customized stepsize schemes yield statistically optimal estimators and achieve near-optimal regret performances. Importantly, these schemes effectively adapt to varying dimensions and unknown horizons.
	
	Prior to presenting the theoretical guarantees and convergence dynamics, we outline the essential assumptions concerning the model.
	
	\begin{assumption}[{\it Random covariate}]	\label{assump:sensing_operators:vec}
		The covariate vector $\X\in\RR^d$ follows the Gaussian distribution $N(\boldsymbol{0},\bSigma)$, where $\bSigma$ is symmetric and positive definite. There exist absolute constants $\kl,\ku>0$ such that $\kl\leq\lambda_{\min}(\bSigma)\leq\lambda_{\max}(\bSigma)\leq\ku$.
	\end{assumption}
	
	\begin{assumption}[{\it Noise distribution}] \label{assump:heavy-tailed}
		Given $\X \in \RR^d$, the noise variable $\xi$ has the conditional density function $h_{\xi|\X}(\cdot)$ and distribution function $H_{\xi|\X}(\cdot)$, respectively. The conditional $\tau$-th quantile of $\xi$ given $\X$ is zero, i.e., $H_{\xi|\X}(0)=\tau$ and $\EE |\xi|<+\infty$. Let $\gamma=\EE|\xi|$. 
		There exist constants $b_0, b_1>0$ such that \footnote{As marked in \cite{shen2023computationally}, the noise assumptions can be slightly relaxed to $|H_{\xi|\X}(x)-H_{\xi|\X}(0)|\geq |x|/b_0$ for $|x|\leq 8(\ku/\kl)^{1/2}\gamma$ and $|H_{\xi|\X}(x)-H_{\xi|\X}(0)|\leq |x|/b_1$. }
		\begin{enumerate}[(a)]
			\item $\inf_{|x|\leq 8(\ku/\kl)^{1/2}\gamma}h_{\xi|\X}(x)\geq b_0^{-1}$,
			\item $\sup_{x\in \RR} h_{\xi|\X}(x)\leq b_1^{-1} $.
		\end{enumerate}
		
	\end{assumption}
	
	Both local lower and global upper bounds on the (conditional) density $h_{\xi|\X}$ are commonly imposed in the QR literature, particularly in non-asymptotic settings (e.g., \cite{BC2011} and \cite{he2021smoothed}).

	\subsection{Online Learning}\label{sec:onesample}
	
	The server receives a new data point $(\X_t,Y_t)$ that follows the linear model \eqref{eq:model}. The online sub-gradient descent algorithm calculates the corresponding sub-gradient $\g_t \in \partial f_t(\Bbeta_t)$ and subsequently updates the current estimate $\Bbeta_t$ in the direction of the negative sub-gradient, all {\it without} storing the new observation. Here, the loss function at time $t$ is given by $f_t(\Bbeta)=\rho_{Q,\tau}(Y_t-\X_t^{\top}\Bbeta)$. More specifically, the algorithm proceeds as follows. We begin with an arbitrary initialization $\Bbeta_0$. At each time step $t$, given the current estimate $\Bbeta_t$ and a newly arrived observation $(\X_t,Y_t)$, the update is performed as 
	$$
	\Bbeta_{t+1}=\Bbeta_{t}-\eta_t\cdot (\tau-\mathbb{I}\{Y_t-\langle\X_t,\Bbeta_t\rangle\})\cdot \X_t,
	$$ 
	where $\eta_t>0$ denotes the stepsize. The choice of $\eta_t$ will be discussed in Theorem~\ref{thm:one_sample}.
	
	The expected excess risk $\EE[f_t(\Bbeta)-f_t(\Bbeta^*) ]$, as a function of $\Bbeta$, demonstrates a two-phase regularity property, visually depicted in Figure~\ref{fig:regularity-one}. As an illustrative example, under the assumptions $\X\sim N(\boldsymbol{0},\I_d)$ and $\xi\sim N(0,\sigma^2)$, \cite{shen2023computationally} derived the population excess risk as
	$$
	\EE \big\{ |\X^{\top}(\Bbeta-\Bbeta^*)+\xi|-|\xi|\big\}=\sqrt{\frac{\pi}{2}}\frac{\|\Bbeta-\Bbeta^*\|^2}{\sqrt{\|\Bbeta-\Bbeta^*\|^2+\sigma^2}+\sigma}.
	$$When $\Bbeta$ is distant from the population risk minimizer $\Bbeta^{\ast}$, the expected excess risk exhibits a first-order lower bound, decaying linearly in $\| \Bbeta-\Bbeta^{\ast} \|_2$. However, as $\Bbeta$ approaches proximity to $\Bbeta^{\ast}$, a quadratic lower bound emerges concerning $\|\Bbeta-\Bbeta^{\ast}\|_2$. For general noise with bounded density, a similar equation is achievable and a two-phase property exists accordingly. Importantly, it is noteworthy that, in contrast to offline works such as \citep{tong2021low, shen2023computationally}, this regularity property pertains to the expectation, given that the empirical loss is based on only one single observation.
	The empirical loss lacks a guaranteed high probability concentration property; for instance, the variance of $f_t(\Bbeta)$ may significantly overshadow its expectation. As a result, the commonly observed monotone-decreasing trend in the estimation error, such as $\ltwo{\Bbeta_{t+1}-\Bbeta^* }\leq \ltwo{\Bbeta_{t}-\Bbeta^* }$ in offline learning algorithms, does not occur. In fact, a more nuanced analysis is essential to understand the convergence dynamics of online sub-gradient algorithms for quantile regression.
	
	\begin{figure}
		\centering
		\includegraphics[width=0.5\textwidth]{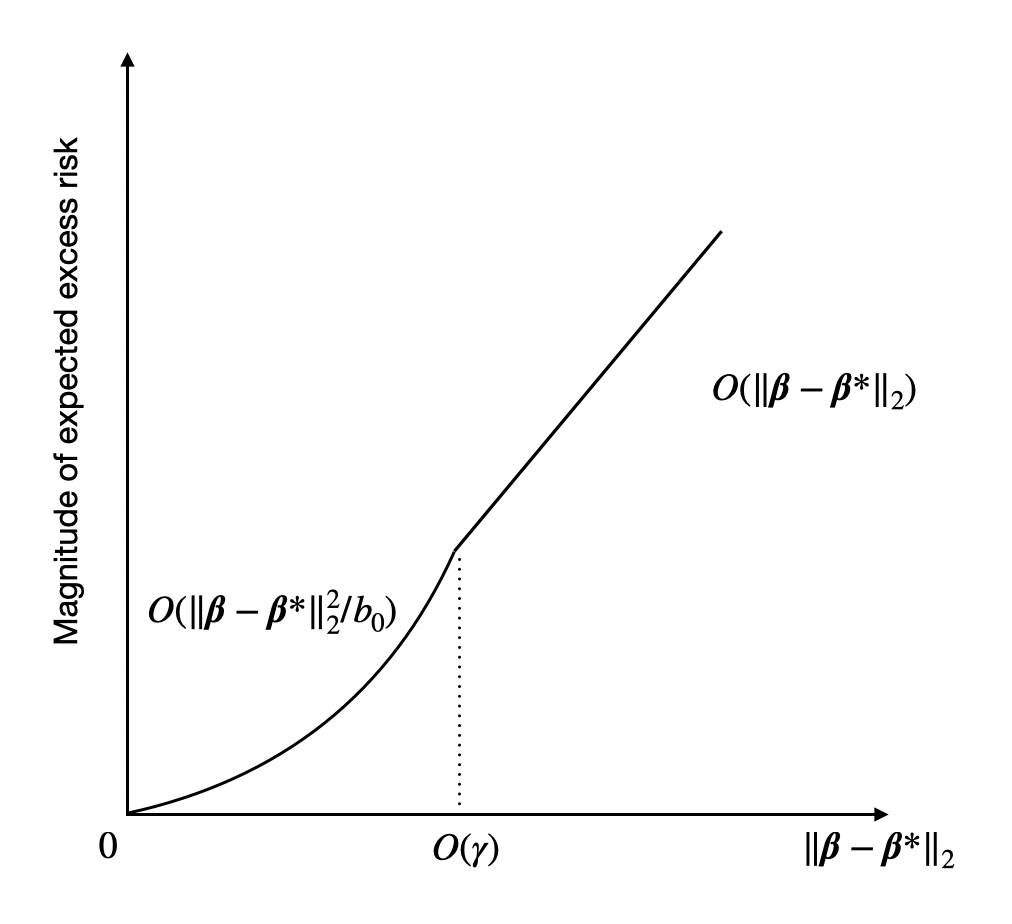}
		\caption{{\it Expected excess risk of the objective function.}   
			$Y$-axis: lower bound of $\EE[f_t(\Bbeta)-f_t(\Bbeta^*) ]$; $X$-axis: the value of $\|\Bbeta-\Bbeta^{\ast}\|_2$. It shows that the lower bound varies from a linear to quadratic dependence on $\|\Bbeta-\Bbeta^{\ast}\|_2$ as $\Bbeta$ gets closer to $\Bbeta^*$. 
		}
		\label{fig:regularity-one}
	\end{figure}

	Interestingly, the following theorem illustrates that, despite the online sub-gradient descent algorithm's inability to ensure error contraction as $\ltwo{\Bbeta_{t+1}-\Bbeta^* }<\ltwo{\Bbeta_{t}-\Bbeta^* }$, the error rates $\|\Bbeta_t-\Bbeta^{\ast}\|_2$ can be bounded by a monotone-decreasing sequence with high probability. The convergence of the online sub-gradient descent algorithm unfolds in two distinct phases, as suggested by the two-phase regularity properties depicted in Figure~\ref{fig:regularity-one}.

	\begin{theorem}
		Under Assumptions~\ref{assump:sensing_operators:vec} and \ref{assump:heavy-tailed} (a), there exist universal positive constants $c_0,c_1,c_2,c_3,C_0$ such that, for an arbitrary initialization $\Bbeta_{0} \in\RR^d$, the sequence $\{\Bbeta_t\}_{t\geq  1}$ generated by the online sub-gradient descent algorithm follows the dynamics outlined below:
		\begin{enumerate}
			\item in phase one\footnote{With this geometrically decaying stepsize, once $\|\Bbeta_t-\Bbeta^*\|_2\asymp O(\gamma)$ is achieved, the estimation error will remain to be this scale both theoretically and practically.}, i.e., when $\ltwo{\Bbeta_{t}-\Bbeta^*}\geq 8 \kl^{-1/2}\gamma$, by selecting a stepsize $\eta_0= \frac{C\sqrt{C_l}}{C_u \bar{\tau}^2 } \frac{D_0}{d} $ and $\eta_t=(1-c_5\frac{\kl}{\ku}\frac{1}{d\bar{\tau}^2})\eta_{t-1}$ with any $D_0\geq\|\Bbeta_0-\Bbeta^*\|$, $C>1$, $c_5<0.1$, it holds with probability at least  $1-\exp(-c_0d)$ that 
			$$\ltwo{\Bbeta_{t+1}-\Bbeta^*}^2\leq C_0\left(1-c_5\frac{\kl}{\ku}\frac{1}{\bar{\tau}^2 d}\right)^{t+1}\cdot\ltwo{\Bbeta_{0}-\Bbeta^*}^2; \footnote{We remark that $C_0$ depends on the choices of $C,c_5$ and we leave the details in the proof.}$$
			the conclusion of phase one occurs after $t_1=O(d\log(\ltwo{\Bbeta_{0}-\Bbeta^*}/\gamma))$ iterations, where $\bar \tau = \max\{\tau, 1-\tau\}$.
			
			\item in phase two, i.e., when $\ltwo{\Bbeta_{t}-\Bbeta^*}< 8\kl^{-1/2} \gamma$, by choosing the stepsize $\eta_t=(b_0/\kl)\cdot(\ca/(t-t_1+\cb d))$ with certain constants $\ca>12$ and $\cb >\ca/(12d)$, it holds with probability exceeding $1-\exp(-c_0d)$ that
			$$
			\ltwo{\Bbeta_{t+1}-\Bbeta^*}^2\leq C^*\frac{\ku}{\kl^2}\frac{d}{t+1-t_1+\cb d} b_0^2,
			$$
			where the constant $C^*$ depends on $\ca,\cb$.
		\end{enumerate}
		\label{thm:one_sample}
	\end{theorem}
	
	In accordance with Theorem~\ref{thm:one_sample}, the algorithm exhibits linear convergence during the first phase, concluding after $t_1$ iterations. This phase requires a total of $t_1$ observations and employs an approximately geometrically decaying stepsize scheme. In the second phase, the algorithm converges sub-linearly, employing a  $O(1/(t+d))$ decaying stepsize scheme. The failure probability in both phases decays exponentially in $d$.
	To our knowledge, our result stands as the first instance in online learning literature to provide an exponential-type high probability bound. Previous works employing the squared loss generally guarantee a polynomial-type tail probability, even under sub-Gaussian noise \citep{cai2023online,han2022online,bastani2020online}.  This enhancement holds significant implications for practical applications. It suggests that our algorithm remains applicable even when the ultimate horizon $T$ is as large as  $e^{O(d)}$. Theorem \ref{thm:one_sample} asserts that, provided $d\ll T\ll e^{c_0d}$, the final estimator $\Bbeta_T$, with high probability, attains an error rate that is minimax optimal under offline settings. Online methods can adapt to the unknown horizon, and offer substantial reductions in computation costs, as well as savings in memory and storage.

	The final error rate achieved after the second phase iterations is proportional to the noise level, independent of the initialization error. In existing works, the initialization error often exerts a lasting impact on the final error rate, even when using the squared loss. For instance, the online algorithm \cite{fan2018statistical} obtains an expected error rate $\EE\ltwo{\Bbeta_T-\Bbeta^*}^2\leq d\ltwo{\Bbeta_0-\Bbeta^*}^2\log (T)/T$. Moreover, the error rate  established by \cite{cai2023online} is $\ltwo{\Bbeta_T-\Bbeta^*}^2\leq d\ltwo{\Bbeta_0-\Bbeta^*}^2/T$, which holds with high probability as long as $T\ll {\rm poly}(d)$. Theorem~\ref{thm:one_sample} reveals that the initialization error only has a short-time effect in the online sub-gradient descent algorithm, dissipating after the second phase iterations begin.  This two-phase convergence phenomenon is elucidated through a meticulous analysis of the algorithm's dynamics. It is worth noting that the two-phase convergence and the transient impact of initialization error are not exclusive to quantile regression; they also manifest in least squares, as per our analysis framework. A comprehensive discussion of these properties for online least squares will be presented in Section~\ref{sec:simulation}. Numerical simulations in that section demonstrate that a two-phase stepsize scheme yields significantly improved accuracy.

	\begin{corollary}	\label{cor:one_sample}
		Assume that the same conditions as in Theorem~\ref{thm:one_sample} hold, and let the horizon satisfy $T\geq C^*d \max\{1, \log(\sqrt{\kl}\ltwo{\Bbeta_{0}-\Bbeta^*}/\gamma)\}$ with a constant $C^{\ast}$ depending on $(\kl,\ku,\ca,\cb)$. Then, with probability at least $1-T\exp(-c_0d)$, the online sub-gradient descent algorithm produces a final estimate with an error rate
		$$
		\ltwo{\Bbeta_T-\Bbeta^*}^2\leq C_3\frac{d}{T}b_0^2,
		$$
		where $C_3>0$ is a constant. 
	\end{corollary}

	Corollary~\ref{cor:one_sample} demonstrates that the online sub-gradient descent algorithm attains the minimax optimal error rate, provided the horizon is not too small. This makes it a favorable choice over the offline approach, considering the benefits of reduced computation and storage costs. However, it is advisable to opt for offline approaches in cases where the horizon is small. A comprehensive examination of their numerical performances is conducted in Section~\ref{sec:simulation}.

	\begin{remark}[{\it Trade-off between short-term accuracy and long-term optimality}]
		Suppose the initialization is already situated in the second phase region, i.e., $\ltwo{\Bbeta_{0}-\Bbeta^*}< 8 \kl^{-1/2}\gamma$. It is advisable to initiate the phase two iterations directly ($t_1=0$), where the choice of the parameter $C_a$ plays a pivotal role in determining short-term accuracy and long-term optimality. Selecting $C_a>12$, according to Theorem~\ref{thm:one_sample}, the updated estimates achieve an error rate of $\ltwo{\Bbeta_{t}-\Bbeta^*}^2\leq C^*\frac{\ku}{\kl^2}\frac{d}{t+\cb d} b_0^2$ with high probability. This upper bound may exceed the initialization error when $t$ is small, indicating potential accuracy fluctuations in the early stage. Nevertheless, the algorithm eventually converges and yields a statistically optimal estimator.
		On the contrary, opting for a sufficiently small parameter $\ca<\min\{12, (\kl/\sqrt{\ku})\ltwo{\Bbeta_{0}-\Bbeta^*}/b_0\}$ results in a more stable convergence in the early stage. However, it sacrifices long-term accuracy, as the updated estimates possess an error bound
		$$
		\ltwo{\Bbeta_{t}-\Bbeta^*}^2\leq \bar{C}^*\left(\frac{d}{t+\cb d}\right)^{\frac{\ca}{12}}\cdot \ltwo{\Bbeta_{0}-\Bbeta^*}^2,
		$$
		achieving a sub-optimal error rate in the end. 
		This phenomenon is illustrated in Figure~\ref{fig:CaCb}. Specifically, in Figure~\ref{fig:CaCb0.2}, when the initialization satisfies $\ltwo{\Bbeta_0-\Bbeta^*}\leq 0.2\EE|\xi|$, it demonstrates that a small $C_a$ ensures short-term accuracy but compromises long-term optimality. On the other hand, a large $C_a$ guarantees long-term optimality at the expense of short-term accuracy. Figure~\ref{fig:CaCb0.02} presents even more pronounced differences, where the initialization satisfies $\ltwo{\Bbeta_0-\Bbeta^*}\leq 0.02\EE|\xi|$.
		\label{rmk:one_sample_init}
	\end{remark}
	
	\begin{figure}
		\centering
		\begin{subfigure}[b]{0.45\textwidth}
			\centering
			\includegraphics[width=\textwidth]{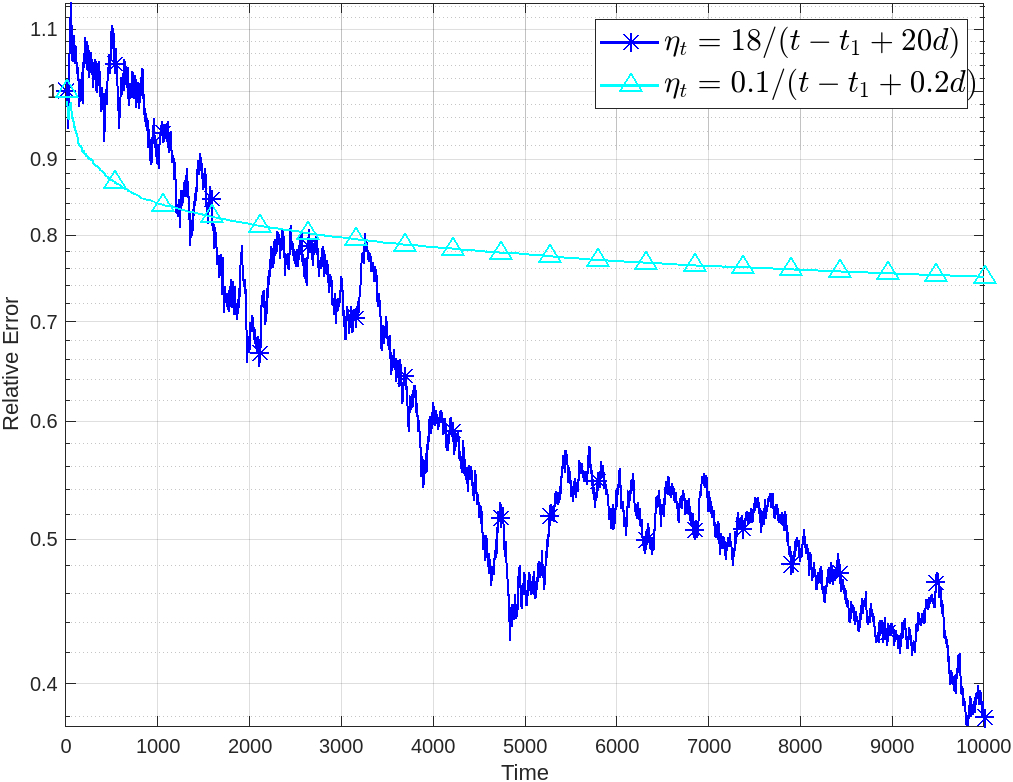}
			\caption{$\frac{\ltwo{\Bbeta^*}}{\EE|\xi|}=0.2$}
			\label{fig:CaCb0.2}
		\end{subfigure}
		\hfill
		\begin{subfigure}[b]{0.45\textwidth}
			\centering
			\includegraphics[width=\textwidth]{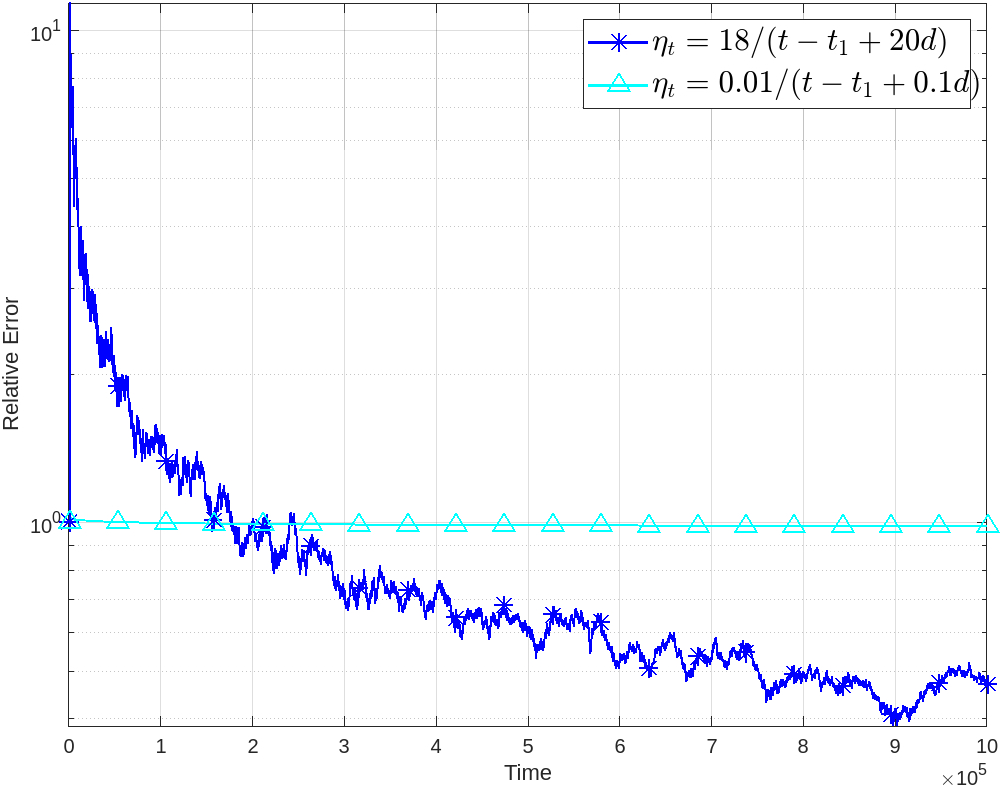}
			\caption{$\frac{\ltwo{\Bbeta^*}}{\EE|\xi|}=0.02$}
			\label{fig:CaCb0.02}
		\end{subfigure}
		\caption{{\it Trade-off between short-term accuracy and long-term optimality}. The initialization $\Bbeta_0=\boldsymbol{0}$ is already in proximity to  $\Bbeta^{\ast}$, enabling the online sub-gradient descent algorithm to bypass the first phase and start the second phase iterations immediately. A small stepsize can ensure short-term accuracy but may compromise losing long-term optimality, whereas a large stepsize can guarantee long-term optimality but at the expense of short-term accuracy. Here, the dimension is set to $d=100$, and the noise follows a Student $t_{\nu}$ distribution with $\nu=1.1$.}
		\label{fig:CaCb}
	\end{figure}

	\begin{theorem}
		Under the same conditions and stepsize scheme as in Theorem~\ref{thm:one_sample}, the online sub-gradient descent algorithm achieves an upper bound on regret given by
		\begin{align*}
			\textsf{Regret}_T\leq C_1(\ku/\kl)\sqrt{\ku} \, d\ltwo{\Bbeta_{0}-\Bbeta^*}+C_2\ca\frac{\ku^2}{\kl^2}\frac{b_0^2}{b_1}d\log\bigg(1+ \frac{T }{\cb d}\bigg),
		\end{align*}
		\label{thm:one_sample_regret}
		for any $T$ with probability over $1-t_1\exp(-cd)$.
	\end{theorem}
	
	The regret upper bound exhibits logarithmic growth with respect to the horizon $T$, aligning with the well-established optimal regret bounds in online optimization \citep{hazan2007logarithmic,orabona2019modern}. The other term in the regret upper bound is solely dependent on the initialization error and remains unaffected by the horizon $T$. This indicates that the initialization error does not exert a persistent influence over the regret. In contrast, regret upper bounds established by \cite{hazan2007logarithmic}, \cite{orabona2019modern}, \cite{cai2023online}, and \cite{han2022online} all involve a multiplicative term of the form $\ltwo{\Bbeta_{0}-\Bbeta^*}^2\log T$, indicating that poor initialization may lead to substantially larger regret.

	\subsection{Batch Learning}\label{sec:batch}

	Suppose the server receives a batch of data points $\calD_t=\{(\X_i^{(t)},Y_i^{(t)})\}_{i=1}^{n_t}$ at each time $t$, consisting of $n_t$ i.i.d. observations that follow the linear model (\ref{eq:model}). The  loss function at time $t$ is given by 
	\begin{align}
		f_t(\Bbeta)=\frac{1}{n_t}\sum_{i=1}^{n_t} \rho_{Q,\tau}(Y_i^{(t)}-\X_i^{(t)\top}\Bbeta),
		\label{eq:loss-batch}
	\end{align}
	which represents the empirical risk based on $n_t$ observations. The special case considered in Section~\ref{sec:onesample}, where only a single data point is observed, corresponds to $n_t=1$. As $n_t$ increases, the objective function (\ref{eq:loss-batch}) becomes more amenable to analysis due to the concentration phenomenon. As before, we allow for an arbitrary initialization $\Bbeta_0$. At each time $t$, the current estimate is $\Bbeta_t$, and upon receiving the new batch $\{(\X_i^{(t)},Y_i^{(t)})\}_{i=1}^{n_t}$, the update is given by
	$$
	\Bbeta_{t+1}=\Bbeta_{t}-\eta_t\cdot \frac{1}{n_t}\sum_{i=1}^{n_t}(\tau-\mathbb{I}\{Y_i^{(t)}-\langle\X_i^{(t)},\Bbeta_t\rangle\})\cdot \X_i^{(t)},
	$$ 
	where $\eta_t$ denotes the stepsize. The choice of $\eta_t$ will be discussed in Theorem~\ref{thm:Batch}. In contrast to our focus, \cite{jiang2022renewable} and \cite{wang2022renewable} examine asymptotic batch learning for quantile regression, specifically in the regime where the batch size $n_t$ tends to infinity. Separately, \cite{do2009proximal} investigates batch learning from an optimization-theoretic perspective. More recently, the advantages of batch learning have been extensively explored in the literature on bandit algorithms \citep{ren2023dynamic, gao2019batched, han2020sequential}. Despite these advances, the computational and statistical aspects of batched quantile regression remain largely underexplored. Relative to the setting considered in Section~\ref{sec:batch}, access to larger batches provides technical advantages in analyzing the non-smooth objective function, especially when the current iterate is far from the ground truth, and gives rise to distinct theoretical properties.

	The excess risk and sub-gradient of the objective function play central roles in characterizing the convergence dynamics of (sub-)gradient descent algorithms. While the expected excess risk properties illustrated in Figure~\ref{fig:regularity-one} also hold for the objective function \eqref{eq:loss-batch}, the sub-gradient of the objective function $f_t(\Bbeta)$, denoted by $\g_t\in \partial f_t(\Bbeta)$, behaves distinctly from that in Section~\ref{sec:onesample}. The three phases in the behavior of the expected sub-gradient norm are illustrated in Figure~\ref{fig:regularity-two}. Essentially, when $\Bbeta$ is sufficiently far from the ground truth, the expected squared norm of the sub-gradient, $\EE\|\g\|_2^2$, is bounded above by a constant that is independent of the distance $\|\Bbeta-\Bbeta^{\ast}\|_2$. When $\Bbeta$ is closer, but not too close, to $\Bbeta^{\ast}$, the expected norm is upper bounded by $O(\|\Bbeta-\Bbeta^{\ast}\|_2)$, and this bound remains independent of the batch size. Finally, when $\Bbeta$ is very close to the ground truth, the expected sub-gradient norm depends only on the batch size, dimensionality, and noise characteristics.

	\begin{figure}[htpb]
		\centering
		\includegraphics[width=0.8\textwidth]{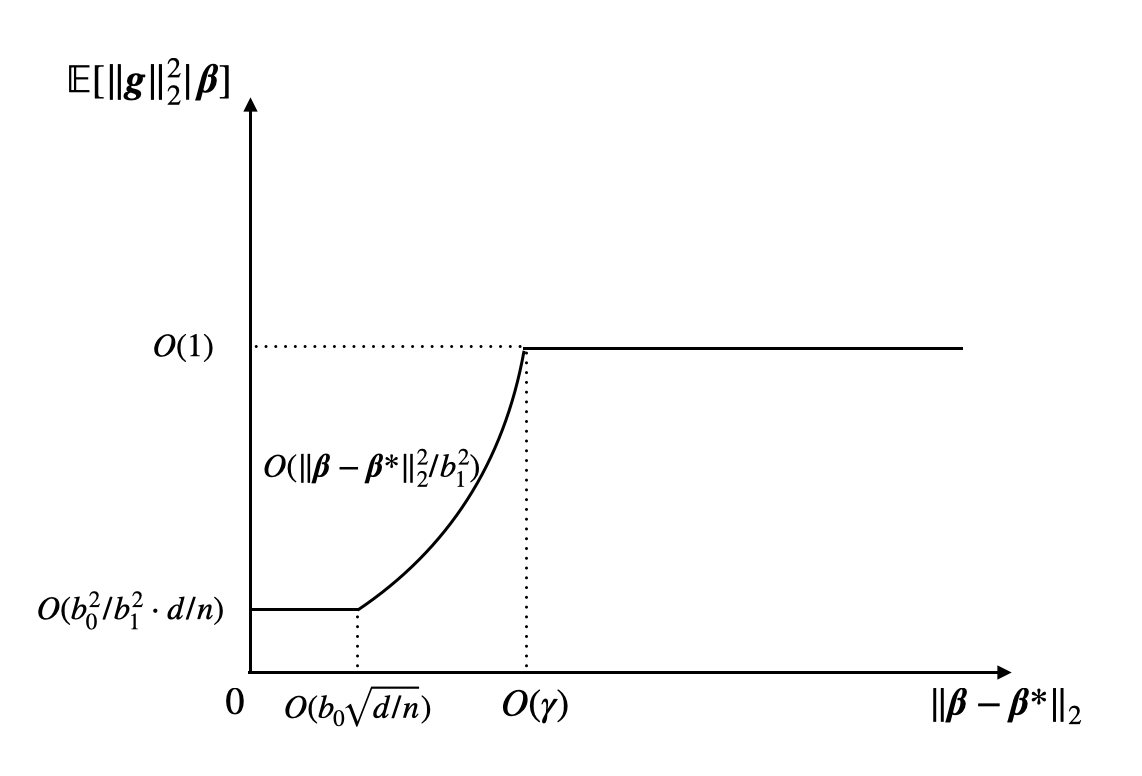}
		\caption{{ Expected length of the sub-gradient for the empirical quantile loss.} $Y$-axis: upper bound of $\EE[\ltwo{\g}^2|\Bbeta]$, where $\g\in\partial \frac{1}{n}\sum_{i=1}^{n} \rho_{Q,\tau}(Y_i-\X_i^{\top}\Bbeta)$; $X$-axis: $\|\Bbeta-\Bbeta^{\ast}\|_2$. It reveals three phases of properties associated with the sub-gradient, depending on the closeness between $\Bbeta$ and $\Bbeta^{\ast}$.}
		\label{fig:regularity-two}
	\end{figure}
	
	Hence, the online sub-gradient descent algorithm applied to \eqref{eq:loss-batch} exhibits a three-phase convergence behavior, provided that the stepsize is carefully chosen to align with the aforementioned statistical properties. For notational convenience, define $\bar \tau = \max\{ \tau, 1-\tau \}$.

	\begin{theorem}
		Suppose Assumptions~\ref{assump:sensing_operators:vec} and \ref{assump:heavy-tailed} hold. Let $\Bbeta_0$ be an arbitrary initialization satisfying $\ltwo{\Bbeta_0-\Bbeta^*}\leq D_0$ for some $D_0>0$. There exist universal constants $\{c_i\}_{i=0}^5,C_0,C_1,C_2>0$ such that if the batch size $n_t\geq n\geq C_0(\ku/\kl) \bar \tau^2 d$ for all $t$, then the sequence $\{\Bbeta_t\}_{t\geq 1}$ generated by the online sub-gradient descent algorithm exhibits the following dynamics:
		\begin{enumerate}[(1).]
			\item during phase one, where $\ltwo{\Bbeta_{t}-\Bbeta^*}\geq 8 \kl^{-1/2} \gamma$, by choosing a stepsize $\eta_0=C\sqrt{\kl}/\ku\cdot D_0$ and $\eta_t=(1-c\kl/\ku)\eta_{t-1}$ with any $D_0\geq\|\Bbeta_t-\Bbeta^*\|_2$, $C>1$, $c<0.1$, it holds with probability exceeding $1-\exp(-c_0d)-\exp\big(-\sqrt{n_t/\log(n_t)}\big)$ that
			$$
			\ltwo{\Bbeta_{t+1}-\Bbeta^*}\leq \bigg(1-c\frac{\kl}{\ku}\bigg)D_0;
			$$
			the conclusion of phase one is reached after $t_1=O(\log(\ltwo{\Bbeta_{0}-\Bbeta^*}/\gamma))$ iterations;

			\item in phase two, characterized by $C_1 C_u^{1/2} C_l^{-1} \bar \tau  \sqrt{ d/n}\cdot b_0\leq\ltwo{\Bbeta_{t}-\Bbeta^*}<8\kl^{-1/2}\gamma$, by selecting a constant stepsize $\eta_t=\eta\in(\kl/\ku^2)(b_1^2/b_0)\cdot[c_1,c_2]$, it holds with probability at least  $1-3\exp(-c_0d)-\exp(-\sqrt{n_t/\log n_t})$ that
			$$
			\ltwo{\Bbeta_{t+1}-\Bbeta^*}^2\leq\bigg(1-0.0005\frac{b_1^2}{b_0^2}\frac{\kl^2}{\ku^2}\bigg)\cdot\ltwo{\Bbeta_{t}-\Bbeta^*}^2;
			$$
			phase two requires $O\big(\log((n/d)(\gamma/b_0))\big)$ iterations to conclude, providing an estimate $\Bbeta_{t_2}$ that satisfies  $\ltwo{\Bbeta_{t_2}-\Bbeta^*}\asymp  C_u^{1/2} C_l^{-1}  \bar \tau \sqrt{ d/n}\cdot b_0$ under the specified conditions;
			
			\item in phase three, characterized by $\ltwo{\Bbeta_{t}-\Bbeta^*}\leq C_1 C_u^{1/2}C_l^{-1}  \bar \tau \sqrt{ d/n}\cdot b_0$, by choosing a stepsize $\eta_t:=  \frac{\ca}{\kl}\frac{b_0}{t+\cb - t_2}$ with arbitrary constants $\ca\geq 12$ and $\cb>\ca/12$, it holds with probability exceeding $1-c_2\exp(-c_1d)-\exp(-\sqrt{n_t/\log n_t})$ that
			$$\ltwo{\Bbeta_{t+1}-\Bbeta^*}^2\leq C^*\frac{\ku^2b_0^2}{\kl^2b_1^2} \frac{\bar \tau^2 }{t+\cb+1-t_2}\frac{d}{n}\cdot \frac{\ku}{\kl^2} b_0^2,$$
			where the constant $C^*$ depends solely on $\ca,\cb$, and $t_2$ denotes the conclusion time of the phase two convergence.
		\end{enumerate}
		\label{thm:Batch}
	\end{theorem}

	According to Theorem~\ref{thm:Batch}, the online sub-gradient descent algorithm exhibits fast linear convergence during phases one and two. The stepsize schemes used in these phases mirror those observed in the two-phase convergence of offline quantile regression \citep{shen2023computationally}. Specifically, the first phase requires $t_1$ iterations and uses a total of $O(d\log(\ltwo{\Bbeta_{0}-\Bbeta^*}/\gamma))$ data vectors. This aligns with the data scale required in phase one of the online learning setting, as outlined in Theorem~\ref{thm:one_sample}. In the third phase, a stepsize of $O(1/t)$ is essential for attaining long-term statistical optimality. The final error rate exhibits a continual decrease as additional data becomes available, demonstrating the efficacy of the proposed stepsize scheme in seamlessly integrating sequentially arriving data with an unknown horizon. Furthermore, at each iteration, the failure probability diminishes exponentially with respect to the dimension.

	\begin{corollary} \label{cor:batch}
		Under the conditions of Theorem~\ref{thm:Batch}, and if the horizon $T\geq C_0(\log(\ltwo{\Bbeta_{0}-\Bbeta^*}/b_0) + \log(n/d))$, then,   with a probability exceeding $1-c_1T\exp(-c_0d)-T\exp(-\sqrt{n/\log n})$, the online sub-gradient descent algorithm produces a final estimate with an error rate given by
		$$
		\ltwo{\Bbeta_T-\Bbeta^*}^2\leq C^{\ast}\frac{\ku^2b_0^2}{\kl^2b_1^2} \frac{\ku \bar \tau }{\kl^2} \cdot \frac{d}{nT} \cdot b_0^2,
		$$
		where the constant $C^{\ast}$ depends only on $C_a$ and $C_b$, and $\bar \tau = \max\{ \tau, 1-\tau\}$.
	\end{corollary}
	
	If all batch sizes satisfy $n_t\asymp n$, then the total number of data points used is $O(Tn)$. In this setting, the error rate established in Corollary~\ref{cor:batch} matches the minimax-optimal rate in the offline setting \citep{he2021smoothed}. Notably, this rate is robust to initialization errors, similar to the rate obtained for online learning in Corollary~\ref{cor:one_sample}. When the initialization is sufficiently close to $\Bbeta^{\ast}$, a trade-off emerges between short-term accuracy and long-term optimality, as discussed in Remark~\ref{rmk:one_sample_init}.

	\begin{theorem}
		Under the same conditions and stepsize scheme as in Theorem~\ref{thm:Batch}, the online sub-gradient descent algorithm attains a regret upper bound that for any $T\geq 1$ with probability over $1-c_1t_1\exp(-c_0d)-t_1\exp(-\sqrt{n/\log n})$,
		\begin{align*}
			\textsf{Regret}_T\leq C_1\frac{\ku^3}{\kl^3}\frac{b_0^2}{b_1^2}\max\{\sqrt{\ku}\ltwo{\Bbeta_{0}-\Bbeta^*},\gamma^2/b_0\}+C_2\frac{\ku^3}{\kl^3}\frac{b_0^3}{b_1^3}\frac{d}{n}b_0\log\left(\frac{T-t_1+\cb}{\cb} \right).
		\end{align*} 
		\label{thm:Batch regret}
	\end{theorem}

	The regret upper bound consists of two terms: one that depends on the initialization error but is independent of $T$, and another that grows logarithmically with $T$ yet remains independent of the initialization. When the total number of data points consumed is equal--denoted by $N$, and assuming $N\geq C_1\max\{ n\log(\ltwo{\Bbeta_{0}-\Bbeta^* }/b_0), n\log(n/d)\}$--both online and batch learning achieve an optimal estimation error rate of $O(b_0^2\,d/N)$. However, batch learning incurs significantly lower regret by leveraging a larger number of data points in each iteration, thereby accelerating the learning process. Specifically, as shown in Theorem~\ref{thm:one_sample_regret}, online learning achieves a regret upper bound of $O(d\ltwo{\Bbeta_{0}-\Bbeta^*}+d\log(N/d))$, whereas batch learning attains a considerably smaller bound of $O(\ltwo{\Bbeta_{0}-\Bbeta^*}+(d/n)\log(N/n))$. It is worth noting, however, that batch learning requires substantially greater storage capacity.

	\subsection{Sequential Learning with Infinite Storage}
	
	Consider a scenario in which the server hosting $\Bbeta_{t+1}$ has unlimited storage capacity, retaining all historical data and updating the estimate only upon the arrival of new samples. Let $\Bbeta_0$ be an arbitrary initialization. At each time $t$, the server incorporates all observations received up to and including time $t$ when computing the update. Let $\{(Y_i^{(t)},\X_i^{(t)})\}_{i=1}^{n_t}$ denote the $n_t$ observations that arrive at time $t$. The loss function at time $t$ is then given by
	\begin{align}\label{eq:infinite-loss}
		f_t(\Bbeta):=\frac{1}{\sum_{l=0}^t n_l}\sum_{l=0}^{t}\sum_{i=1}^{n_l}\rho_{Q,\tau}(Y_i^{(l)}-\X_i^{(l)\top}\Bbeta),
	\end{align}
	which corresponds to the empirical risk over all accumulated data. The online sub-gradient descent algorithm updates the iterate via $\Bbeta_{t+1}=\Bbeta_{t}-\eta_t\cdot \g_t$, where $\eta_t$ is the stepsize and $\g_t\in\partial f_t(\Bbeta_t)$ is a sub-gradient. Specifically, we use the update rule
	$$
	\Bbeta_{t+1}=\Bbeta_{t}-\eta_t\cdot \frac{1}{\sum_{l=0}^tn_l}\sum_{l=0}^{n_l}\sum_{i=1}^{n_t}(\tau-\mathbb{I}\{Y_i^{(l)}-\langle\X_i^{(l)},\Bbeta_t\rangle\})\cdot \X_i^{(l)} .
	$$ 
	The stepsize scheme $\eta_t$ will be discussed in Theorem~\ref{thm:infite storage}. Although this algorithm requires iterating over the entire dataset accumulated so far at each update, it remains computationally more efficient than a full offline optimization approach \citep{shen2023computationally}.

	The objective function in (\ref{eq:infinite-loss}) retains its inherent two-phase structure, as illustrated in Figure~\ref{fig:regularity-one}. Notably, although this formulation incurs substantial storage and computational costs, these demands facilitate a faster convergence rate of the algorithm, as established in the following theorem.

	\begin{theorem}
		Suppose Assumptions~\ref{assump:sensing_operators:vec} and ~\ref{assump:heavy-tailed} hold, and let $\Bbeta_0$ be an arbitrary initialization satisfying   $\ltwo{\Bbeta_0-\Bbeta^*}\leq D_0$ for some $D_0>0$. There exist absolute positive constants $c_0,c_1,c_2,C_1,C_2$ such that, if the initial storage $n_0\geq C_1d$, then the sequence $\{\Bbeta_t\}_{t\geq 1}$ generated by the online sub-gradient descent algorithm has the following dynamics:
		\begin{enumerate}[(1).]
			\item During phase one, when $\ltwo{\Bbeta_{t}-\Bbeta^*}\geq 8\gamma$, by selecting a stepsize $\eta_t:= \frac{\sqrt{\kl}}{8\ku} \left(1-\frac{1}{100}\frac{\kl}{\ku}\right)^{t}\cdot D_0$, it holds with probability at least $1-\exp(-c_0d)-\exp(-\sqrt{\sum_{l=0}^{t} n_l/\log \sum_{l=0}^{t} n_l})$ that
			$$
			\ltwo{\Bbeta_{t+1}-\Bbeta^*}\leq\bigg(1-\frac{1}{100}\frac{\kl}{\ku}\bigg)^{t+1}\cdot D_0 .
			$$
			Phase one concludes after  $t_1=O\big(\log(\|\Bbeta_0-\Bbeta^{\ast}\|_2/\gamma)\big)$ iterations. 
			
			\item In phase two, when $\ltwo{\Bbeta_{t}-\Bbeta^*}< 8\gamma$, by selecting a stepsize $\eta_t:=\eta\asymp \frac{\kl}{\ku^2}\cdot\frac{b_1^2}{b_0}$, it holds with probability at least $1-\exp(-c_0d)-\exp(-\sqrt{\sum_{l=0}^{t} n_l/\log \sum_{l=0}^{t} n_l})$ that
			\begin{align*}
				\ltwo{\Bbeta_{t+1}-\Bbeta^*}^2\leq \bigg(1-c_1\frac{b_1^2}{b_0^2}\cdot\frac{\kl^2}{\ku^2}\bigg)\cdot\ltwo{\Bbeta_{t}-\Bbeta^*}^2 + c_2\frac{d}{\ku\sum_{l=0}^t n_l}\cdot b_1^2.
			\end{align*}
		\end{enumerate}
		\label{thm:infite storage}
	\end{theorem}
	
	As shown in Theorem~\ref{thm:infite storage}, the convergence behavior of the algorithm proceeds in two distinct phases. The initial phase requires a geometrically decaying sequence of stepsizes, while the second phase adopts a constant stepsize.  Both phases are proven to exhibit linear convergence. For simplicity, assume that $n_l=n$ holds for all $l\geq 1$. During the second phase (i.e., for $t > t_1$), the error rate satisfies
	\begin{align*}
		\ltwo{\Bbeta_{t+1}-\Bbeta^*}^2\leq \bigg(1-c_1\frac{b_1^2}{b_0^2}\cdot\frac{\kl^2}{\ku^2}\bigg)^{t+1-t_1}\cdot\ltwo{\Bbeta_{t_1}-\Bbeta^*}^2 +C_2 \frac{d}{n_0+tn}\cdot\frac{\ku}{\kl^2}\cdot b_0^2\cdot\bigg(\frac{b_0}{b_1}\cdot\frac{\ku}{\kl}\bigg)^2,
	\end{align*}
	which leads to the following corollary.

	\begin{corollary}
		Assume the same conditions as in Theorem~\ref{thm:infite storage}, and suppose that $n_l=n$ for all $l\geq 1$. If the total time horizon satisfies $T\geq C^*\ltwo{\Bbeta_{0}-\Bbeta^*}/\gamma$, for some constant $C^*$ that may depend on $(\kl,\ku,n_0, n)$, with probability over $1-\sum_{t=0}^{T}\exp(-\sqrt{ (n_0+tn)/\log  (n_0+tn)}) \\-T\exp(-c_0d)$, it holds that
		$$
		\ltwo{\Bbeta_{T}-\Bbeta^*}\leq C\left(\frac{b_0}{b_1}\cdot\frac{\ku}{\kl}\right)^2\cdot \frac{\ku}{\kl^2}\cdot \frac{d}{n_0+Tn} b_0^2.
		$$ 
	\end{corollary}
	
	The corollary above demonstrates that, despite the server updating only once upon the arrival of each new batch of samples, the procedure consistently achieves statistical optimality, even when the total time horizon is unknown.

	\section{Numerical Experiments}
	
	In this section, we present experiments to empirically validate our theoretical findings. We first generate synthetic data for simulation studies, followed by an application of our proposed methods to a real-data example.
	
	\subsection{Simulations}
	
	\label{sec:simulation}
	
	In this section, we numerically examine the performance of the proposed online QR algorithms. We first demonstrate the effectiveness of our stepsize scheme, and then compare the final accuracy of different online estimators with that of the offline estimate. We further illustrate the advantage of online quantile regression over online least squares regression. It is worth noting that existing online quantile algorithms \citep{jiang2022renewable,sun2023online,chen2024renewable} depend on storage of $O(d^2)$ summary statistics. Additionally, they require data to arrive in batches with a diverging batch size. Their requirements limit their applicability for empirical performance comparisons. We use the relative error $\ltwo{\Bbeta_t-\Bbeta^*}/\ltwo{\Bbeta^*}$ as the main metric. For simplicity, the initialization $\Bbeta_0$ is set to be $\boldsymbol{0}$ throughout our numerical studies. 
	
	\medskip
	\noindent
	{\sc Stepsize Scheme}. We begin by demonstrating the effectiveness of our stepsize scheme, as outlined in Theorem~\ref{thm:one_sample}, which is theoretically guided by statistical regularities. The selection of parameters in the proposed stepsize is quite flexible, both theoretically (as outlined in Theorem~\ref{thm:one_sample}) and in practice. For instance, concerning the parameter $C_a$ required by the stepsize schedule in the second phase of iterations, we observe that our algorithm always performs well as long as it is not excessively small.

	In all the experiments, the stepsize for the first phase is scheduled as $\eta_t=(1-0.5/d)^t\eta_0$, where $\eta_0$ is the initial stepsize. The choice of stepsize in second phase involves specific parameters where we set $\ca = 20$ and $\cb=30$. We fix the dimension at $d=100$, the unknown horizon  $T=10^5$, and sample $t_{\nu}$-distributed noise with a degree of freedom $\nu=1.1$. The performance of proposed stepsize scheme is compared with those of two existing alternative stepsize scheme: the $\eta_t=O(1/t)$ decaying scheme and a constant stepsize scheme $\eta_t\equiv \emph{const}$. For detailed discussions on these stepsize schemes, we refer to \cite{duchi2009efficient} and \cite{zhang2004solving}.

	The convergence performances of online sub-gradient descent under the aforementioned three stepsize schemes are presented in Figure~\ref{fig:stepsize}. For both the moderate and strong SNR cases, our proposed stepsize scheme can ensure a linear convergence of the online sub-gradient descent algorithm in the first phase. The linear convergence behavior stops once the algorithm reaches at a sufficiently accurate estimate, i.e., when the error rate is dominated by the noise scale $\EE |\xi|$. Our proposed  scheme then resets the stepsize and the algorithm enters second phase of iterations. The error rate continues to decrease in the second phase, which is sub-linear and exhibits an $O(1/t)$ convergence rate. The two-phase convergence phenomenon is consistent with the theoretical discoveries. In contrast, the constant stepsize scheme can also achieve an error rate the same as the one our algorithm achieves at the end of first phase iterations. However, it cannot further improve the estimate or may converge too slowly resulting into a statistically sub-optimal estimate. It is also observed that employing a relatively large constant stepsize can facilitate faster convergence in the initial stage, which was claimed by the theoretical results in \cite{cai2023online}. The performance under the $O(1/t)$ stepsize scheme is considerably inferior to those under the other two stepsize schemes.

	\begin{figure}
		\centering
		\begin{subfigure}[b]{0.45\textwidth}
			\centering
			\includegraphics[width=\textwidth]{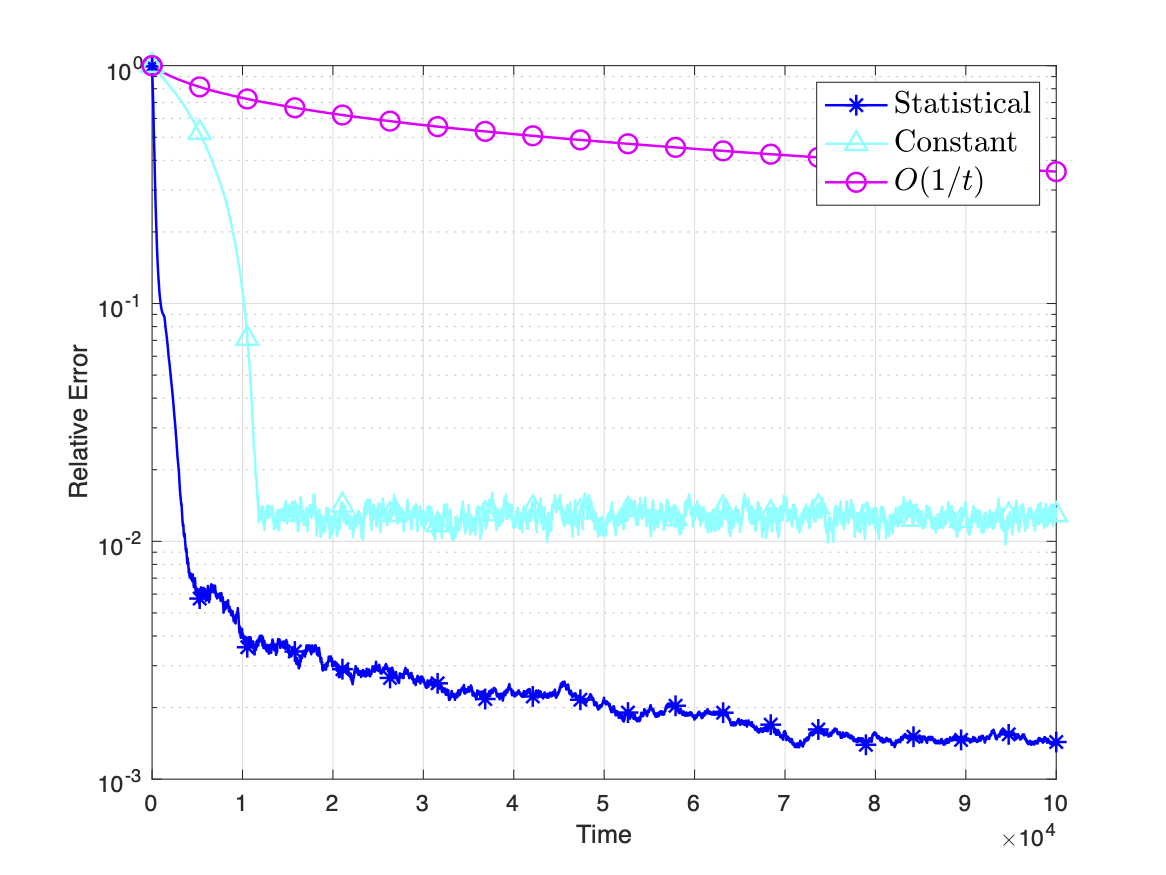}
			\caption{$\frac{\ltwo{\Bbeta^*}}{\EE|\xi|}=20$}
			\label{fig:step-SNR20}
		\end{subfigure}
		\hfill
		\begin{subfigure}[b]{0.45\textwidth}
			\centering
			\includegraphics[width=\textwidth]{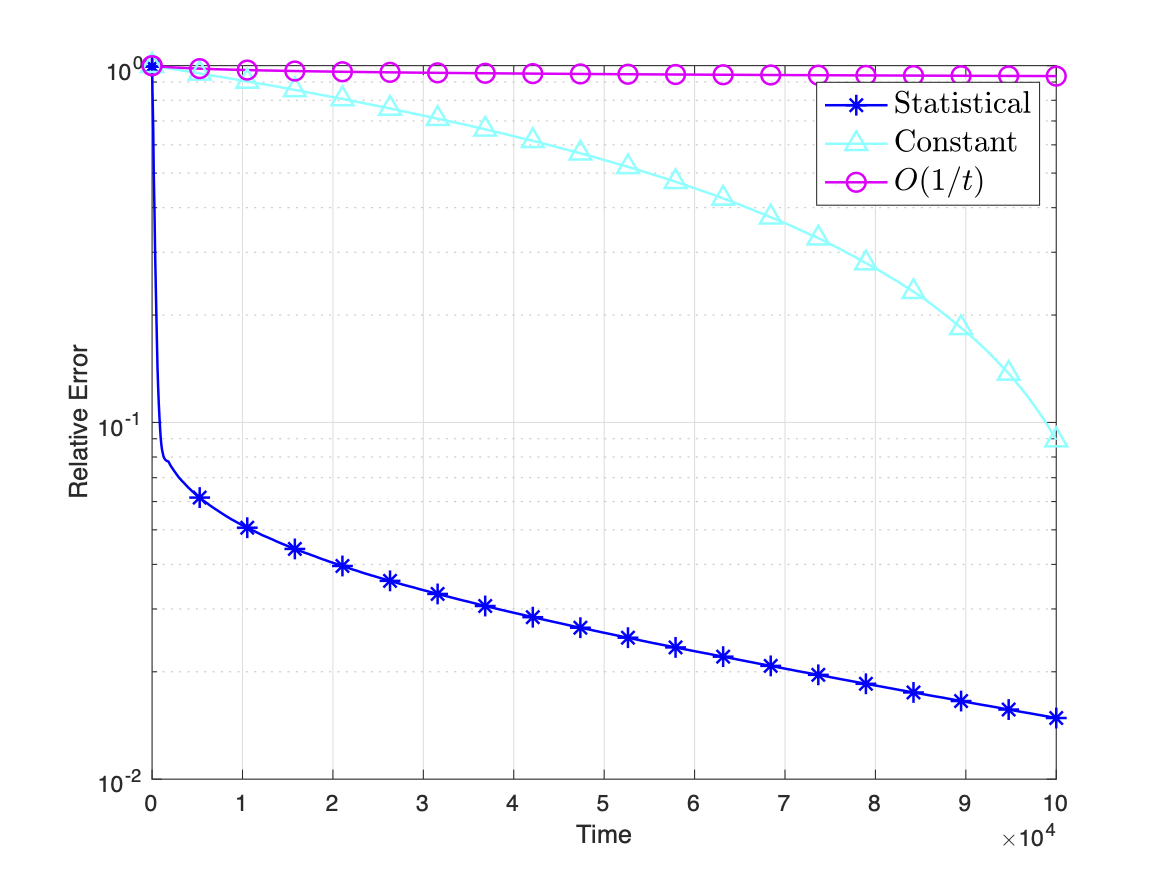}
			\caption{$\frac{\ltwo{\Bbeta^*}}{\EE|\xi|}=200$}
			\label{fig:step-SNR200}
		\end{subfigure}
		\caption{Relative error \emph{versus} time/iterations in online (one-sample) learning ($n_t\equiv 1$). The dimension $d=100$, unknown horizon $T=10^5$, and quantile loss parameter $\tau=1/2$. The convergence performances of online sub-gradient descent are examined under three stepsize schemes: \emph{Statistical} stands for our stepsize scheme guided by Theorem~\ref{thm:one_sample}, \emph{Constant} stands for the stepsize scheme $\eta_t\equiv \emph{const}$ \citep{zhang2004solving}, $O(1/t)$ means the decaying stepsize scheme $\eta_t=O(1/t)$ \citep{duchi2009efficient}. Left(a): moderate SNR; right(b): strong SNR. 
		}
		\label{fig:stepsize}
	\end{figure}
	
	\medskip
	\noindent 
	{\sc Statistical Accuracy Comparisons}. We now evaluate the statistical accuracy of the final estimator output by the online sub-gradient descent algorithm. The error rate achieved by offline quantile regression is used as the benchmark. The dimension $d$ and noise distribution are set the same as those in the previous set of simulations. Both the online one-sample learning and batch learning are studied in the experiment. The total sample size is set at $n=20,000$. Figure~\ref{fig:accSample20000} displays a box plot of error rates based on 50 replications. For each simulation, the online batch learning algorithm and one-sample learning algorithm take approximately 2 and 10 seconds, respectively, on a MacBook Pro 2020. The offline learning, implemented using the  \textsf{quantreg} package \citep{portnoy1997gaussian}, takes more than $2$ minutes. Figure~\ref{fig:accSample20000} shows that offline learning achieves the best statistical accuracy when the total sample size is relatively small, while online one-sample learning and batch learning achieve comparable statistical accuracy.  
	These are consistent with our theoretical findings. Simulation results for the large sample size $n=50,000$ are displayed in Figure~\ref{fig:accSample50000}, in which case the offline learning achieves only slightly better accuracy than its online counterparts. However, online algorithms enjoy much higher computational efficiencies. On the same Mac Pro, the online one-sample, batch, and offline learning methods take approximately 30, 5, and 520 seconds, respectively.

	\begin{figure}
		\centering
		\begin{subfigure}[b]{0.45\textwidth}
			\centering
			\includegraphics[width=\textwidth]{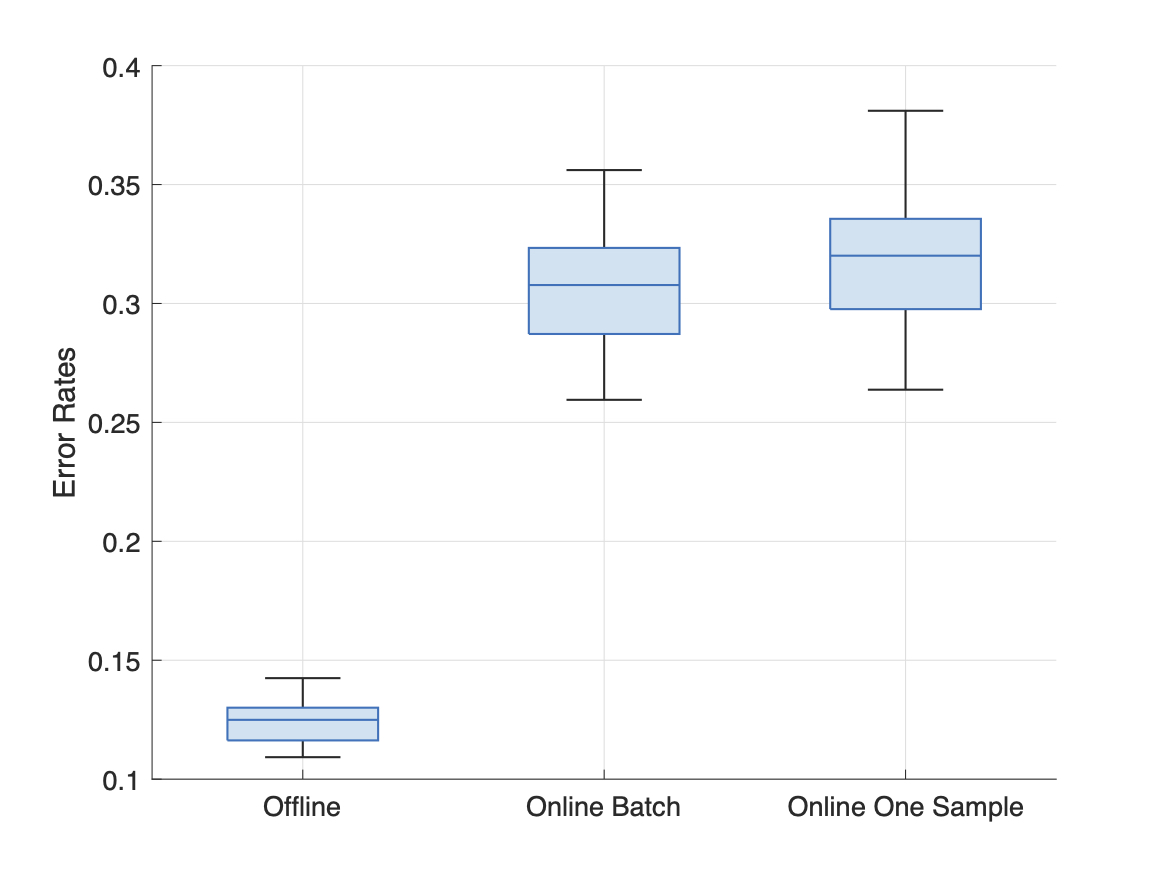}
			\caption{$\tau=0.25$}
			\label{fig:acc_sample20000tau025}
		\end{subfigure}
		\hfill
		\begin{subfigure}[b]{0.45\textwidth}
			\centering
			\includegraphics[width=\textwidth]{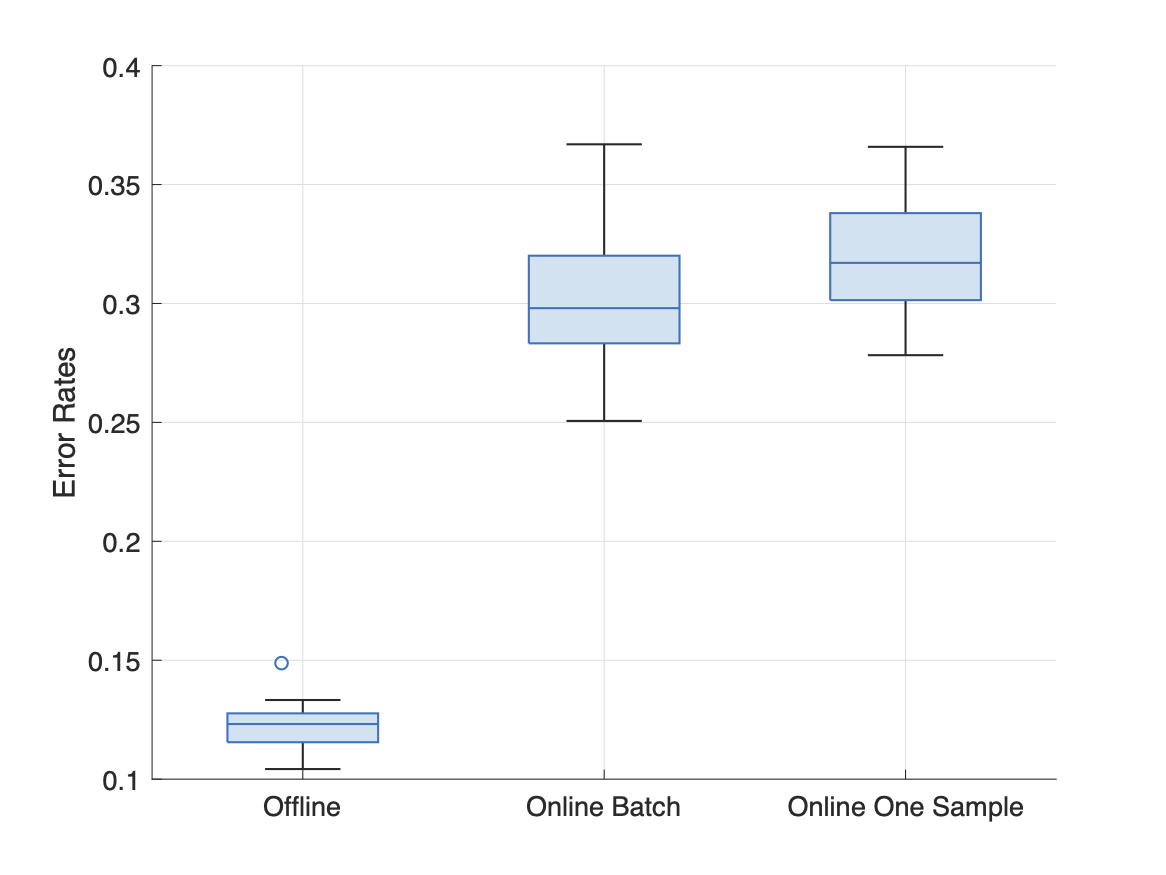}
			\caption{$\tau=0.75$}
			\label{fig:acc_sample20000tau075}
		\end{subfigure}
		\caption{Error rates of offline and online regression using quantile loss $\rho_{Q,\tau}(\cdot)$. Online One Sample refers to the online learning algorithm studied in Section~\ref{sec:onesample}. The dimension $d=100$, total sample size $n=20,000$, the batch size $n_t\equiv 100$, and noise has a $t_{1.1}$ distribution. Box-plots are drawn based on 50 independent simulations.}
		\label{fig:accSample20000}
	\end{figure}
	
	\begin{figure}
		\centering
		\begin{subfigure}[b]{0.45\textwidth}
			\centering
			\includegraphics[width=\textwidth]{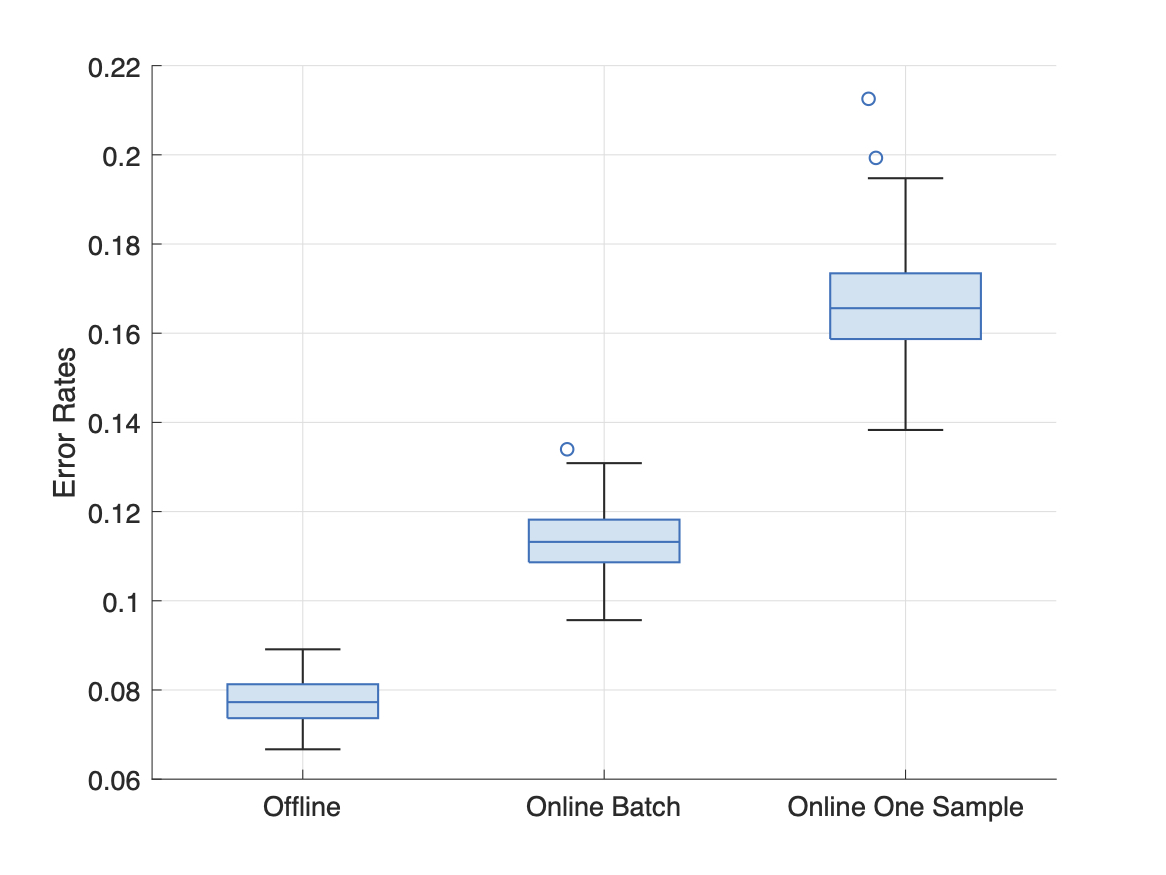}
			\caption{$\tau=0.25$}
			\label{fig:acc_sample50000tau025}
		\end{subfigure}
		\hfill
		\begin{subfigure}[b]{0.45\textwidth}
			\centering
			\includegraphics[width=\textwidth]{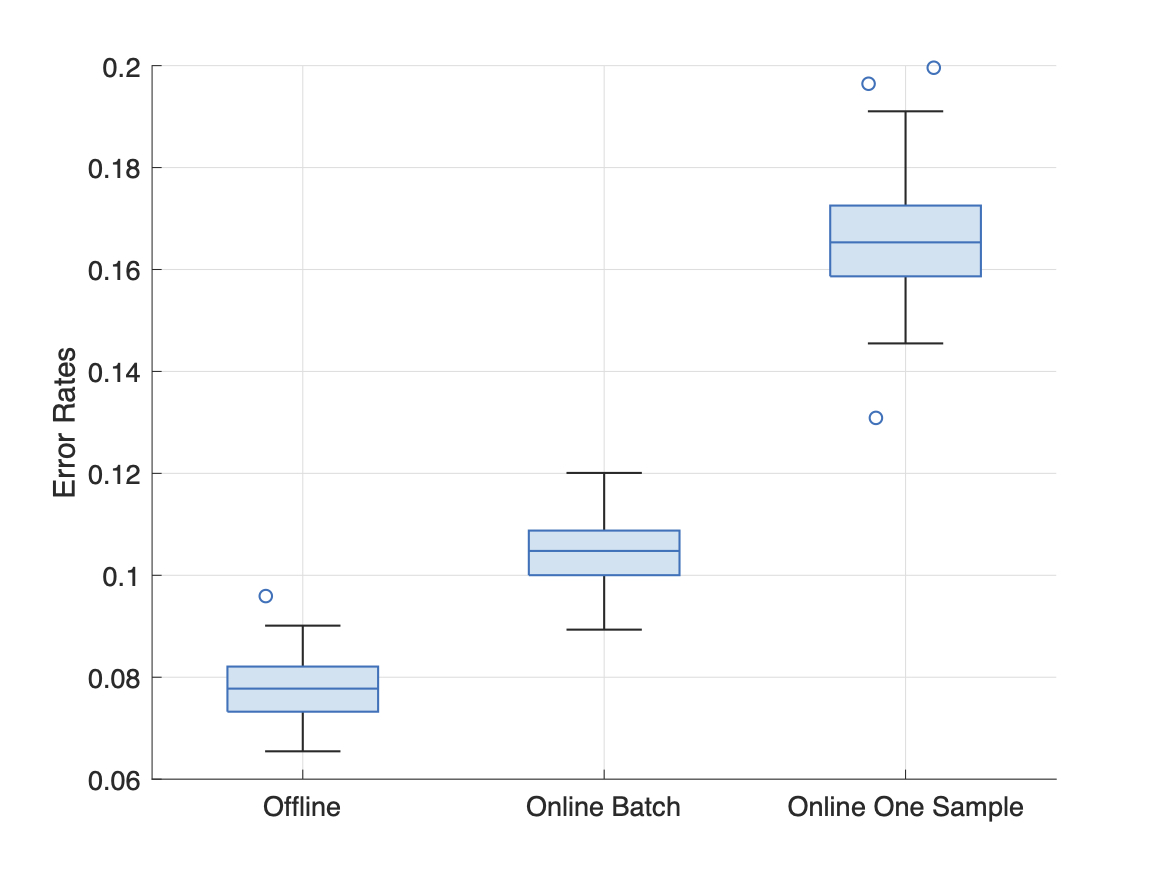}
			\caption{$\tau=0.75$}
			\label{fig:acc_sample50000tau075}
		\end{subfigure}
		\caption{Error rates of offline and online regression using quantile loss $\rho_{Q,\tau}(\cdot)$. Online One Sample refers to the online learning algorithm studied in Section~\ref{sec:onesample}. The dimension $d=100$, total sample size $n=50,000$, the batch size $n_t\equiv 200$, and noise has a $t_{1.1}$ distribution. Box-plots are drawn based on 50 independent simulations.}
		\label{fig:accSample50000}
	\end{figure}
	
	\medskip
	\noindent
	{\sc Convergence Dynamics Comparisons}. While online QR is motivated primarily for treating heavy-tailed noise, it is still of interest to examine its performance under Gaussian noise. Towards that end, we compare the performance of our proposed online QR algorithms with that of classical online least squares algorithm \citep{zhang2004solving}. We will show that the proposed online QR algorithms are not only robust in the presence of heavy-tailed noise or responses, but also are as efficient as the classical online least squares algorithm if the noise is Gaussian.

	While online learning with square loss has been extensively studied \citep{orabona2019modern,hazan2016introduction}, its stepsize scheme guided by statistical properties remains relatively under-explored. Here we briefly explain the appropriate stepsize scheme to achieve long-term statistical optimality for online least squares algorithm with a focus on sub-Gaussian noise and undetermined horizon.  Considering the square loss function at time $t$ as $f_t(\Bbeta):=(Y_t-\X_t^{\top}\Bbeta)^2$, the expected length of gradient $\g_t$ satisfies
	\begin{align*}
		\EE\left[\|\g_t\|_2^2\big|\Bbeta_t\right]\leq 4d\ku(\EE\xi^2+\ku\ltwo{\Bbeta_t-\Bbeta^*}^2).
	\end{align*}
	The gradient length is decided primarily by $\ltwo{\Bbeta_t-\Bbeta^* }^2$ when it is large. Conversely, when $\Bbeta_t$ is sufficiently close to the oracle, the gradient length is determined by $\EE\xi^2$. This implies an appropriate stepsize scheme for online least squares algorithm should also consists of two phases: a constant stepsize $\eta_t\equiv \eta=O(d^{-1})$ in the first phase; a decaying stepsize schedule $\eta_t=O\big((t+d)^{-1}\big)$ in the second phase. The detailed proof of the convergence dynamics under this stepsize scheme is almost identical to that of Theorem~\ref{thm:one_sample} and thus omitted.
	
	We compare the convergence dynamics of online one-sample QR algorithm and online least squares algorithm (equipped with the aforementioned two-phase stepsize scheme). The dimension $d=100$, quantile loss parameter $\tau=0.5$, and horizon is unknown. The results under Gaussian noise and $t_{1.1}$ noise are presented in Figure~\ref{fig:dynamics}. The convergence dynamics of the two algorithms are comparable under Gaussian noise, both showing a linear convergence in the first phase and an $O(1/t)$ decaying rate afterwards. They achieve almost the same statistical accuracy in the end. Under $t_{1.1}$-distributed noise, online least squares algorithm does not converge. In contrast, the proposed online QR algorithm ensures stable convergence and achieves error rates comparable to those under Gaussian noise.

	\begin{figure}
		\centering
		\begin{subfigure}[b]{0.45\textwidth}
			\centering
			\includegraphics[width=\textwidth]{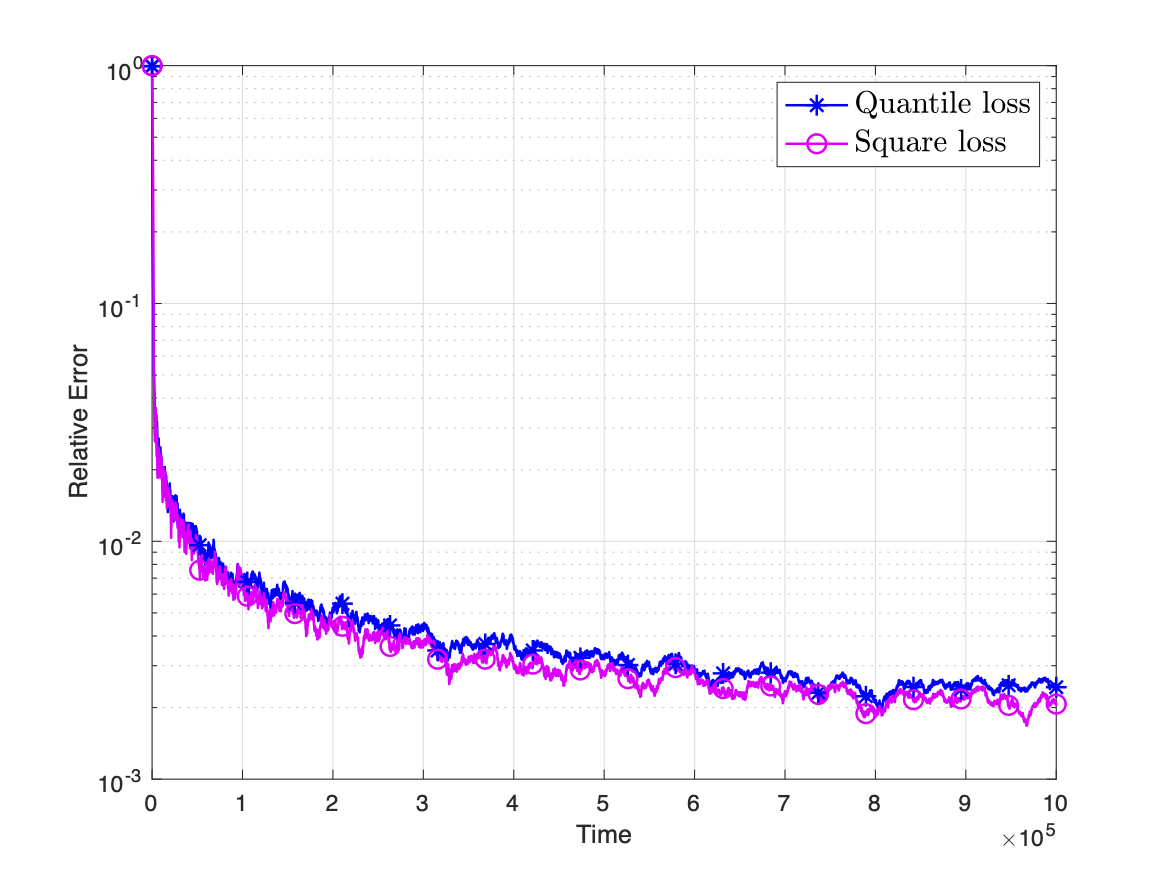}
			\caption{$\frac{\ltwo{\Bbeta^*}}{\EE|\xi|}=20$, Gaussian noise}
			\label{fig:dyn-GaussianSNR20}
		\end{subfigure}
		\hfill
		\begin{subfigure}[b]{0.45\textwidth}
			\centering
			\includegraphics[width=\textwidth]{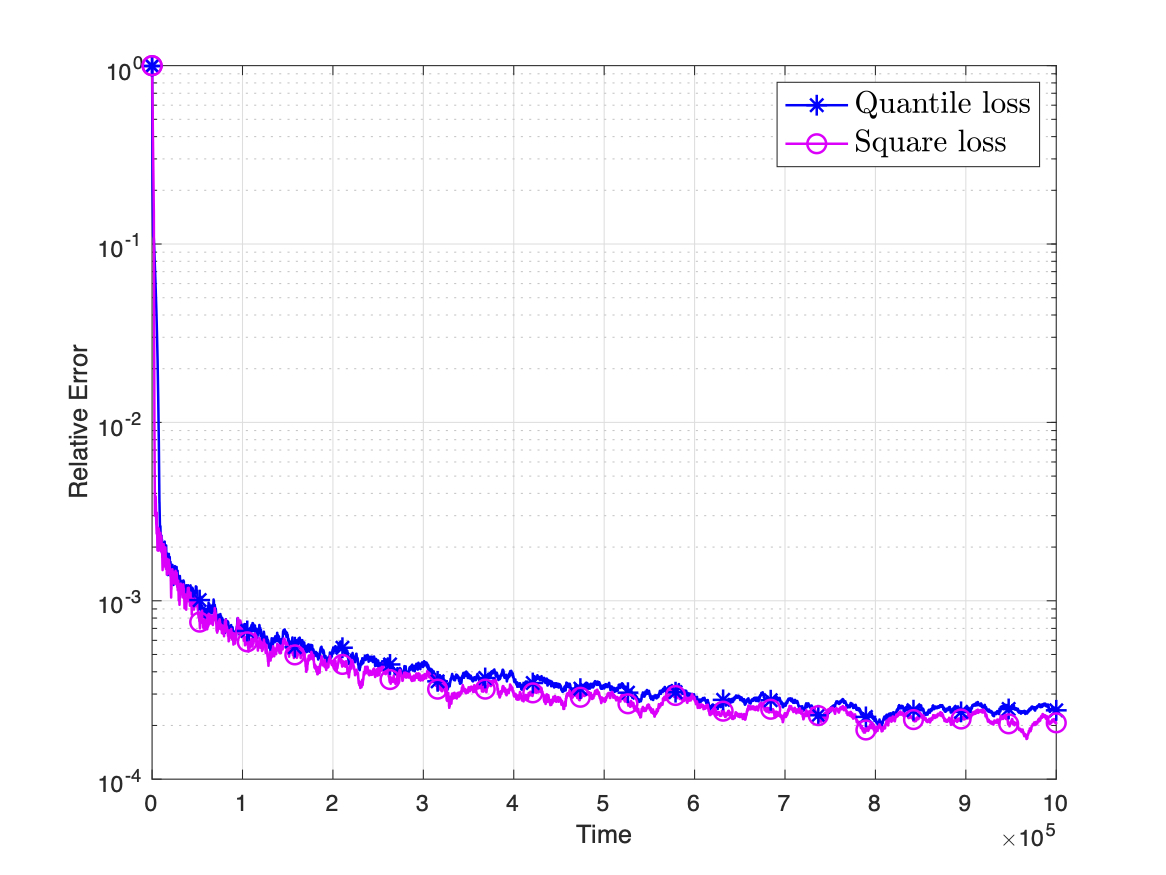}
			\caption{$\frac{\ltwo{\Bbeta^*}}{\EE|\xi|}=200$, Gaussian noise}
			\label{fig:dyn-GaussianSNR200}
		\end{subfigure}
		
		\begin{subfigure}[b]{0.45\textwidth}
			\centering
			\includegraphics[width=\textwidth]{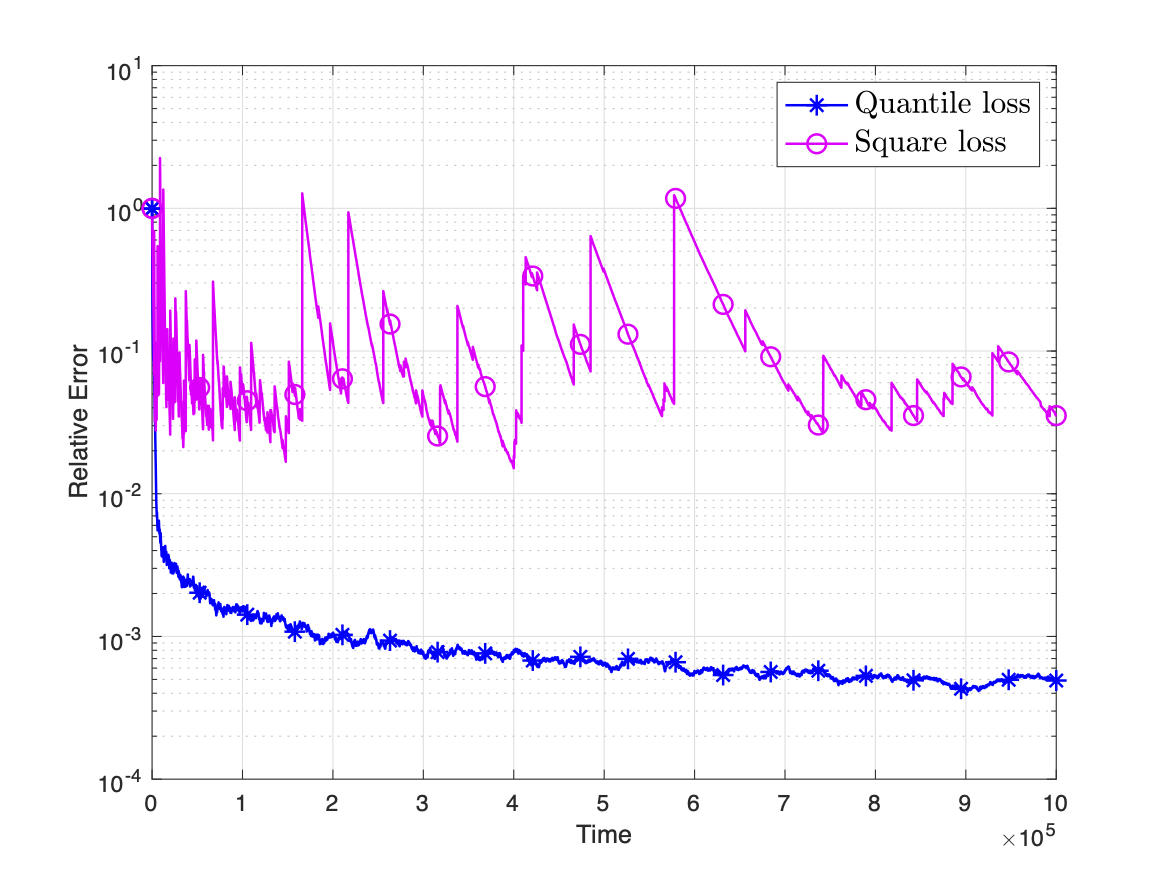}
			\caption{$\frac{\ltwo{\Bbeta^*}}{\EE|\xi|}=20$, $t_{1.1}$ noise}
			\label{fig:dyn-StudentSNR20}
		\end{subfigure}
		\hfill
		\begin{subfigure}[b]{0.45\textwidth}
			\centering
			\includegraphics[width=\textwidth]{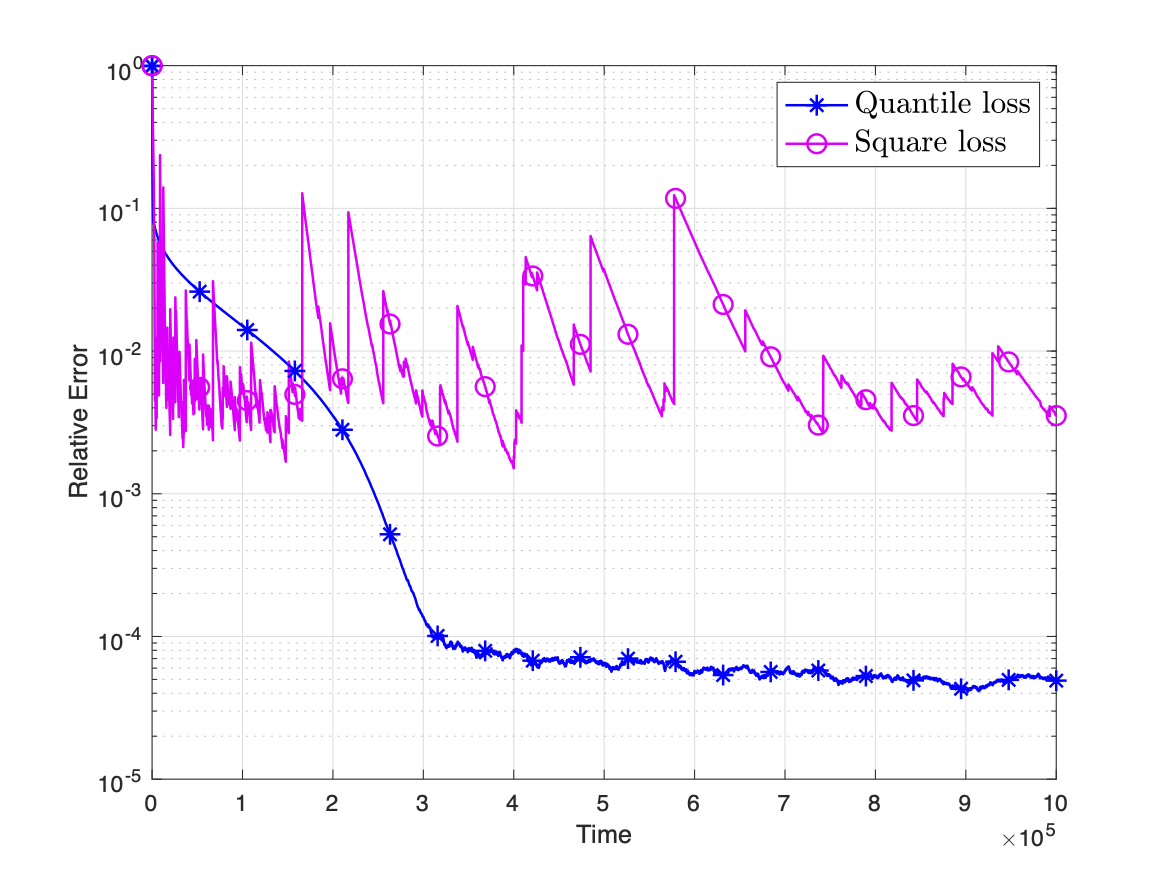}
			\caption{$\frac{\ltwo{\Bbeta^*}}{\EE|\xi|}=200$,  $t_{1.1}$ noise}
			\label{fig:dyn-StudentSNR200}
		\end{subfigure}
		
		\caption{Relative error \emph{versus} time/iterations for online QR and least squares algorithms. The dimension is $d=100$, the (unknown) time horizon is $T=10^5$, and the quantile level is $\tau=1/2$. The proposed online QR algorithm is robust to heavy-tailed noise and achieves comparable performance to online least squares under Gaussian noise.}
		\label{fig:dynamics}
	\end{figure}

	\medskip
	\noindent
	{\sc Parameter Sensitivity}. Here we empirically examine the flexibility of tuning parameters discussed in Theorem~\ref{thm:one_sample}. The simulation results demonstrating parameter sensitivity are presented in Figure~\ref{fig:parameter sensitivity}. Following the same settings as before, we set $d=100$, $\kl=\ku=1$ and $\| \Bbeta^* \|_2/\gamma = 20$. Figure~\ref{fig:parameter1} illustrates the impact of varying the geometric decay rate $c \in \{ 0.5, \ldots, 0.05 \}$ for the first-phase stepsize, defined as $\eta_t = (1 - c/d)^t \eta_0$, while the second-phase stepsize is fixed as $\eta_t = 15 / (t - t_1 + 20d)$. Interestingly, when $t$ is sufficiently large, different choices of the decay rate $c$ result in similar convergence behavior. This suggests that the decay rate $c$ in the first-phase stepsize is indeed flexible in practice. Figures~\ref{fig:parameter2}--\ref{fig:parameter4} further explore convergence dynamics under various values of $\ca$ and $\cb$. Specifically, Figure~\ref{fig:parameter3} shows the sensitivity to $\cb$ with $\ca = 15$ fixed, while Figure~\ref{fig:parameter4} examines the influence of $\ca$ with $\cb = 100$ held constant. These results indicate that as long as $\ca$ and $\cb$ are not too small, the estimation error remains on the same scale, and variations in $\cb$ have little effect on performance. However, if both parameters are too small (e.g., $\ca = 1$ and $\cb = 0$), the theoretical $O(1/t)$ convergence rate is no longer guaranteed. Furthermore, Figures~\ref{fig:parameter2}--\ref{fig:parameter4} depict the convergence dynamics under varying values of $\ca$ and $\cb$. Specifically, Figure~\ref{fig:parameter3} illustrates the effect of different choices of $\cb$ while fixing $\ca = 15$, and Figure~\ref{fig:parameter4} examines the sensitivity to $\ca$ with $\cb = 100$ held constant. These results suggest that as long as $\ca$ and $\cb$ are not too small, the estimation error remains on the same scale, and the choice of $\cb$ has minimal impact on performance. In contrast, when both parameters are set to overly small values, for example, $\ca = 1$ and $\cb = 0$, the convergence guarantee at the $O(1/t)$ rate no longer holds. These empirical findings support the theoretical flexibility of $\ca$ and $\cb$ as established in Theorem~\ref{thm:one_sample}.

	\begin{figure}
		\centering
		\begin{subfigure}[b]{0.45\textwidth}
			\centering
			\includegraphics[width=\textwidth]{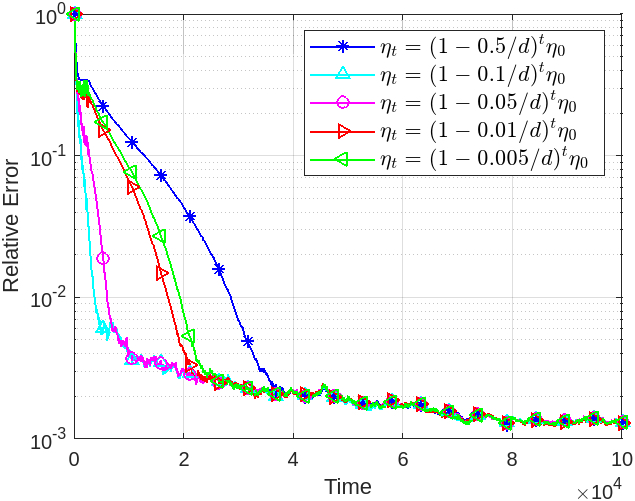}
			\caption{Convergence under varying geometrically decaying step sizes for $t\leq t_1$, followed by a common fixed stepsize for $t\geq t_1$.}
			\label{fig:parameter1}
		\end{subfigure}
		\hfill
		\begin{subfigure}[b]{0.45\textwidth}
			\centering
			\includegraphics[width=\textwidth]{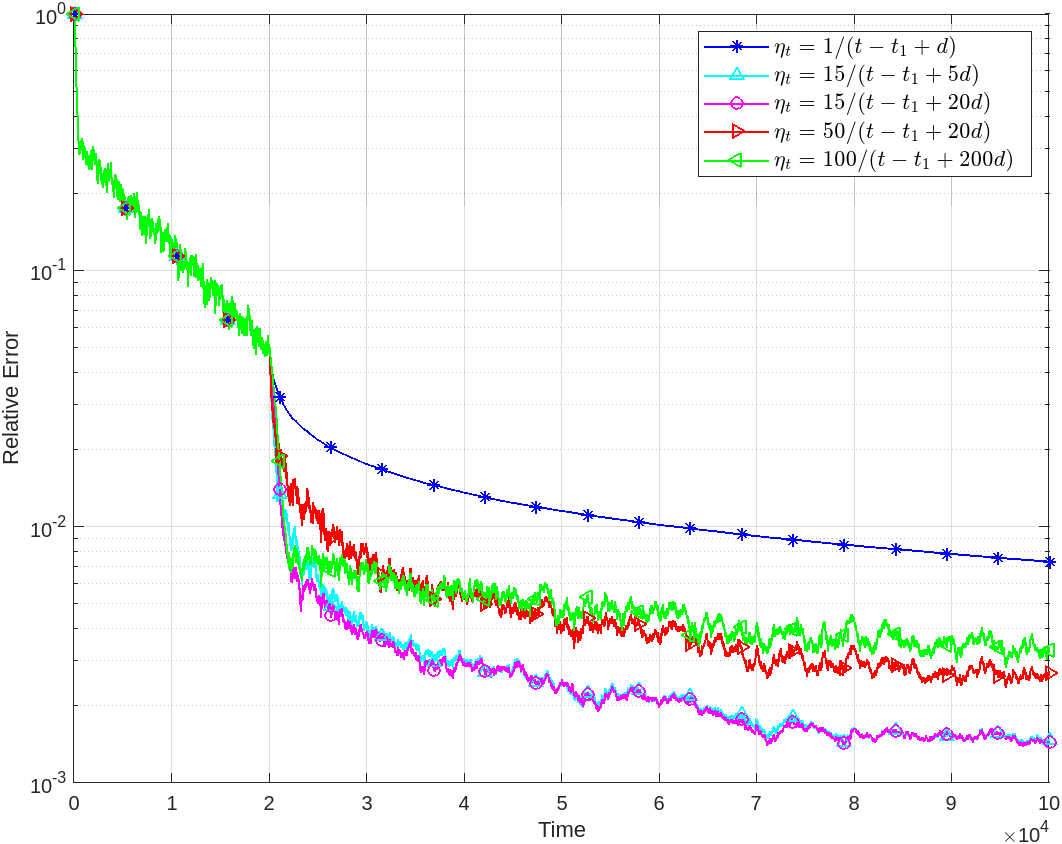}
			\caption{Convergence under varying values of $\ca$ and $\cb$ for $t\geq t_1$, with a common stepsize applied for $t\leq t_1$.}
			\label{fig:parameter2}
		\end{subfigure}
		
		\begin{subfigure}[b]{0.45\textwidth}
			\centering
			\includegraphics[width=\textwidth]{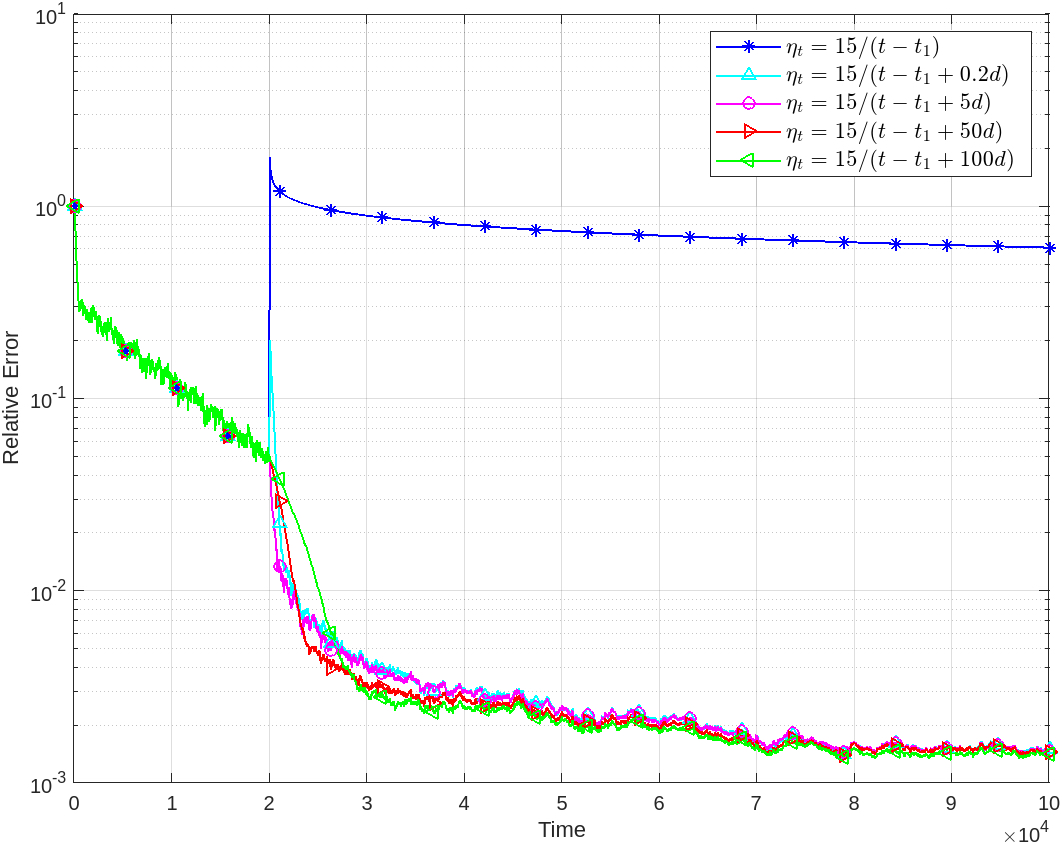}
			\caption{Convergence under a fixed $\ca$ for $t\geq t_1$, and a shared stepsize scheme for $t\leq t_1$.}
			\label{fig:parameter3}
		\end{subfigure}
		\hfill
		\begin{subfigure}[b]{0.45\textwidth}
			\centering
			\includegraphics[width=\textwidth]{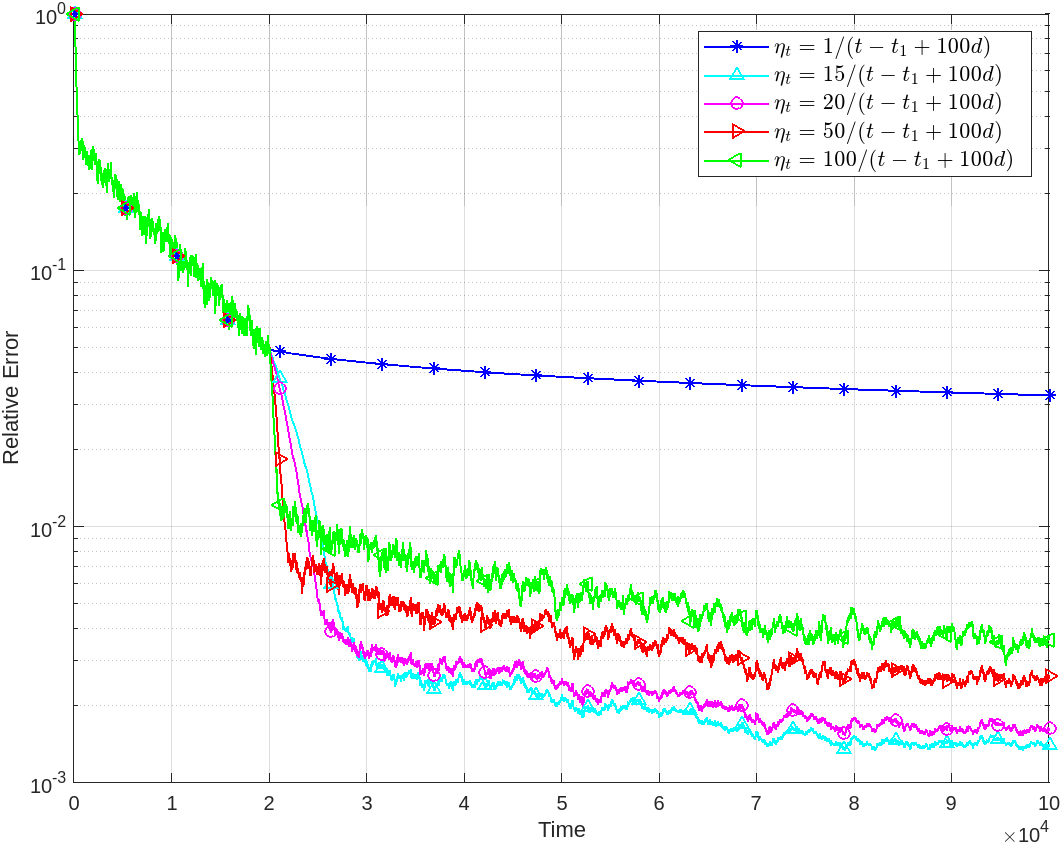}
			\caption{Convergence under a fixed value of $\cb$ for $t\geq t_1$, and a shared stepsize scheme for $t\leq t_1$.}
			\label{fig:parameter4}
		\end{subfigure}
		\caption{Relative error \emph{versus} time/iteration for online QR. The dimension is $d=100$, with a signal-to-noise ratio of $\frac{\|\Bbeta^*\|}{\EE|\xi|}=20$. The noise is generated from a $t$-distribution with 1.1 degrees of freedom. The time horizon is set to $T=10^5$.}
		\label{fig:parameter sensitivity}
	\end{figure}

	\medskip
	\noindent
	{\sc Regret Dynamics}. In this paragraph, we examine the regret dynamics illustrated in Figure~\ref{fig:regret}, under the setting where $d=200$, $T=10^5$, and the noise follows a $t$-distribution with 1.1 degrees of freedom ($t_{1.1}$). The left panel (Figure~\ref{fig:regret1}) corresponds to the stepsize schemes $\eta_t=(1-0.05/d)^t\eta_0$ and $\ca/(t-t_1+\cb d)$. The pink star-dotted line represents a benchmark function of the form $b + \log(1+t/a)$. As shown in Figure~\ref{fig:regret1}, the curves exhibit highly parallel growth patterns. Figure~\ref{fig:regret2} displays the regret values over time $t$ under different stepsize schemes, all showing a logarithmic trend of the form $b + \log(1+t/a)$, albeit with different parameterizations.

	\begin{figure}
		\centering
		\begin{subfigure}[b]{0.45\textwidth}
			\centering
			\includegraphics[width=\textwidth]{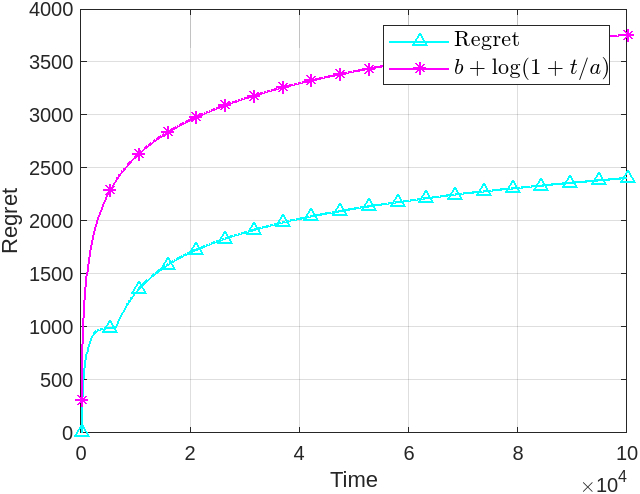}
			\caption{Regret Dynamics with Oracle $\log(t)$ Function}
			\label{fig:regret1}
		\end{subfigure}
		\hfill
		\begin{subfigure}[b]{0.45\textwidth}
			\centering
			\includegraphics[width=\textwidth]{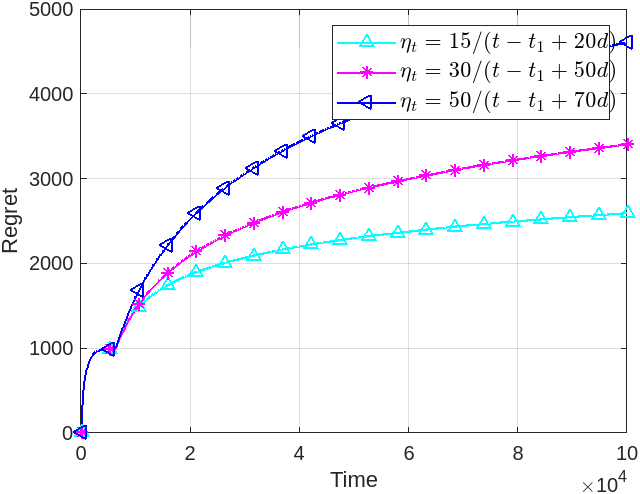}
			\caption{Regret Dynamics under Different Stepsizes}
			\label{fig:regret2}
		\end{subfigure}
		\caption{Regret dynamics \emph{versus} time/iteration for online QR. The dimension $d=100$, signal-to-noise ratio $\frac{\|\Bbeta^*\|}{\EE|\xi|}=20$, unknown horizon $T=10^5$, $t_{1.1}$ noise.}
		\label{fig:regret}
	\end{figure}

	\subsection{Real Data Example}
	
	In this section, we analyze a real-world dataset on news popularity provided by \cite{moniz2018multi}. Our focus is on news articles categorized under Economy that were published on Facebook. The objective is to predict the final popularity of each news item using three sets of predictors: the sentiment score of the headline text, the sentiment score of the body text, and the popularity levels measured at successive time intervals post-publication ($0-6h, \ldots, 18-24h$). The dataset contains a total of $N=29928$ samples and $d=7$ predictors. We randomly split the dataset into a training set with $n_1=20000$ observations and a test set with $n_2=9928$ observations. We apply the proposed online QR and batch learning algorithms to the training set. As benchmarks, we include the offline QR, implemented using \texttt{quantreg} \citep{portnoy1997gaussian}, and the ordinary least squares (OLS) estimators. It is worth noting that this dataset was also examined by \cite{chen2020quantile}, although their analysis was based on a different subset of the data, as some records from the original source are no longer available. We report the test errors $n_2^{-1} \sum_{(Y_i,\X_i)\in \text{TestData}} |Y_i-\X_i^{\top}\widehat{\Bbeta}|$ in Table~\ref{table:realdata}. The outcome variable $Y$ ranges from approximately $-1$ to $10^4$, with an average absolute value on the test set exceeding $50$. Despite this wide range, the prediction errors reported in Table~\ref{table:realdata} are substantially smaller, indicating strong predictive performance. Among the three QR methods--online, batch, and offline--both online and batch QR yield slightly lower prediction errors compared to offline QR. All three quantile-based approaches, however, significantly outperform the OLS in terms of test error. Moreover, consistent with the findings of \cite{chen2020quantile}, we observe that the estimated coefficients for the sentiment scores of both the headline and body text are close to zero. This suggests that, in this setting, sentiment features provide limited predictive value for news popularity.

	\begin{table}
		\begin{center}
			\begin{tabular}{|m{5em}|m{5em}|m{8em}|m{5em}|m{3em}||}
				\hline\hline
				Method& Online QR & Online Batch QR & Offline QR &  OLS \\
				\hline
				Test Error& 8.5824 & 8.5084 & 8.1895& 9.8752\\
				\hline
			\end{tabular}
		\end{center}
		\caption{Prediction performance on the news popularity dataset.}
		\label{table:realdata}
	\end{table}

	\section{Proof of Theorem~\ref{thm:one_sample}}
	\label{sec:proof_online}
	
	This section presents the proof of Theorem~\ref{thm:one_sample}. Proofs of the remaining results are deferred to the supplementary material. The update procedure $\Bbeta_{t+1}=\Bbeta_t-\eta_t\g_t$ yields
	\begin{align}
		\ltwo{\Bbeta_{t+1}-\Bbeta^*}^2&=\ltwo{\Bbeta_{t}-\Bbeta^*}^2-2\eta_t\inp{\Bbeta_{t}-\Bbeta^*}{\g_t}+\eta_t^2\ltwo{\g_t}^2 , 
		\label{eq:update one sample}
	\end{align}
	where the sub-gradient $\g_t$ is given by
	$\g_t=\X_t \{ -\tau\cdot1_{(Y_t>\X_t^{\top}\Bbeta_{t})}+(1-\tau)\cdot 1_{(Y_t<\X_t^{\top}\Bbeta_{t})} +\delta \cdot1_{(Y_t=\X_t^{\top}\Bbeta_{t})} \}$, and $\delta$ can take any value in the range $[-\tau,1-\tau]$. It follows that
	$$
	\ltwo{\g_t}^2 \geq\min\{\tau, 1-\tau \}^2 \ltwo{\X_t\cdot 1_{(Y_t\neq \X_t\Bbeta_t ) }}^2+\delta^2 \ltwo{\X_t\cdot 1_{ ( Y_t= \X_t\Bbeta_t ) }}^2
	$$
	and
	$$  
	\ltwo{\g_t}^2 \leq\max\{\tau , 1-\tau \}^2 \ltwo{\X_t\cdot 1_{(Y_t\neq \X_t\Bbeta_t)}}^2+\delta^2 \ltwo{\X_t\cdot 1_{(Y_t= \X_t\Bbeta_t) }}^2.
	$$
	By Lemma~\ref{teclem:gaussian_vector}, $\ltwo{\X_t}^2\leq 2\ku d$ holds with probability at least $1-\exp(-cd)$, which further implies that $\ltwo{\g_t}^2\leq 2\bar{\tau}^2 \ku   d$ for all $t\geq 0$, where $\bar \tau = \max(\tau, 1-\tau)$.
	
	\medskip
	\noindent
	{\sc  First Phase Analysis.}
	We aim to establish the convergence dynamics $$\ltwo{\Bbeta_{t}-\Bbeta^*}\leq C^* \left(1-c_5\frac{\kl}{\ku}\frac{1}{\bar{\tau}^2d}\right)^{t}\ltwo{\Bbeta_{0}-\Bbeta^*},$$ through induction, where the constant $C^*>1$ remains fixed throughout the proof and $0<c_5<1$ is some sufficiently small constant to be specified later.
	
	Initially, the inequality $\ltwo{\Bbeta_{0}-\Bbeta^*}\leq C^* \ltwo{\Bbeta_{0}-\Bbeta^*}$ holds trivially. For notational convenience, define $$D_l:=(1-c_5\frac{\kl}{\ku}\frac{1}{\bar{\tau}^2d})^{l}\ltwo{\Bbeta_{0}-\Bbeta^*} \text{ for } l=0, 1, \ldots.$$ Assuming the inductive hypothesis $\ltwo{\Bbeta_{l}-\Bbeta^*}\leq C^* D_l$ holds for all $l=0,1,\dots,t$, our goal is to establish that $\ltwo{\Bbeta_{t+1}-\Bbeta^*}\leq C^*  D_{t+1}$.

	By convexity and the property of sub-gradient, $\inp{\Bbeta_{t}-\Bbeta^*}{\g_t}\geq f_t(\Bbeta_{t})-f_t(\Bbeta^*)$. Substituting this into equation \eqref{eq:update one sample} yields
	\begin{multline*}
		\ltwo{\Bbeta_{t+1}-\Bbeta^*}^2 \leq \ltwo{\Bbeta_{t}-\Bbeta^*}^2-2\eta_t \{ (f_t(\Bbeta_{t}) -f_t(\Bbeta^*) \} + 2\eta_t^2 \ku\bar{\tau}^2d \\
		= \ltwo{\Bbeta_{t}-\Bbeta^*}^2-2\eta_t\EE \{ f_t(\Bbeta_{t}) -f_t(\Bbeta^*) | \Bbeta_{t} \} + 2 \eta_t^2 \ku\bar{\tau}^2d\\
		-2\eta_t \big[ f_t(\Bbeta_{t}) -f_t(\Bbeta^*)-\EE\{ f_t(\Bbeta_{t}) -f_t(\Bbeta^*) | \Bbeta_{t} \}  \big] .
	\end{multline*}
	For the expected difference $\EE \{ f_t(\Bbeta_{t}) -f_t(\Bbeta^*) | \Bbeta_{t}  \}$, Lemma~\ref{teclem:lowerboundofESR} implies
	\begin{align*}
		\EE \{ f_t(\Bbeta_{t}) -f_t(\Bbeta^*) | \Bbeta_{t}  \}
		&\geq\sqrt{\frac{\kl}{2\pi}} \ltwo{\Bbeta_{t}-\Bbeta^*}-\gamma.
	\end{align*}
	During phase one, where $\ltwo{\Bbeta_{t}-\Bbeta^*}\geq 8 \kl^{-1/2} \gamma$, it follows from the above lower bound that
	\begin{multline*}
		\ltwo{\Bbeta_{t+1}-\Bbeta^*}^2
		\leq \ltwo{\Bbeta_{t}-\Bbeta^*}^2-\frac{1}{3}\eta_t\sqrt{\kl} \ltwo{\Bbeta_{t}-\Bbeta^*} + 2\eta_t^2 \ku\bar{\tau}^2d\\
		-2\eta_t  \big[ f_t(\Bbeta_{t}) -f_t(\Bbeta^*)-\EE\{ f_t(\Bbeta_{t}) -f_t(\Bbeta^*) | \Bbeta_{t} \}  \big] .
	\end{multline*} 
	We further rewrite the above inequality as
	\begin{multline*}
		\ltwo{\Bbeta_{t+1}-\Bbeta^*}^2\leq\left(1-\frac{\sqrt{\kl}}{3}\frac{\eta_t}{\|\Bbeta_t-\Bbeta^*\|_2}\right)\|\Bbeta_t-\Bbeta^*\|_2^2+2\eta_t^2\bar{\tau}^2\ku d\\
		-2\eta_t  \big[ f_t(\Bbeta_{t}) -f_t(\Bbeta^*)-\EE\{ f_t(\Bbeta_{t}) -f_t(\Bbeta^*) | \Bbeta_{t} \}  \big].
	\end{multline*}
	By the induction hypothesis, $\|\Bbeta_t-\Bbeta^*\|\leq C^*D_t$, and hence
	\begin{multline*}
		\ltwo{\Bbeta_{t+1}-\Bbeta^*}^2\leq\left(1-\frac{\sqrt{\kl}}{3}\frac{\eta_t}{C^*D_t}\right)\|\Bbeta_t-\Bbeta^*\|_2^2+2\eta_t^2\bar{\tau}^2\ku d\\
		-2\eta_t  \big[ f_t(\Bbeta_{t}) -f_t(\Bbeta^*)-\EE\{ f_t(\Bbeta_{t}) -f_t(\Bbeta^*) | \Bbeta_{t} \}  \big]
	\end{multline*}
	Substituting the stepsize $\eta_t=c\frac{\sqrt{\kl}}{\ku}\frac{D_t} {\bar{\tau}^2d}$ into the inequality above, we obtain
	\begin{multline*}
		\ltwo{\Bbeta_{t+1}-\Bbeta^*}^2\leq\left(1-\frac{c}{3C^*}\frac{\kl}{\ku} \frac{1}{\bar{\tau}^2 d }\right)\|\Bbeta_t-\Bbeta^*\|_2^2+2c^2\frac{\kl}{\ku} \frac{1}{\bar{\tau}^2 d }D_t^2\\
		-2 c\frac{\sqrt{\kl}}{\ku}\frac{D_t}{\bar{\tau}^2d} \big[ f_t(\Bbeta_{t}) -f_t(\Bbeta^*)-\EE\{ f_t(\Bbeta_{t}) -f_t(\Bbeta^*) | \Bbeta_{t} \}  \big].
	\end{multline*}
	By applying this iterative bound repeatedly starting from $t=0$, we obtain
	\begin{align*}
		& \ltwo{\Bbeta_{t+1}-\Bbeta^*}^2  \\
		& \leq \left(1-\frac{c}{3C^*}\frac{\kl}{\ku}\frac{1}{\bar{\tau}^2d}\right)^{t+1}\ltwo{\Bbeta_{0}-\Bbeta^*}^2+2c^2\frac{\kl}{\ku}\frac{1}{d}\frac{1}{\bar{\tau}^2}\sum_{l=0}^{t}\left(1-\frac{c}{3C^*}\frac{\kl}{\ku}\frac{1}{\bar{\tau}^2d}\right)^{t-l}D_l^2\\
		&~~~~~ -2c\frac{\sqrt{\kl}}{\ku}\frac{1}{\bar{\tau}^2d} \sum_{l=0}^{t} \left(1-\frac{c}{3C^*}\frac{\kl}{\ku}\frac{1}{\bar{\tau}^2d}\right)^{t-l}D_l\big[ f_l(\Bbeta_{l}) -f_l(\Bbeta^*)-\EE\{f_t(\Bbeta_{t}) -f_t(\Bbeta^*) | \Bbeta_{t} \}  \big].
	\end{align*}
	We begin by bounding the second term on the right-hand side of the above inequality as
	\begin{align*}
		&  \sum_{l=0}^{t}\left(1-\frac{c}{3C^*}\frac{\kl}{\ku}\frac{1}{\bar{\tau}^2d}\right)^{t-l}D_l^2 \\
		& =\sum_{l=0}^{t}\left(1-\frac{c}{3C^*}\frac{\kl}{\ku}\frac{1}{\bar{\tau}^2d}\right)^{t-l}\left(1-c_5\frac{\kl}{\ku}\frac{1}{\bar{\tau}^2d}\right)^{2l}D_0^2\\
		& \leq \left(1-\frac{c}{3C^*}\frac{\kl}{\ku}\frac{1}{\bar{\tau}^2d}\right)^{t}\frac{1}{\frac{c}{3C^*}-c_5}\cdot\frac{\ku}{\kl}\bar{\tau}^2d \\
		& \leq\frac{6C^*}{c}\left(1-\frac{c}{3C^*}\frac{\kl}{\ku}\frac{1}{\bar{\tau}^2d}\right)^{2t}\frac{\ku}{\kl}\bar{\tau}^2d ,
	\end{align*}
	where the second line follows from the inequality $1-2a\leq(1-a)^2 $ and the condition $c_5\leq c/(6C^*)$. This further implies
	\begin{align*}
		2c^2\frac{\kl}{\ku}\frac{1}{d}\frac{1}{\bar{\tau}^2}\sum_{l=0}^{t}\left(1-\frac{c}{3C^*}\frac{\kl}{\ku}\frac{1}{\bar{\tau}^2d}\right)^{t-l}D_l^2\leq 12cC^*\left(1-\frac{c}{3C^*}\frac{\kl}{\ku}\frac{1}{\bar{\tau}^2d}\right)^{t}D_0^2.
	\end{align*}
	
	To bound the last stochastic term, under the inductive hypothesis $\cap_{l=0}^t\{\ltwo{\Bbeta_{l}-\Bbeta^*}^2\leq C^{*2}D_l^2\}$, we have 
	$$
	\|f_l(\Bbeta_{l}) -f_l(\Bbeta^*)-\EE\left[f_l(\Bbeta_{l}) -f_l(\Bbeta^*)\big| \Bbeta_{l}\right] \|_{\Psi_2}\leq C_1C^*\sqrt{\ku}\bar{\tau}D_l , \, l = 0, 1, \ldots, t,
	$$ 
	where $\|\cdot\|_{\Psi_2}$ denotes the Orlicz norm associated with the Orlicz function $\Psi(x)=\exp(x^2)-1$. Then, by applying a variant of Azuma's inequality \citep{shamir2011variant}, it follows that with probability at least  $1-\exp(-c_1d)$, 
	\begin{align*}
		&~~~~~\left|\sum_{l=0}^{t} \left(1-\frac{c}{3}\frac{\kl}{\ku}\frac{1}{\bar{\tau}^2d}\right)^{t-l}D_l\big[f_l(\Bbeta_{l}) -f_l(\Bbeta^*)-\EE \{ f_l(\Bbeta_{l}) -f_l(\Bbeta^*) | \Bbeta_{l}\} \big]\right|\\
		&\leq \bar{\tau} C^{*} \sqrt{\ku d\sum_{l=0}^{t} \left(1-\frac{c}{3C^*}\frac{\kl}{\ku}\frac{1}{\bar{\tau}^2d}\right)^{2t-2l}D_l^4}\\
		&= \bar{\tau}C^{*}\sqrt{ \ku d\sum_{l=0}^{t} \left(1-\frac{c}{3C^*}\frac{\kl}{\ku}\frac{1}{\bar{\tau}^2d}\right)^{2t-2l} \left(1-c_5\frac{\kl}{\ku}\frac{1}{\bar{\tau}^2d}\right)^{4l}}\times D_0^2\\
		&\leq \bar{\tau}C^{*}\sqrt{ \ku d \left(1-\frac{c}{3C^*}\frac{\kl}{\ku}\frac{1}{\bar{\tau}^2d}\right)^{2t} \cdot\frac{6C^*}{c}\frac{\ku}{\kl}\bar{\tau}^2d}\times D_0^2\\
		&\leq \bar{\tau}^2d\times\left(1-c_5\frac{\kl}{\ku}\frac{1}{\bar{\tau}^2d}\right)^{2t} D_{0}^2\times \sqrt{\frac{\ku^2}{\kl}}\sqrt{\frac{C^{*3}}{c}},
	\end{align*}
	where the last two inequalities follow from the bound $1-2a\leq (1-a)^2$ for $0<a<1$, together with the condition $c_5\leq c/(6C^*)$.

	Putting together the pieces, we conclude that for any $C^*$ satisfying $ c\leq C^*\leq c/(6c_5)$, it holds
	\begin{align*}
		&~~~~~\ltwo{\Bbeta_{t+1}-\Bbeta^*}^2\\
		&\leq\left(1-\frac{c}{3C^*}\frac{\kl}{\ku}\frac{1}{\bar{\tau}^2d}\right)^{t+1}\ltwo{\Bbeta_{0}-\Bbeta^*}^2+12cC^*\left(1-\frac{c}{3C^*}\frac{\kl}{\ku}\frac{1}{\bar{\tau}^2d}\right)^{t}D_0^2+\sqrt cC^{*3/2}D_{t+1}^2\\
		&\leq \frac{1}{3}D_{t+1}^2+12cC^*\left(1-c_5\frac{\kl}{\ku}\frac{1}{\bar{\tau}^2d}\right)^{2t}D_0^2\exp\left(t\log\left(\frac{1-\frac{c}{3C^*}\frac{\kl}{\ku}\frac{1}{\bar{\tau}^2d}}{(1-c_5\frac{\kl}{\ku}\frac{1}{\bar{\tau}^2d})^2}\right)\right)+\sqrt cC^{*3/2}D_{t+1}^2\\
		&\leq C^{*2}D_{t+1}^2 .
	\end{align*}
	This completes the proof of convergence in phase one.

	\medskip
	\noindent
	{\sc Second Phase Analysis}.
	The verification of the second phase is also carried out by induction.  The proof for the step $t_1+1$ is omitted due to its simplicity, as it follows analogously from the subsequent argument. We assume $\cap_{l=t_1}^{t}\{ \ltwo{\Bbeta_{l}-\Bbeta^*}^2\leq\frac{C^*d}{l-t_1+\cb d}\frac{\ku}{\kl^2}b_0^2\}$ and proceed to analyze the behavior at iteration $t+1$. Within this second phase, Lemma~\ref{teclem:loss_expectation} establishes a quadratic lower bound on the expected excess loss, that is,
	$$
	\EE\left[f_t(\Bbeta_{t}) -f_t(\Bbeta^*)\big|\Bbeta_{t}\right]\geq\frac{\kl}{12b_0}\ltwo{\Bbeta_t-\Bbeta^*}^2.
	$$
	Subsequently, the right-hand side of equation \eqref{eq:update one sample} can be upper bounded by
	\begin{align*}
		\ltwo{\Bbeta_{t+1}-\Bbeta^*}^2&\leq \ltwo{\Bbeta_{t}-\Bbeta^*}^2-\frac{1}{6b_0}\eta_t\kl \ltwo{\Bbeta_{t}-\Bbeta^*}^2 + 2\eta_t^2 \ku\bar{\tau}^2d\\
		&~~~~~~~~~~~~~~~~~~~~~~-2\eta_t\left[f_t(\Bbeta_{t}) -f_t(\Bbeta^*)-\EE\left[f_t(\Bbeta_{t}) -f_t(\Bbeta^*)\big| \Bbeta_{t}\right] \right].
	\end{align*}
	Incorporating the chosen stepsize $\eta_t=\frac{\ca}{t-t_1+\cb d}\frac{b_0}{\kl}$ preceding equation, we obtain
	\begin{align*}
		\ltwo{\Bbeta_{t+1}-\Bbeta^*}^2&\leq\left(1-\frac{\ca}{12(t-t_1+\cb d)}\right)\ltwo{\Bbeta_{t}-\Bbeta^*}^2+2b_0^2\bar{\tau}^2\frac{\ku}{\kl^2}\frac{\ca^2}{(t-t_1+\cb d)^2}d\\
		&~~~~~~~~~~~~~-2\frac{\ca}{t-t_1+\cb d}\frac{b_0}{\kl}\left[f_t(\Bbeta_{t}) -f_t(\Bbeta^*)-\EE\left[f_t(\Bbeta_{t}) -f_t(\Bbeta^*)\big| \Bbeta_{t}\right] \right].
	\end{align*}
	Summing the above inequality over $t$ from $t_1$ onward, we arrive at
	\begin{equation}
		\begin{split}
			&~~~\ltwo{\Bbeta_{t+1}-\Bbeta^*}^2\\
			&\leq\prod_{l=t_1}^{t}\left(1-\frac{\ca}{12(l-t_1+\cb d)}\right)\ltwo{\Bbeta_{t_1}-\Bbeta^*}^2\\
			&~~~+2b_0^2\bar{\tau}^2\frac{\ku}{\kl^2}d\sum_{l=t_1}^{t}\left(1-\frac{\ca}{12(l+1-t_1+\cb d)}\right)\cdots \left(1-\frac{\ca}{12(t-t_1+\cb d)}\right)\frac{\ca^2}{(l-t_1+\cb d)^2}\\
			&~~~-2\frac{b_0}{\kl} \sum_{l=t_1}^{t}\left(1-\frac{\ca}{12(l+1-t_1+\cb d)}\right)\cdots \left(1-\frac{\ca}{12(t-t_1+\cb d)}\right) \frac{\ca}{l-t_1+\cb d}\\&~~~~~~~~~~\times\left[f_l(\Bbeta_{l}) -f_l(\Bbeta^*)-\EE\left[f_l(\Bbeta_{l}) -f_l(\Bbeta^*)\big| \Bbeta_{l}\right] \right].
		\end{split}
		\label{eq2}
	\end{equation}

	A sharp bound on the above equation constitutes a key step in the proof. We proceed by analyzing each of the three terms on the right-hand side separately.
	Note that when $\ca>12$, we have 
	$$
	\bigg(1-\frac{\ca}{12(l-t_1+\cb d)}\bigg)\leq \frac{l-t_1+\cb d}{l+1-t_1+\cb d}.
	$$
	Then, the first term on the right-hand side of \eqref{eq2} satisfies
	$$
	\prod_{l=t_1}^{t}\left(1-\frac{\ca}{12(l-t_1+\cb d)}\right)\ltwo{\Bbeta_{t_1}-\Bbeta^*}^2\leq\frac{\cb d}{t+1-t_1+\cb d}\ltwo{\Bbeta_{t_1}-\Bbeta^*}^2. 
	$$
	Consider the product sequence, which can be bounded as
	\begin{equation}
		\begin{split}
			&~~~\left(1-\frac{\ca}{12(l+1-t_1+\cb d)}\right)\cdots \left(1-\frac{\ca}{12(t-t_1+\cb d)}\right)\\&=\exp\left(\sum_{k=l+1}^{t}\log\left(1-\frac{\ca}{12(k-t_1+\cb d)}\right)\right)\\
			&\leq \exp\left(-\sum_{k=l+1}^{t}\frac{\ca}{12(k-t_1+\cb d)}\right)\leq \exp\left(-\frac{\ca}{12}\int_{l+1-t_1}^{t+1-t_1} \frac{1}{x+\cb d}\; dx\right)\\
			&\leq\exp\left(-\frac{\ca}{12}\log\left(\frac{t+1-t_1+\cb d}{l+1-t_1+\cb d}\right)\right)=\left(\frac{l+1-t_1+\cb d}{t+1-t_1+\cb d}\right)^{\frac{\ca}{12}},
			\label{eq:1-1}
		\end{split}
	\end{equation}
	where the first inequality uses the bound $\log(1+x)\leq x$, and the third line leverages the relationship between integrals and discrete sums. Consequently, the second term in equation \eqref{eq2} can be bounded as
	\begin{align*}
		&~~~\sum_{l=t_1}^{t}\left(1-\frac{\ca}{12(l+1-t_1+\cb d)}\right)\cdots \left(1-\frac{\ca}{12(t-t_1+\cb d)}\right)\frac{\ca^2}{(l-t_1+\cb d)^2}\\
		&\leq \sum_{l=t_1}^{t}\left(\frac{l+1-t_1+\cb d}{t+1-t_1+\cb d}\right)^{\frac{\ca}{12}}\frac{\ca^2}{(l-t_1+\cb d)^2}\\
		&\leq \left(\frac{1+\cb d}{\cb d}\right)^{\frac{\ca}{12}}\frac{\ca^2}{( t+1-t_1+\cb d)^{\frac{\ca}{12}}}\sum_{l=t_1}^{t}(l-t_1+\cb d)^{\frac{\ca}{12}-2}.
	\end{align*}
	With $\ca>12$, we have $$\sum_{l=t_1}^{t}(l-t_1+\cb d)^{\frac{\ca}{12}-2}\leq\int_{0}^{t-t_1+1}(x+\cb d)^{\frac{\ca}{12}-2}\; dx=\frac{1}{\frac{\ca}{12}-1}(t-t_1+1+\cb d )^{\frac{\ca}{12}-1}.$$ Combining these results, we obtain the following upper bound for the second term in equation~\eqref{eq2}:
	\begin{align*}
		&~~~\sum_{l=t_1}^{t}\left(1-\frac{\ca}{12(l+1-t_1+\cb d)}\right)\cdots \left(1-\frac{\ca}{12(t-t_1+\cb d)}\right)\frac{\ca^2}{(l-t_1+\cb d)^2}\\
		&\leq \left(\frac{1+\cb d}{\cb d}\right)^{\frac{\ca}{12}}\frac{\ca^2}{ t+1-t_1+\cb d}.
	\end{align*}
	Assuming $\cap_{l=t_1}^t\{\ltwo{\Bbeta_{l}-\Bbeta^*}^2\leq\frac{C^*d}{l-t_1+\cb d}\frac{\ku}{\kl^2}b_0^2\}$, it follows that
	\begin{align*}
		\| f_l(\Bbeta_{l}) -f_l(\Bbeta^*)-\EE\left[f_l(\Bbeta_{l}) -f_l(\Bbeta^*)\big| \Bbeta_{l}\right]\|_{\Psi_2}\leq C\bar{\tau}\sqrt{\frac{C^*d}{l-t_1+\cb d}\frac{\ku^2}{\kl^2} }b_0.
	\end{align*}
	By Theorem 2 in \cite{shamir2011variant} and equation~\eqref{eq:1-1}, with probability exceeding $1-c\exp(-cd)$, the third term in equation~\eqref{eq2} can be bounded by
	\begin{align*}
		&~~~\frac{b_0}{\kl}\sum_{l=t_1}^{t}\left(1-\frac{\ca}{12(l+1-t_1+\cb d)}\right)\cdots \left(1-\frac{\ca}{12(t-t_1+\cb d)}\right) \frac{\ca}{l-t_1+\cb d}\\&~~~~~~~~~~~~~~~~~~~~~~~~~~~~~~~~~~~~~~~~~~~~~~~~~~~~~~~~~~~~~~\times\left[f_l(\Bbeta_{l}) -f_l(\Bbeta^*)-\EE\left[f_l(\Bbeta_{l}) -f_l(\Bbeta^*)\big| \Bbeta_{l}\right] \right]\\
		&\leq \bar{\tau}\ca\frac{\ku}{\kl^2}\sqrt{\sum_{l=t_1}^{t}\left(\frac{l+1-t_1+\cb d}{t+1-t_1+\cb d} \right)^{\frac{\ca}{6}}\frac{1}{(l-t_1+\cb d)^2} \frac{C^*d^2}{l-t_1+\cb d}}b_0^2\\
		&=\bar{\tau}\ca\frac{\ku}{\kl^2}b_0^2\frac{\sqrt{C^*}d}{(t+1-t_1+\cb d)^{\frac{\ca}{12}}}\sqrt{\left(\frac{1+\cb d}{\cb d}\right)^{\frac{\ca}{6}}\sum_{l=t_1}^t (l-t_1+\cb d)^{\frac{\ca}{6}-3}}\\
		&\leq \bar{\tau}\ca\frac{\ku}{\kl^2}b_0^2\frac{\sqrt{C^*}d}{t+1-t_1+\cb d}\times \left(\frac{1+\cb d}{\cb d}\right)^{\frac{\ca}{12}}.
	\end{align*}
	Finally, combining the upper bounds of the three terms, we obtain
	\begin{align*}
		\ltwo{\Bbeta_{t+1}-\Bbeta^*}^2&\leq \frac{\cb d}{t+1-t_1+\cb d}\ltwo{\Bbeta_{t_1}-\Bbeta^*}^2+2\left(\frac{1+\cb d}{\cb d}\right)^{\frac{\ca}{12}}\frac{\ca^2}{ t+1-t_1+\cb d}\\
		&~~~~~~~~~~~~~~~~~~~~~~~~~~~~~~~~~~~~~~~~~~~~~~+\ca\frac{\ku}{\kl^2}b_0^2\frac{\sqrt{C^*}d}{t+1-t_1+\cb d}\times \left(\frac{1+\cb d}{\cb d}\right)^{\frac{\ca}{12}}\\
		&\leq C^*\frac{\ku}{\kl^2}b_0^2\frac{d}{t+1-t_1+\cb d},
	\end{align*}
	where $\{ (1+\cb d)/(\cb d\}^{\frac{\ca}{12}}<2^{\ca/12}$ is independent of $d$, and $C^*$ is a sufficiently large constant that may depend on  $\ca,\cb$. This completes the proof of Theorem~\ref{thm:one_sample}. 

	\section{Discussions}
	
	This paper focuses on online quantile regression in low-dimensional settings. The analytical framework developed herein can also be extended to the study of stochastic sub-gradient methods for low-rank regression under quantile loss. Let $\mathbf{M}^* \in \mathbb{R}^{d_1 \times d_2}$ denote the true low-rank matrix with $\operatorname{rank}(\mathbf{M}^*) = r$. At each time step $t$, we observe data $(Y_t, \mathbf{X}_t)$, where $Y_t = \inp{\mathbf{X}_t}{\mathbf{M}^*} + \xi_t$, and $\mathbf{X}_t \in \mathbb{R}^{d_1 \times d_2}$ is the sensing matrix. The corresponding loss function at time $t$ is defined as $f_t(\mathbf{M}) = \rho_{Q,\tau}(Y_t - \inp{\mathbf{X}_t}{\mathbf{M}})$. A central challenge in this setting is to maintain the low-rank structure of the iterates over time. To this end, the online Riemannian optimization framework updates the estimate according to
	$$
	\M_{t+1}=\operatorname{SVD}_r(\M_t-\eta_t\calP_{\TT_t}(\G_t)),
	$$
	where $\G_t\in\partial f_t(\M_t)$ is the sub-gradient, and $\calP_{\TT_t}(\cdot)$ denotes the projection onto the tangent space at $\mathbf{M}_t$. Detailed expressions for $\calP_{\TT_t}(\mathbf{G}_t)$ and further discussion of Riemannian optimization techniques can be found in \citet{vandereycken2013low} and \citet{mishra2014fixed}. Each iteration retains only the leading $r$ singular components, raising an important conceptual question: to what extent does this truncation preserve essential statistical information? A comprehensive theoretical investigation of online low-rank regression under quantile loss is beyond the scope of this paper and is left for future research.

	\section{Acknowledgement}
	The authors would like to thank Professor Genevera Allen and the two anonymous reviewers for their insightful comments and constructive suggestions that helped improve this work. Yinan Shen is the corresponding author and acknowledges support from the WiSE Supplemental Faculty Support program at the University of Southern California. Dong Xia's research was partially supported by the Hong Kong Research Grants Council under Grant GRF 16302323. Wen-Xin Zhou's research was supported by the National Science Foundation under Grant DMS-2401268 and by the Australian Research Council Discovery Project Grant DP230100147.

	\bibliographystyle{plainnat}
	\bibliography{reference}

\begin{thebibliography}{65}
\providecommand{\natexlab}[1]{#1}
\providecommand{\url}[1]{\texttt{#1}}
\expandafter\ifx\csname urlstyle\endcsname\relax
  \providecommand{\doi}[1]{doi: #1}\else
  \providecommand{\doi}{doi: \begingroup \urlstyle{rm}\Url}\fi

\bibitem[Bach and Moulines(2013)]{bach2013non}
Francis Bach and Eric Moulines.
\newblock Non-strongly-convex smooth stochastic approximation with convergence
  rate ${O}(1/n)$.
\newblock \emph{Advances in Neural Information Processing Systems}, 26, 2013.

\bibitem[Bastani and Bayati(2020)]{bastani2020online}
Hamsa Bastani and Mohsen Bayati.
\newblock Online decision making with high-dimensional covariates.
\newblock \emph{Operations Research}, 68\penalty0 (1):\penalty0 276--294, 2020.

\bibitem[Belloni and Chernozhukov(2011)]{BC2011}
Alexandre Belloni and Victor Chernozhukov.
\newblock $\ell_1$-penalized quantile regression in high-dimensional sparse
  models.
\newblock \emph{The Annals of Statistics}, 39\penalty0 (1):\penalty0 82--130,
  2011.

\bibitem[Bottou(1999)]{bottou1999line}
L{\'e}on Bottou.
\newblock Online learning and stochastic approximations.
\newblock In David Saad, editor, \emph{On-line Learning in Neural Networks},
  pages 9--42. Cambridge University Press, 1999.

\bibitem[Cai et~al.(2023)Cai, Li, and Xia]{cai2023online}
Jian-Feng Cai, Jingyang Li, and Dong Xia.
\newblock Online tensor learning: {C}omputational and statistical trade-offs,
  adaptivity and optimal regret.
\newblock \emph{arXiv preprint arXiv:2306.03372}, 2023.

\bibitem[Cambier and Absil(2016)]{cambier2016robust}
L{\'e}opold Cambier and P.-A. Absil.
\newblock Robust low-rank matrix completion by riemannian optimization.
\newblock \emph{SIAM Journal on Scientific Computing}, 38\penalty0
  (5):\penalty0 S440--S460, 2016.

\bibitem[Cesa-Bianchi and Lugosi(2006)]{cesa2006prediction}
Nicolo Cesa-Bianchi and G{\'a}bor Lugosi.
\newblock \emph{Prediction, Learning, and Games}.
\newblock Cambridge University Press, 2006.

\bibitem[Charisopoulos et~al.(2021)Charisopoulos, Chen, Davis, Diaz, Ding, and
  Drusvyatskiy]{charisopoulos2021low}
Vasileios Charisopoulos, Yudong Chen, Damek Davis, Mateo Diaz, Lijun Ding, and
  Dmitriy Drusvyatskiy.
\newblock Low-rank matrix recovery with composite optimization: good
  conditioning and rapid convergence.
\newblock \emph{Foundations of Computational Mathematics}, 21\penalty0
  (6):\penalty0 1505--1593, 2021.

\bibitem[Chen and Zhou(2020)]{chen2020quantile}
Lanjue Chen and Yong Zhou.
\newblock Quantile regression in big data: A divide and conquer based strategy.
\newblock \emph{Computational Statistics \& Data Analysis}, 144:\penalty0
  106892, 2020.

\bibitem[Chen et~al.(2019)Chen, Liu, and Zhang]{chen2019quantile}
Xi~Chen, Weidong Liu, and Yichen Zhang.
\newblock Quantile regression under memory constraint.
\newblock \emph{The Annals of Statistics}, 47\penalty0 (6):\penalty0
  3244--3273, 2019.

\bibitem[Chen and Yuan(2024)]{chen2024renewable}
Xuerong Chen and Senlin Yuan.
\newblock Renewable quantile regression with heterogeneous streaming datasets.
\newblock \emph{Journal of Computational and Graphical Statistics}, 33\penalty0
  (4):\penalty0 1185--1201, 2024.

\bibitem[Delyon and Juditsky(1993)]{delyon1993accelerated}
Bernard Delyon and Anatoli Juditsky.
\newblock Accelerated stochastic approximation.
\newblock \emph{SIAM Journal on Optimization}, 3\penalty0 (4):\penalty0
  868--881, 1993.

\bibitem[Do et~al.(2009)Do, Le, and Foo]{do2009proximal}
Chuong~B Do, Quoc~V Le, and Chuan-Sheng Foo.
\newblock Proximal regularization for online and batch learning.
\newblock In \emph{Proceedings of the 26th Annual International Conference on
  Machine Learning}, pages 257--264, 2009.

\bibitem[Duchi and Singer(2009)]{duchi2009efficient}
John Duchi and Yoram Singer.
\newblock Efficient online and batch learning using forward backward splitting.
\newblock \emph{Journal of Machine Learning Research}, 10\penalty0
  (99):\penalty0 2899--2934, 2009.

\bibitem[Duchi et~al.(2011)Duchi, Hazan, and Singer]{duchi2011adaptive}
John Duchi, Elad Hazan, and Yoram Singer.
\newblock Adaptive subgradient methods for online learning and stochastic
  optimization.
\newblock \emph{Journal of Machine Learning Research}, 12\penalty0
  (7):\penalty0 2121--2159, 2011.

\bibitem[Fan et~al.(2018)Fan, Gong, Li, and Sun]{fan2018statistical}
Jianqing Fan, Wenyan Gong, Chris~Junchi Li, and Qiang Sun.
\newblock Statistical sparse online regression: A diffusion approximation
  perspective.
\newblock In \emph{International Conference on Artificial Intelligence and
  Statistics}, pages 1017--1026. PMLR, 2018.

\bibitem[Fernandes et~al.(2021)Fernandes, Guerre, and Horta]{FGH2021}
Marcelo Fernandes, Emmanuel Guerre, and Eduardo Horta.
\newblock Smoothing quantile regressions.
\newblock \emph{Journal of Business and Economic Statistics}, 39\penalty0
  (1):\penalty0 338--357, 2021.

\bibitem[Finn et~al.(2019)Finn, Rajeswaran, Kakade, and Levine]{finn2019online}
Chelsea Finn, Aravind Rajeswaran, Sham Kakade, and Sergey Levine.
\newblock Online meta-learning.
\newblock In \emph{International Conference on Machine Learning}, pages
  1920--1930. PMLR, 2019.

\bibitem[Gao et~al.(2019)Gao, Han, Ren, and Zhou]{gao2019batched}
Zijun Gao, Yanjun Han, Zhimei Ren, and Zhengqing Zhou.
\newblock Batched multi-armed bandits problem.
\newblock \emph{Advances in Neural Information Processing Systems}, 32, 2019.

\bibitem[Goldenshluger and Zeevi(2013)]{goldenshluger2013linear}
Alexander Goldenshluger and Assaf Zeevi.
\newblock A linear response bandit problem.
\newblock \emph{Stochastic Systems}, 3\penalty0 (1):\penalty0 230--261, 2013.

\bibitem[Han et~al.(2022)Han, Sun, and Zhang]{han2022online}
Qiyu Han, Will~Wei Sun, and Yichen Zhang.
\newblock Online statistical inference for matrix contextual bandit.
\newblock \emph{arXiv preprint arXiv:2212.11385}, 2022.

\bibitem[Han et~al.(2024)Han, Luo, Lin, and Huang]{han2023online}
Ruijian Han, Lan Luo, Yuanyuan Lin, and Jian Huang.
\newblock Online inference with debiased stochastic gradient descent.
\newblock \emph{Biometrika}, 111\penalty0 (1):\penalty0 93--108, 2024.

\bibitem[Han et~al.(2020)Han, Zhou, Zhou, Blanchet, Glynn, and
  Ye]{han2020sequential}
Yanjun Han, Zhengqing Zhou, Zhengyuan Zhou, Jose Blanchet, Peter~W Glynn, and
  Yinyu Ye.
\newblock Sequential batch learning in finite-action linear contextual bandits.
\newblock \emph{arXiv preprint arXiv:2004.06321}, 2020.

\bibitem[Hazan(2016)]{hazan2016introduction}
Elad Hazan.
\newblock Introduction to online convex optimization.
\newblock \emph{Foundations and Trends{\textregistered} in Optimization},
  2\penalty0 (3-4):\penalty0 157--325, 2016.

\bibitem[Hazan et~al.(2007{\natexlab{a}})Hazan, Agarwal, and
  Kale]{hazan2007logarithmic}
Elad Hazan, Amit Agarwal, and Satyen Kale.
\newblock Logarithmic regret algorithms for online convex optimization.
\newblock \emph{Machine Learning}, 69:\penalty0 169--192, 2007{\natexlab{a}}.

\bibitem[Hazan et~al.(2007{\natexlab{b}})Hazan, Rakhlin, and
  Bartlett]{hazan2007adaptive}
Elad Hazan, Alexander Rakhlin, and Peter Bartlett.
\newblock Adaptive online gradient descent.
\newblock \emph{Advances in Neural Information Processing Systems}, 20,
  2007{\natexlab{b}}.

\bibitem[He and Shao(2000)]{he2000parameters}
Xuming He and Qi-Man Shao.
\newblock On parameters of increasing dimensions.
\newblock \emph{Journal of Multivariate Analysis}, 73\penalty0 (1):\penalty0
  120--135, 2000.

\bibitem[He et~al.(2023)He, Pan, Tan, and Zhou]{he2021smoothed}
Xuming He, Xiaoou Pan, Kean~Ming Tan, and Wen-Xin Zhou.
\newblock Smoothed quantile regression with large-scale inference.
\newblock \emph{Journal of Econometrics}, 232\penalty0 (2):\penalty0 367--388,
  2023.

\bibitem[Hoffman et~al.(2010)Hoffman, Bach, and Blei]{hoffman2010online}
Matthew Hoffman, Francis Bach, and David Blei.
\newblock Online learning for latent dirichlet allocation.
\newblock \emph{Advances in Neural Information Processing Systems}, 23, 2010.

\bibitem[Jiang and Yu(2022)]{jiang2022renewable}
Rong Jiang and Keming Yu.
\newblock Renewable quantile regression for streaming data sets.
\newblock \emph{Neurocomputing}, 508:\penalty0 208--224, 2022.

\bibitem[Johnson and Zhang(2013)]{johnson2013accelerating}
Rie Johnson and Tong Zhang.
\newblock Accelerating stochastic gradient descent using predictive variance
  reduction.
\newblock \emph{Advances in Neural Information Processing Systems}, 26, 2013.

\bibitem[Koenker(2005)]{koenker2005quantile}
Roger Koenker.
\newblock \emph{Quantile Regression}.
\newblock Cambridge University Press, 2005.

\bibitem[Koenker and Bassett~Jr(1978)]{koenker1978regression}
Roger Koenker and Gilbert Bassett~Jr.
\newblock Regression quantiles.
\newblock \emph{Econometrica}, 46\penalty0 (1):\penalty0 33--50, 1978.

\bibitem[Koenker et~al.()Koenker, Chernozhukov, He, and
  Peng]{koenker2017handbook}
Roger Koenker, Victor Chernozhukov, Xuming He, and Liming Peng.
\newblock \emph{Handbook of Quantile Regression}.
\newblock CRC Press.

\bibitem[Koenker and d'Orey(1987)]{koenker1987algorithm}
Roger~W Koenker and Vasco d'Orey.
\newblock Algorithm as 229: Computing regression quantiles.
\newblock \emph{Applied Statistics}, pages 383--393, 1987.

\bibitem[Langford et~al.(2009)Langford, Li, and Zhang]{langford2009sparse}
John Langford, Lihong Li, and Tong Zhang.
\newblock Sparse online learning via truncated gradient.
\newblock \emph{Journal of Machine Learning Research}, 10\penalty0
  (3):\penalty0 777--801, 2009.

\bibitem[LeCun et~al.(1989)LeCun, Boser, Denker, Henderson, Howard, Hubbard,
  and Jackel]{lecun1989backpropagation}
Yann LeCun, Bernhard Boser, John~S Denker, Donnie Henderson, Richard~E Howard,
  Wayne Hubbard, and Lawrence~D Jackel.
\newblock Backpropagation applied to handwritten zip code recognition.
\newblock \emph{Neural Computation}, 1\penalty0 (4):\penalty0 541--551, 1989.

\bibitem[Li et~al.(2020)Li, Zhu, Man-Cho~So, and Vidal]{li2020nonconvex}
Xiao Li, Zhihui Zhu, Anthony Man-Cho~So, and Rene Vidal.
\newblock Nonconvex robust low-rank matrix recovery.
\newblock \emph{SIAM Journal on Optimization}, 30\penalty0 (1):\penalty0
  660--686, 2020.

\bibitem[Ma and Fattahi(2023)]{ma2023global}
Jianhao Ma and Salar Fattahi.
\newblock Global convergence of sub-gradient method for robust matrix recovery:
  Small initialization, noisy measurements, and over-parameterization.
\newblock \emph{Journal of Machine Learning Research}, 24\penalty0
  (96):\penalty0 1--84, 2023.

\bibitem[McMahan and Streeter(2010)]{mcmahan2010adaptive}
H~Brendan McMahan and Matthew Streeter.
\newblock Adaptive bound optimization for online convex optimization.
\newblock In \emph{COLT}, 2010.

\bibitem[Mhammedi et~al.(2019)Mhammedi, Koolen, and
  Van~Erven]{mhammedi2019lipschitz}
Zakaria Mhammedi, Wouter~M Koolen, and Tim Van~Erven.
\newblock Lipschitz adaptivity with multiple learning rates in online learning.
\newblock In \emph{Conference on Learning Theory}, pages 2490--2511. PMLR,
  2019.

\bibitem[Mishra et~al.(2014)Mishra, Meyer, Bonnabel, and
  Sepulchre]{mishra2014fixed}
Bamdev Mishra, Gilles Meyer, Silvere Bonnabel, and Rodolphe Sepulchre.
\newblock Fixed-rank matrix factorizations and riemannian low-rank
  optimization.
\newblock \emph{Computational Statistics}, 29:\penalty0 591--621, 2014.

\bibitem[Moniz and Torgo(2018)]{moniz2018multi}
Nuno Moniz and Luis Torgo.
\newblock Multi-source social feedback of online news feeds.
\newblock \emph{arXiv preprint arXiv:1801.07055}, 2018.

\bibitem[Orabona(2019)]{orabona2019modern}
Francesco Orabona.
\newblock A modern introduction to online learning.
\newblock \emph{arXiv preprint arXiv:1912.13213}, 2019.

\bibitem[Parikh and Boyd(2014)]{parikh2014proximal}
Neal Parikh and Stephen Boyd.
\newblock Proximal algorithms.
\newblock \emph{Foundations and Trends{\textregistered} in Optimization},
  1\penalty0 (3):\penalty0 127--239, 2014.

\bibitem[Portnoy and Koenker(1997)]{portnoy1997gaussian}
Stephen Portnoy and Roger Koenker.
\newblock The {G}aussian hare and the {L}aplacian tortoise: computability of
  squared-error versus absolute-error estimators.
\newblock \emph{Statistical Science}, 12\penalty0 (4):\penalty0 279--300, 1997.

\bibitem[Puchkin et~al.(2024)Puchkin, Gorbunov, Kutuzov, and
  Gasnikov]{puchkin2023breaking}
Nikita Puchkin, Eduard Gorbunov, Nickolay Kutuzov, and Alexander Gasnikov.
\newblock Breaking the heavy-tailed noise barrier in stochastic optimization
  problems.
\newblock In \emph{Proceedings of The 27th International Conference on
  Artificial Intelligence and Statistics}, volume 238, pages 856--864. PMLR,
  2024.

\bibitem[Rajalakshmi et~al.(2019)Rajalakshmi, Saranya, and
  Shanmugavadivu]{rajalakshmi2019pattern}
M~Rajalakshmi, P~Saranya, and P~Shanmugavadivu.
\newblock Pattern recognition-recognition of handwritten document using
  convolutional neural networks.
\newblock In \emph{2019 IEEE International Conference on Intelligent Techniques
  in Control, Optimization and Signal Processing (INCOS)}, pages 1--7. IEEE,
  2019.

\bibitem[Rakhlin et~al.(2011)Rakhlin, Shamir, and Sridharan]{rakhlin2011making}
Alexander Rakhlin, Ohad Shamir, and Karthik Sridharan.
\newblock Making gradient descent optimal for strongly convex stochastic
  optimization.
\newblock \emph{arXiv preprint arXiv:1109.5647}, 2011.

\bibitem[Ren and Zhou(2023)]{ren2023dynamic}
Zhimei Ren and Zhengyuan Zhou.
\newblock Dynamic batch learning in high-dimensional sparse linear contextual
  bandits.
\newblock \emph{Management Science}, 70\penalty0 (2), 2023.

\bibitem[Roux et~al.(2012)Roux, Schmidt, and Bach]{roux2012stochastic}
Nicolas Roux, Mark Schmidt, and Francis Bach.
\newblock A stochastic gradient method with an exponential convergence rate for
  finite training sets.
\newblock \emph{Advances in Neural Information Processing Systems}, 25, 2012.

\bibitem[Savage(1951)]{savage1951theory}
Leonard~J Savage.
\newblock The theory of statistical decision.
\newblock \emph{Journal of the American Statistical Association}, 46\penalty0
  (253):\penalty0 55--67, 1951.

\bibitem[Shamir(2011)]{shamir2011variant}
Ohad Shamir.
\newblock A variant of {A}zuma's inequality for martingales with subgaussian
  tails.
\newblock \emph{arXiv preprint arXiv:1110.2392}, 2011.

\bibitem[Shen et~al.(2023)Shen, Li, Cai, and Xia]{shen2023computationally}
Yinan Shen, Jingyang Li, Jian-Feng Cai, and Dong Xia.
\newblock Computationally efficient and statistically optimal robust
  high-dimensional linear regression.
\newblock \emph{arXiv preprint arXiv:2305.06199}, 2023.

\bibitem[Soleymani and Paquet(2020)]{soleymani2020financial}
Farzan Soleymani and Eric Paquet.
\newblock Financial portfolio optimization with online deep reinforcement
  learning and restricted stacked autoencoder—deepbreath.
\newblock \emph{Expert Systems with Applications}, 156:\penalty0 113456, 2020.

\bibitem[Streeter and McMahan(2010)]{streeter2010less}
Matthew Streeter and H~Brendan McMahan.
\newblock Less regret via online conditioning.
\newblock \emph{arXiv preprint arXiv:1002.4862}, 2010.

\bibitem[Sun et~al.(2023)Sun, Wang, Cai, Yao, and Wang]{sun2023online}
Xiaofei Sun, Hongwei Wang, Chao Cai, Mei Yao, and Kangning Wang.
\newblock Online renewable smooth quantile regression.
\newblock \emph{Computational Statistics \& Data Analysis}, 185:\penalty0
  107781, 2023.

\bibitem[Tong et~al.(2021)Tong, Ma, and Chi]{tong2021low}
Tian Tong, Cong Ma, and Yuejie Chi.
\newblock Low-rank matrix recovery with scaled subgradient methods: Fast and
  robust convergence without the condition number.
\newblock \emph{IEEE Transactions on Signal Processing}, 69:\penalty0
  2396--2409, 2021.

\bibitem[Vandereycken(2013)]{vandereycken2013low}
Bart Vandereycken.
\newblock Low-rank matrix completion by riemannian optimization.
\newblock \emph{SIAM Journal on Optimization}, 23\penalty0 (2):\penalty0
  1214--1236, 2013.

\bibitem[Vershynin(2018)]{vershynin2018high}
Roman Vershynin.
\newblock \emph{High-dimensional Probability: An Introduction with Applications
  in Data Science}.
\newblock Cambridge University Press, 2018.

\bibitem[Wang et~al.(2022)Wang, Wang, and Li]{wang2022renewable}
Kangning Wang, Hongwei Wang, and Shaomin Li.
\newblock Renewable quantile regression for streaming datasets.
\newblock \emph{Knowledge-Based Systems}, 235:\penalty0 107675, 2022.

\bibitem[Xia and Yuan(2021)]{xia2021statistical}
Dong Xia and Ming Yuan.
\newblock Statistical inferences of linear forms for noisy matrix completion.
\newblock \emph{Journal of the Royal Statistical Society Series B: Statistical
  Methodology}, 83\penalty0 (1):\penalty0 58--77, 2021.

\bibitem[Zhang(2004)]{zhang2004solving}
Tong Zhang.
\newblock Solving large scale linear prediction problems using stochastic
  gradient descent algorithms.
\newblock In \emph{Proceedings of the 21st International Conference on Machine
  Learning}, page 116, 2004.

\bibitem[Zinkevich(2003)]{zinkevich2003online}
Martin Zinkevich.
\newblock Online convex programming and generalized infinitesimal gradient
  ascent.
\newblock In \emph{Proceedings of the 20th International Conference on Machine
  Learning}, pages 928--936, 2003.

\bibitem[Zou et~al.(2020)Zou, Wang, and Li]{zou2020consistent}
Changliang Zou, Guanghui Wang, and Runze Li.
\newblock Consistent selection of the number of change-points via
  sample-splitting.
\newblock \emph{The Annals of Statistics}, 48\penalty0 (1):\penalty0 413, 2020.

\end{thebibliography}
	
	\newpage
	
	\appendix

	\begin{center}
		\textbf{\Large{Supplementary Material to ``Online Quantile Regression"}}
	\end{center}
	\section{Proof of Main Results}
	This section presents proofs of the main results for each setting.
	
	\subsection{Proof of Online Learning}
	In this subsection, we shall first prove Remark~\ref{rmk:one_sample_init}, and then prove the regret bound, which is a subsequent result of the expected estimation error rates.
	
	\subsubsection{Proof of Remark~\ref{rmk:one_sample_init}: Online Learning with Good Initialization}
	We are going to prove the convergence dynamics when $\ltwo{\Bbeta_{0}-\Bbeta^*}< 8\sqrt{\kl^{-1}}\gamma$. According to proof of Theorem~\ref{thm:one_sample}, we have
	\begin{equation}
		\begin{split}
			&~~~\ltwo{\Bbeta_{t+1}-\Bbeta^*}^2\leq\prod_{l=0}^{t}\left(1-\frac{\ca}{12(l+\cb d)}\right)\ltwo{\Bbeta_{0}-\Bbeta^*}^2\\
			&~~~+2b_0^2\bar{\tau}^2\frac{\ku}{\kl^2}d\sum_{l=0}^{t}\left(1-\frac{\ca}{12(l+1+\cb d)}\right)\cdots \left(1-\frac{\ca}{12(t+\cb d)}\right)\frac{\ca^2}{(l+\cb d)^2}\\
			&~~~-2\frac{b_0}{\kl} \sum_{l=0}^{t}\left(1-\frac{\ca}{12(l+1+\cb d)}\right)\cdots \left(1-\frac{\ca}{12(t+\cb d)}\right) \frac{\ca}{l+\cb d}\\&~~~~~~~~~~\times\left[f_l(\Bbeta_{l}) -f_l(\Bbeta^*)-\EE\left[f_l(\Bbeta_{l}) -f_l(\Bbeta^*)\big| \Bbeta_{l}\right] \right],
		\end{split}
		\label{eq1}
	\end{equation}
	also, the following upper bound of the product series holds
	\begin{align*}
		\left(1-\frac{\ca}{12(l+\cb d)}\right)\cdots \left(1-\frac{\ca}{12(t+\cb d)}\right)
		\leq\left(\frac{l+1+\cb d}{t+1+\cb d}\right)^{\frac{\ca}{12}}.
	\end{align*}
	\noindent\underline{\underline{Case One: large $\ca$}} If $\ca>12$, the analyses are exactly the ones in Theorem~\ref{thm:one_sample} and the upper bound is
	\begin{align*}
		\ltwo{\Bbeta_{t+1}-\Bbeta^*}^2\leq C^*\frac{\ku}{\kl^2}\frac{d}{t+1-t_1+\cb d} b_0^2.
	\end{align*}
	\noindent\underline{\underline{Case One: small $\ca$}} If $\ca<12$, the first term of equation~\eqref{eq1} is bounded with
	\begin{align*}
		\prod_{l=0}^{t}\left(1-\frac{\ca}{12(l+\cb d)}\right)\ltwo{\Bbeta_{0}-\Bbeta^*}^2\leq \left(\frac{1+\cb d}{t+1+\cb}\right)^{\frac{\ca}{12}}.
	\end{align*}
	In addition, the second term of equation~\eqref{eq1} has the following upper bound,
	\begin{align*}
		&~~~~b_0^2\bar{\tau}^2\frac{\ku}{\kl^2}d\sum_{l=0}^{t}\left(1-\frac{\ca}{12(l+1+\cb d)}\right)\cdots \left(1-\frac{\ca}{12(t+\cb d)}\right)\frac{\ca^2}{(l+\cb d)^2}\\
		&\leq b_0^2\bar{\tau}^2\frac{\ku}{\kl^2}d\sum_{l=0}^{t} \left(\frac{l+1+\cb d}{t+1+\cb d}\right)^{\frac{\ca}{12}}\frac{\ca^2}{(l+\cb d)^2}\\
		&\leq  b_0^2\bar{\tau}^2\frac{\ku}{\kl^2}d\frac{\ca^2}{(t+1+\cb d)^{\frac{\ca}{12}}}\sum_{l=0}^{t}\frac{1}{(l+1+\cb d)^{2-\frac{\ca}{12}}}\\
		&\leq 4 b_0^2\bar{\tau}^2\frac{\ku}{\kl^2}\frac{ d^{\frac{\ca}{12}}\ca^2}{(t+1+\cb d)^{\frac{\ca}{12}}},
	\end{align*}
	where the last line uses $\cb\geq 1$ and $$\sum_{l=0}^{t}\frac{1}{(l+1+\cb d)^{2-\frac{\ca}{12}}}\leq\int_{0}^{t+1}\frac{1}{(x+ \cb d)^{2-\frac{\ca}{12}} }\; dx\leq\frac{1}{1-\ca/12}\frac{1}{(\cb d)^{1-\frac{\ca}{12}}}.$$ It is worth noting that, under the event of $\cap_{l=0}^t \{\ltwo{\Bbeta_{l}-\Bbeta^*}^2\leq C^*(\frac{d}{l+\cb d})^{\frac{\ca}{12}} \ltwo{\Bbeta_{0}-\Bbeta^*}^2\}$, it has 
	\begin{align*}
		\|f_l(\Bbeta_{l}) -f_l(\Bbeta^*)-\EE\left[f_l(\Bbeta_{l}) -f_l(\Bbeta^*)\big| \Bbeta_{l}\right] \|_{\psi_2} \leq \bar{\tau}\sqrt{\ku C^*\left(\frac{d}{l+\cb d}\right)^{\frac{\ca}{12}} \ltwo{\Bbeta_{0}-\Bbeta^*}^2}.
	\end{align*}
	Thus according to Theorem 2 in \cite{shamir2011variant}, the upper bound of the last term for equation~\eqref{eq1} is obtained with probability exceeding $1-\exp(-cd)$,
	\begin{align*}
		&~~~\frac{b_0}{\kl} \sum_{l=0}^{t}\left(1-\frac{\ca}{12(l+1+\cb d)}\right)\cdots \left(1-\frac{\ca}{12(t+\cb d)}\right) \frac{\ca}{l+\cb d}\\
		&~~~~~~~~~~~~~~~~~~~~~~~~~~~~~~~~~~~~~~~~~~~~~~~~~~~~~~~~~~~~~~~~\times\left[f_l(\Bbeta_{l}) -f_l(\Bbeta^*)-\EE\left[f_l(\Bbeta_{l}) -f_l(\Bbeta^*)\big| \Bbeta_{l}\right] \right]\\
		&\leq \bar{\tau}\ca\frac{b_0}{\kl}\sqrt{C^*\ku d \sum_{l=0}^t \left(\frac{l+1+\cb d}{t+1+\cb d}\right)^{\frac{\ca}{6} } \frac{1}{(l+\cb d)^2} \left(\frac{d}{l+\cb d}\right)^{\frac{\ca}{12}} \ltwo{\Bbeta_{0}-\Bbeta^*}^2}\\
		&\leq \bar{\tau}\sqrt{C^*}\sqrt{\ku}\frac{1}{\kl}\frac{d^{\frac{\ca }{12}}}{(t+1+\cb d)^{\frac{\ca}{12}}}\cdot\ca b_0\ltwo{\Bbeta_{0}-\Bbeta^*}.
	\end{align*}
	Thus altogether, by having $\ca<\min\{12, (\kl/\sqrt{\ku})\ltwo{\Bbeta_{0}-\Bbeta^*}/b_0\}$, we obtain the following bound for equation~\eqref{eq1} with some sufficient large constant $C^*$, where $C^*$ may depend on the constants $\ca,\cb$,
	\begin{align*}
		\ltwo{\Bbeta_{t+1}-\Bbeta^*}^2\leq C^*\left(\frac{d}{t+1+\cb d}\right)^{\frac{\ca}{12}} \ltwo{\Bbeta_{0}-\Bbeta^*}^2.
	\end{align*}

	\subsubsection{Proof of Theorem~\ref{thm:one_sample_regret}}
	Before proving the regret bound, we first prove the following lemma, which establishes the expected error rate dynamics. It is worth noting that it is not a consequence of the empirical dynamics Theorem~\ref{thm:one_sample}.
	\begin{lemma}
		Assume the same conditions and stepsizes as Theorem~\ref{thm:one_sample}. Then in the first phase we have $\sum_{t\leq t_1}\EE\ltwo{\Bbeta_{t}-\Bbeta^*}\leq C\max\{\tau^2,(1-\tau)^2\}d\frac{\ku}{\kl}D_0+Ct_1\exp(-cd)D_0$ and in the second phase we have $\EE\ltwo{\Bbeta_{t+1}-\Bbeta^*}^2\leq C\frac{\ku}{\kl^2}\frac{d}{t+1-t_1+\cb d} b_0^2$.
		\label{lem:one_sampe_expectation}
	\end{lemma}
	\begin{proof}
		Note that the expectation of $\EE\left[\ltwo{\Bbeta_{t+1}-\Bbeta^*}^2\right]$ is
		\begin{align*}
			\EE\left[\ltwo{\Bbeta_{t+1}-\Bbeta^*}^2\right]&=\EE\left[\EE\left[\ltwo{\Bbeta_{t+1}-\Bbeta^*}^2\big| \Bbeta_{t}\right]\right]\\
			&=\EE\left[\EE\left[\ltwo{\Bbeta_{t}-\Bbeta^*}^2-2\eta_t\inp{\Bbeta_{t}-\Bbeta^*}{\g_t}+\eta_t^2\ltwo{\g_t}^2\big| \Bbeta_{t}\right]\right].
		\end{align*}
		First, we would consider the inner conditional expectation. According to the sub-gradient definition, and the independence between $\eta_t$ and $(\X_{t},Y_{t})$, we have
		\begin{multline*}
			\EE\left[\ltwo{\Bbeta_{t}-\Bbeta^*}^2-2\eta_t\inp{\Bbeta_{t}-\Bbeta^*}{\g_t}+\eta_t^2\ltwo{\g_t}^2\big| \Bbeta_{t}\right]\\\leq\ltwo{\Bbeta_{t}-\Bbeta^*}^2-2\eta_t\EE\left[f_t(\Bbeta_t)-f_t(\Bbeta^*)\big|\Bbeta_{t}\right]+\eta_t^2 \EE\left[\ltwo{\g_t}^2\big| \Bbeta_{t}\right].
		\end{multline*}
		\medskip
		\noindent
		{\sc First Phase Analysis}.	We inherit the notations in Section~\ref{sec:proof_online}. Denote the event $\calE_{t}:=\{\|\Bbeta_t-\Bbeta^*\|_2\leq CD_t\}$ and Theorem~\ref{thm:one_sample} shows $\PP(\cup_{t=0,\dots,t_1}\calE_t^{{\rm c}})\leq t_1\exp(-c_0d)$. Then we have 
		\begin{equation}
			\begin{split}
				\EE\|\Bbeta_{t+1}-\Bbeta^*\|_2&=\EE\left\{\|\Bbeta_{t+1}-\Bbeta^*\|_2\cdot\II\{\calE_{t+1}\}\right\}+\EE\left\{\|\Bbeta_{t+1}-\Bbeta^*\|_2\cdot\II\{\calE_{t+1}^{{\rm c}}\}\right\}\\
				&\leq CD_{t+1}+\PP\left( \calE_{t+1}^{{\rm c}}\right)\cdot\left(\EE\|\Bbeta_{t+1}-\Bbeta^*\|_2^2\right)^{\frac{1}{2}}.
			\end{split}
			\label{eq:1.1.2}
		\end{equation}
		Hence we only need to provide a rough upper bound for $\EE\|\Bbeta_t-\Bbeta^*\|^2$. Note that $\EE[f_t(\Bbeta_t)-f_t(\Bbeta^*)\big|\Bbeta_{t}]\geq 0$ and therefore, we have
		\begin{align*}
			\EE\left[\ltwo{\Bbeta_{t+1}-\Bbeta^*}^2\right]&\leq \EE\left[\|\Bbeta_t-\Bbeta^*\|^2\right]+\eta_t^2\cdot\max\{\tau,1-\tau\}^2\ku d\\
			&\leq\EE\left[\|\Bbeta_t-\Bbeta^*\|^2\right]+\frac{C}{\max\{\tau,1-\tau\}^2}\frac{\kl}{\ku} \frac{1}{d}D_t^2,
		\end{align*}
		which implies $$\EE\left[\ltwo{\Bbeta_{t+1}-\Bbeta^*}^2\right]\leq \|\Bbeta_0-\Bbeta^*\|_2^2+ \frac{C}{\max\{\tau,1-\tau\}^2}\frac{\kl}{\ku} \frac{1}{d}\sum_{l=0}^tD_l^2\leq CD_0^2.$$ Inserting the inequality above into equation~\eqref{eq:1.1.2}, we obtain 
		\begin{align*}
			\EE\|\Bbeta_{t+1}-\Bbeta^*\|_2\leq CD_{t+1}+ \PP\left( \calE_{t+1}^{{\rm c}}\right)D_0.
		\end{align*}
		Thus we have
		\begin{align*}
			\sum_{t=0}^{t_1}\EE\|\Bbeta_{t+1}-\Bbeta^*\|_2&\leq \sum_{t=0}^{t_1}D_t+C\PP(\cup_{t=0,\dots,t_1}\calE_t^{\rm c})D_0\\
			&\leq C\max\{\tau^2,(1-\tau)^2\}d\frac{\ku}{\kl}D_0+Ct_1\exp(-cd)D_0.
		\end{align*}
		
		\medskip
		\noindent
		{\sc Second Phase Analysis}. 
		Lemma~\ref{teclem:loss_expectation} shows that in the second phase, the lower bound of the expected excess risk is a second-order one, $\EE[f_t(\Bbeta_{t}) -f_t(\Bbeta^*)|\Bbeta_{t}]\geq \frac{1}{12b_0}\kl\ltwo{\Bbeta_{t}-\Bbeta^*}^2$. Thus we have
		\begin{align*}
			\EE\left[\ltwo{\Bbeta_{t+1}-\Bbeta^*}^2\big| \Bbeta_{t}\right]&\leq \ltwo{\Bbeta_{t}-\Bbeta^*}^2-\frac{1}{6b_0}\eta_t\kl\ltwo{\Bbeta_{t}-\Bbeta^*}^2+\eta_t^2\ku\max\{\tau^2,(1-\tau)^2\}d.
		\end{align*}
		By specifying the stepsize $\eta_t=\frac{\ca}{t-t_1+\cb d}\frac{b_0}{\kl}$, we obtain
		\begin{align*}
			\EE\left[\ltwo{\Bbeta_{t+1}-\Bbeta^*}^2\big| \Bbeta_{t}\right]&\leq\left(1-\frac{\ca}{6(t-t_1+\cb d)}\right)\ltwo{\Bbeta_{t}-\Bbeta^*}^2+\frac{\ku}{\kl^2}\frac{\ca^2}{(t-t_1+\cb d)^2}db_0^2.
		\end{align*}
		Taking the expectation over $\Bbeta_{t}$ on both sides of the equation and substituting the upper bound at time $t$, $\EE[\ltwo{\Bbeta_{t}-\Bbeta^*}^2]\leq C\frac{\ku}{\kl^2}\frac{\ca d}{t-t_1+\cb d}b_0^2$, into the above equation, we obtain
		\begin{align*}
			&~~~~~\EE\left[\ltwo{\Bbeta_{t+1}-\Bbeta^*}^2\right]\\
			&\leq C\left(1-\frac{\ca}{6(t-t_1+\cb d)}\right)\frac{\ku}{\kl^2}\frac{\ca d}{t-t_1+\cb d}b_0^2+\frac{\ku}{\kl^2}\frac{\ca^2}{(t-t_1+\cb d)^2}db_0^2\\
			&=C\left(1-\frac{\ca/2}{6(t-t_1+\cb d)}\right)\frac{\ku}{\kl^2}\frac{\ca d}{t-t_1+\cb d}b_0^2.
		\end{align*}
		Note that $1-\frac{\ca/2}{6(t-t_1+\cb d)}\leq1-\frac{1}{t-t_1+\cb d}\leq \frac{t-t_1+\cb d}{t+1-t_1+\cb d}$, and hence we finally obtain $\EE[\ltwo{\Bbeta_{t+1}-\Bbeta^*}^2 ]\leq C\frac{\ku}{\kl^2}\frac{\ca d}{t+1-t_1+\cb d}b_0^2$, which completes the proof for the second phase.
	\end{proof}

	We are now ready to establish the regret bound stated in Theorem~\ref{thm:one_sample_regret}.
	By the expected excess risk bound in Lemma~\ref{teclem:loss_expectation}, we have
	$\Bbeta^*=\arg\max_{\Bbeta}\EE\sum_{t=0}^{T} f_{t}(\Bbeta_{t})-f_t(\Bbeta)$.  In the first phase, it has
	\begin{align*}
		& \sum_{t=0}^{t_1}\EE\left[f_t(\Bbeta_{t})-f_t(\Bbeta^*)\right] \\
		&=\sum_{t=0}^{t_1}\EE\left[\EE\left[f_t(\Bbeta_{t})-f_t(\Bbeta^*)\big|\Bbeta_t\right]\right]\leq\max\{\tau,1-\tau\}\sqrt{\ku}\sum_{t=0}^{t_1}\EE\ltwo{\Bbeta_{t}-\Bbeta^*}\\
		&\leq C\max\{\tau^3,(1-\tau)^3\}\sqrt{\ku}d\frac{\ku}{\kl}D_0+\max\{\tau,1-\tau\}\sqrt{\ku}t_1\exp(-cd) ,
	\end{align*}
	where we use the Lipschitz continuity of $f_t(\cdot)$, and the final inequality follows from Lemma~\ref{lem:one_sampe_expectation}. For the second phase ($t>t_1$), we have
	\begin{align*}
		\EE\left[f_t(\Bbeta_{t})-f_t(\Bbeta^*)\right]&=\EE\left[\EE\left[f_t(\Bbeta_{t})-f_t(\Bbeta^*)\big|\Bbeta_t\right]\right]\leq\frac{\ku}{b_1}\EE\ltwo{\Bbeta_{t}-\Bbeta^*}^2\\
		&\leq 2\frac{\ku^2}{\kl^2}\frac{d}{t-t_1+\cb d}b_0^2,
	\end{align*}
	where the first inequality follows from Lemma~\ref{teclem:loss_expectation}, and the second from Lemma~\ref{lem:one_sampe_expectation}. The proof is thus complete.

	\subsection{Proof of Batch Learning}

	In this subsection, we develop the batch learning framework in a sequential manner. We begin with the proof of Theorem~\ref{thm:Batch}, proceed to analyze convergence under well-controlled initialization, and conclude with the proof of Theorem~\ref{thm:Batch regret}.

	\subsubsection{Proof of Theorem~\ref{thm:Batch}}
	
	The update rule $\Bbeta_{t+1}=\Bbeta_{t}-\eta_t\g_t$ ensures that 
	\begin{align}
		\ltwo{\Bbeta_{t+1}-\Bbeta^*}^2
		&= \ltwo{\Bbeta_t-\Bbeta^*}^2-2\eta_t\inp{\Bbeta_t-\Bbeta^*}{\g_t}+\eta_t^2\ltwo{\g_t}^2.
		\label{eq:batch_update}
	\end{align}
	The behavior of the inner product term $\inp{\Bbeta_t-\Bbeta^*}{\g_t}$ and the gradient norm $\ltwo{\g_t}^2$ varies significantly depending on the proximity of the current iterate $\Bbeta_t$ to the true parameter $\Bbeta^*$. Therefore, we analyze these quantities by distinguishing between regimes where $\Bbeta_t$ is either close to or far from $\Bbeta^*$. 
	
	\medskip
	\noindent
	{\sc First Phase Analysis.}
	We shall prove the convergence dynamics by induction. Initially, it holds obviously for $\Bbeta_0$. We continue to assume for the iteration of $t$, it has $\ltwo{\Bbeta_{t}-\Bbeta^*}\leq (1-\frac{1}{100}\frac{\kl}{\ku})^{t}\cdot D_0$ and we are going to prove the convergence dynamics at $t+1$. For convenience, we denote $$D_t:=(1-\frac{1}{100}\frac{\kl}{\ku})^{t}\cdot D_0, \quad D_{t+1}:=(1-\frac{1}{100}\frac{\kl}{\ku})^{t+1}\cdot D_0.$$ We proceed to assume the event
	\begin{multline*}
		\bcalE_t:=\left\{\sup_{\Bbeta_1,\Bbeta_2\in\RR^d} \left|f_t(\Bbeta_1)-f_t(\Bbeta_2)-\EE\left[ f_t(\Bbeta_1)-f_t(\Bbeta_2)\right]\right| \right.\\\left. \cdot\ltwo{\Bbeta_1-\Bbeta_2}^{-1}\leq C_1\max\{\tau,1-\tau\}\sqrt{\frac{\ku d}{n_t}}\right\}
	\end{multline*}
	holds. Specifically, Proposition~\ref{tecprop:empirical} establishes that $\PP(\bcalE_t)\geq1-\exp(-C_2d)-\exp(-\sqrt{n_t/\log n_t})$. Under the event $\bcalE_t$, the intermediate term in \eqref{eq:batch_update} can be lower bounded as
	$$
	\inp{\Bbeta_t-\Bbeta^*}{\g_t}\geq f_t(\Bbeta_t)-f_t(\Bbeta^*)\geq \EE\left[f_t(\Bbeta_t)-f_t(\Bbeta^*)\big|\Bbeta_t \right]-C_1\max\{\tau,1-\tau\}\sqrt{\frac{\ku d}{n_t}}\ltwo{\Bbeta_t-\Bbeta^*},
	$$
	where the first inequality follows from the sub-gradient definition, and the second holds on the event $\bcalE_t$. Additionally, the convexity of the quantile loss, combined with  Lemma~\ref{teclem:gaussian_expecation}, yields
	\begin{align*}
		\EE\left[f_t(\Bbeta_t)-f_t(\Bbeta^*)\big|\Bbeta_t \right]&\geq \frac{1}{n_t}\sum_{i=1}^{n_t}\EE\left[\rho_{Q,\tau}(\inp{\X_i^{(t)}}{\Bbeta^*-\Bbeta_t})-\rho_{Q,\tau}(-\xi_i)-\rho_{Q,\tau}(\xi_i)\big| \Bbeta_t\right]\\
		&\geq \frac{1}{\sqrt{2\pi}} \sqrt{\kl}\ltwo{\Bbeta_t-\Bbeta^*}-\gamma.
	\end{align*}
	Thus, if the batch size satisfies $n_t\geq C(\ku/\kl)\max\{\tau^2,(1-\tau)^2\} d$, and the iterate lies in the outer phase region $\ltwo{\Bbeta_t-\Bbeta^*}\geq8\sqrt{\kl^{-1}}\gamma$, it holds that
	\begin{align*}
		\inp{\Bbeta_t-\Bbeta^*}{\g_t}\geq \frac{1}{6} \sqrt{\kl}\ltwo{\Bbeta_t-\Bbeta^*}.
	\end{align*}
	Moreover, Lemma~\ref{teclem:sub-gradient} shows that under the event $\bcalE_t$, $\ltwo{\g_t}\leq\sqrt{\ku}$. Combining these, the update equation~\eqref{eq:batch_update} can be upper bounded as
	\begin{align*}
		\ltwo{\Bbeta_{t+1}-\Bbeta^*}^2&\leq \ltwo{\Bbeta_{t}-\Bbeta^*}^2- \frac{1}{3}\eta_t\sqrt{\kl}\ltwo{\Bbeta_t-\Bbeta^*}+ \eta_t^2\ku.
	\end{align*}
	We now treat the right-hand side of the above inequality as a quadratic function in $\ltwo{\Bbeta_t-\Bbeta^*}$. Using the inductive assumption $\ltwo{\Bbeta_t-\Bbeta^*}\leq D_{t}$, we obtain
	\begin{align*}
		\ltwo{\Bbeta_{t+1}-\Bbeta^*}^2&\leq D_t^2- \frac{1}{3}\eta_t\sqrt{\kl}D_t+ \eta_t^2\ku.
	\end{align*}
	The stepsize $\eta_t=(1-\frac{1}{100}\frac{\kl}{\ku})^{t}\eta_0\in \frac{\sqrt{\kl}}{\ku}D_t\cdot [1/8,5/24]$ implies that
	\begin{align*}
		\ltwo{\Bbeta_{t+1}-\Bbeta^*}^2\leq \left(1-\frac{1}{50}\frac{\kl}{\ku}\right)D_t^2,
	\end{align*}
	which leads to the inductive conclusion $\ltwo{\Bbeta_{t+1}-\Bbeta^*}\leq (1-\frac{1}{100}\frac{\kl}{\ku})^{t+1}D_0=D_{t+1}$.

	\medskip
	\noindent
	{\sc Second Phase Analysis}.
	In the second phase, we shall still assume event $\bcalE_t$ holds, which is defined and discussed in the first phase analyses. Note that according to Lemma~\ref{teclem:loss_expectation}, in this case, where $\ltwo{\Bbeta_t-\Bbeta^*}\leq 8\sqrt{\kl^{-1}}\gamma$, the expectation of excess risk has a quadratic lower bound,
	\begin{align*}
		\EE\left[f_t(\Bbeta_t)-f_t(\Bbeta^*)\big|\Bbeta_t \right]
		&\geq \frac{1}{12b_0}\kl\ltwo{\Bbeta_t-\Bbeta^*}^2.
	\end{align*}
	The second phase region has the constraints $\ltwo{\Bbeta_{t}-\Bbeta^*}\geq C\max\{\tau,1-\tau\} \sqrt{(\ku/\kl^2)d/n}b_0$. Then under event $\bcalE_t$, we have
	\begin{align*}
		f_t(\Bbeta_t)-f_t(\Bbeta^*)&\geq \frac{1}{12b_0}\kl\ltwo{\Bbeta_t-\Bbeta^*}^2-C_1\max\{\tau,1-\tau\}\sqrt{\ku d/n_t}\ltwo{\Bbeta_t-\Bbeta^*}\\
		&\geq \frac{1}{24b_0}\kl\ltwo{\Bbeta_t-\Bbeta^*}^2.
	\end{align*}
	In addition, Lemma~\ref{teclem:sub-gradient} proves that under the event $\bcalE_t$, the following holds in the second phase region:
	$$\ltwo{\g_t}^2\leq 4\frac{1}{b_1^2}\ku^2\ltwo{\Bbeta-\Bbeta^*}^2.$$
	Then together with the sub-gradient definition $\inp{\g_t}{\Bbeta_t-\Bbeta^*}\geq f_t(\Bbeta_t)-f_t(\Bbeta^*)$, equation~\eqref{eq:batch_update} can be upper bounded as
	\begin{align*}
		\ltwo{\Bbeta_{t+1}-\Bbeta^*}^2&\leq\ltwo{\Bbeta_{t}-\Bbeta^*}^2-\eta_t\frac{1}{12b_0}\kl\ltwo{\Bbeta_{t}-\Bbeta^*}^2+\eta_t^2 \frac{4}{b_1^2}\ku^2\ltwo{\Bbeta-\Bbeta^*}^2.
	\end{align*}
	By having stepsize $\eta_t\in \frac{\kl}{\ku^2}\frac{b_1^2}{b_0}\cdot [c_1,c_2]$, we obtain
	$$\ltwo{\Bbeta_{t+1}-\Bbeta^*}^2\leq\left(1-0.0005\frac{b_1^2}{b_0^2}\frac{\kl^2}{\ku^2}\right)\ltwo{\Bbeta_{t}-\Bbeta^*}^2.$$
	
	\medskip
	\noindent
	{\sc Third Phase Analysis}. We shall prove the convergence dynamics by induction. The proof at $t_2$ is trivial, which can be obtained similarly to the following analyses, and hence it is skipped. Then we are going to prove $\ltwo{\Bbeta_{t+1}-\Bbeta^*}$ when the desired inequality holds for $t_2,\dots,t$. It is worth noting that Lemma~\ref{teclem:sub-gradient} and Proposition~\ref{tecprop:empirical} show: in the region of $\ltwo{\Bbeta_t-\Bbeta^*}\leq C_1\max\{\tau,1-\tau\}\sqrt{(\ku/\kl^2)d/n}\cdot b_0$,  with probability exceeding $1-\exp(-C_1d)-\exp(-\sqrt{n_t/\log n_t})$, it has $$\ltwo{\g_t}^2\leq C_3^2\bar{\tau}^2\ku\frac{d}{n}\frac{\ku^2b_0^2}{\kl^2b_1^2},$$
	where we denote $\bar{\tau}:=\max\{\tau,1-\tau\}$, for convenience. According to the loss function expectation calculations in Lemma~\ref{teclem:loss_expectation}, we have
	\begin{align*}
		\EE\left[f_t(\Bbeta_{t})-f_t(\Bbeta^*)\big| \Bbeta_{t}\right]&\geq \frac{1}{24b_0}\ltwo{\Bbeta_{t}-\Bbeta^*}^2.
	\end{align*}
	Then together with the sub-gradient definition, we have
	\begin{align*}
		\ltwo{\Bbeta_{t+1}-\Bbeta^*}^2&\leq \ltwo{\Bbeta_{t}-\Bbeta^*}^2-\eta_t\frac{1}{12b_0}\kl\ltwo{\Bbeta_{t}-\Bbeta^*}^2+C_3^2\eta_t^2\bar{\tau}^2\ku\frac{d}{n}\frac{\ku^2b_0^2}{\kl^2b_1^2}\\
		&~~~~~~~~~~~~~~~~~~~+2\eta_t\left[f_t(\Bbeta_t)-f_t(\Bbeta^*)-\EE\left[f_t(\Bbeta_t)-f_t(\Bbeta^*)\big|\Bbeta_{t} \right]\right].
	\end{align*}
	Insert the stepsize $\eta_t=\frac{\ca}{t-t_2+\cb}\frac{b_0}{\kl}$  into the above equation and then it arrives at
	\begin{align*}
		\ltwo{\Bbeta_{t+1}-\Bbeta^*}^2&\leq \left(1-\frac{\ca}{12(t-t_2+\cb)}\right)\ltwo{\Bbeta_{t}-\Bbeta^*}^2+C_3^2\bar{\tau}^2\frac{\ca^2}{(t-t_2+\cb)^2}\frac{\ku}{\kl^2}\frac{d}{n}\frac{\ku^2b_0^2}{\kl^2b_1^2}b_0^2\\
		&~~~~~~~~~~~~~~~~~~~+2\frac{b_0}{\kl}\frac{\ca}{t-t_2+\cb}\left[f_t(\Bbeta_t)-f_t(\Bbeta^*)-\EE\left[f_t(\Bbeta_t)-f_t(\Bbeta^*)\big|\Bbeta_{t} \right]\right].
	\end{align*}
	Applying the above bound repeatedly from $t_2$, we have
	\begin{equation}
		\begin{split}
			&~~~~\ltwo{\Bbeta_{t+1}-\Bbeta^*}^2\\
			&\leq\prod_{l=t_2}^{t}\left(1-\frac{\ca}{12(l-t_2+\cb)}\right)\ltwo{\Bbeta_{t_2}-\Bbeta^*}^2\\
			&+C_3^2\bar{\tau}^2\frac{\ku}{\kl^2}\frac{\ku^2b_0^2}{\kl^2b_1^2}\frac{d}{n}b_0^2\sum_{l=t_2}^{t}\left(1-\frac{\ca}{12(l+1-t_2+\cb)}\right)\cdots\left(1-\frac{\ca}{12(t-t_2+\cb)}\right) \frac{\ca^2}{(l-t_2+\cb)^2}\\
			&+\frac{b_0}{\kl}\sum_{l=t_2}^{t}\left(1-\frac{\ca}{12(l+1-t_2+\cb)}\right)\cdots\left(1-\frac{\ca}{12(t-t_2+\cb)}\right)\frac{\ca}{l-t_2+\cb}\\&~~~~~~~~~~~~~~~~~~~~~~~~~~ \times\left[f_l(\Bbeta_l)-f_l(\Bbeta^*)-\EE\left[f_l(\Bbeta_l)-f_l(\Bbeta^*)\big|\Bbeta_{l} \right]\right].
		\end{split}
		\label{eq3}
	\end{equation}
	A sharp analysis of equation~\eqref{eq3} is key to the proof. Notice that with $\ca\geq 12$, we have $1-\frac{\ca}{12(l-t_2+\cb)}\leq1-\frac{1}{l-t_2+\cb}\leq \frac{l-t_2+\cb}{l+1-t_2+\cb}$. In this way, the first term on the right hand side could be bounded with
	\begin{align*}
		\prod_{l=t_2}^{t}\left(1-\frac{\ca}{12(l-t_2+\cb)}\right)\ltwo{\Bbeta_{t_2}-\Bbeta^*}^2\leq\frac{\cb}{t+1-t_2+\cb}\ltwo{\Bbeta_{t_2}-\Bbeta^*}^2.
	\end{align*}
	The product sequence satisfies, for each $l=t_2,\dots,t$, that
	\begin{equation}
		\begin{split}
			&~~~~\left(1-\frac{\ca}{12(l+1-t_2+\cb)}\right)\cdots\left(1-\frac{\ca}{12(t-t_2+\cb)}\right)\\
			&=\exp\left(\sum_{k=l+1}^{t}\log\left(1- \frac{\ca}{12(k-t_2+\cb)}\right)\right)\leq\exp\left(-\sum_{k=l+1}^{t} \frac{\ca}{12(k-t_2+\cb)}\right)\\
			&\leq\exp\left(- \frac{\ca}{12}\int_{l+1-t_2}^{t+1-t_2}\frac{1}{x+\cb}\; dx\right)= \exp\left(- \frac{\ca}{12} \log\left(\frac{t+1-t_2+\cb}{l+1-t_2+\cb}\right)\right)\\
			&=\left(\frac{l+1-t_2+\cb}{t+1-t_2+\cb}\right)^{\frac{\ca}{12}}.
		\end{split}
		\label{eq:2-1}
	\end{equation}
	The second term of equation~\eqref{eq3} thus has the following upper bound:
	\begin{align*}
		&~~~\sum_{l=t_2}^{t}\left(1-\frac{\ca}{12(l+1-t_2+\cb)}\right)\cdots\left(1-\frac{\ca}{12(t-t_2+\cb)}\right) \frac{\ca^2}{(l-t_2+\cb)^2}\\
		&\leq\sum_{l=t_2}^{t}\left(\frac{l+1-t_2+\cb}{t+1-t_2+\cb}\right)^{\frac{\ca}{12}}\frac{\ca^2}{(l-t_2+\cb)^2}\\
		&\leq\left(\frac{\cb+1}{\cb}\right)^{\frac{\ca}{12}}\frac{\ca^2}{\left(t+1-t_2+\cb \right)^{\frac{\ca}{12}}}\sum_{l=t_2}^{t} \left(l-t_2+\cb\right)^{\frac{\ca}{12}-2}.
	\end{align*}
	With $\ca>12$, we have $$\sum_{l=t_2}^{t} \left(l-t_2+\cb\right)^{\frac{\ca}{12}-2}\leq\int_{0}^{t-t_2+1}(x+\cb)^{\frac{\ca}{12}-2}\; dx= \frac{1}{\frac{\ca}{12}-1}(t-t_2+1+\cb)^{\frac{\ca}{12}-1}.$$
	Hence, we have the upper bound of the second term
	\begin{align*}
		&~~~~\sum_{l=t_2}^{t}\left(1-\frac{\ca}{12(l+1-t_2+\cb)}\right)\cdots\left(1-\frac{\ca}{12(t-t_2+\cb)}\right) \frac{\ca^2}{(l-t_2+\cb)^2}\\
		&\leq\frac{1}{\frac{\ca}{12}-1}\left(\frac{\cb+1}{\cb}\right)^{\frac{\ca}{12}} \frac{1}{t+1+\cb-t_2}.
	\end{align*}
	Then, it suffices to bound the last term of equation~\eqref{eq3}. It is worth noting that under the event of $\cap_{l=t_2}^{t}\left\{ \ltwo{\Bbeta_l-\Bbeta^*}^2\leq C^*\frac{\ku}{\kl^2}\frac{\ku^2b_0^2}{\kl^2b_1^2}\frac{1}{l-t_2+\cb}\frac{d}{n}b_0^2\right\}$, we have
	$$\|f_l(\Bbeta_l)-f_l(\Bbeta^*)-\EE\left[f_l(\Bbeta_l)-f_l(\Bbeta^*)\big|\Bbeta_{l} \right]\|_{\Psi_2}\leq\sqrt{C^*\frac{\ku^2}{\kl^2}\frac{\ku^2b_0^2}{\kl^2b_1^2}\frac{1}{l-t_2+\cb}\frac{d}{n^2} }b_0.$$
	Furthermore, invoking Azuma's sub-Gaussian inequality (e.g., Theorem 2 in \cite{shamir2011variant}) in conjunction with equation~\eqref{eq:2-1}, it holds with probability exceeding $1-\exp(-cd)$ that
	\begin{align*}
		&~~~~\frac{b_0}{\kl}\sum_{l=t_2}^{t}\left(1-\frac{\ca}{12(l+1-t_2+\cb)}\right)\cdots\left(1-\frac{\ca}{12(t-t_2+\cb)}\right)\frac{\ca}{l-t_2+\cb}\\
		&~~~~~~~~~~~~~~~~~~~~~~~~~~~~~~~~~~~~~~~~~~~~~~~~~~~~~~~~ \times\left[f_l(\Bbeta_l)-f_l(\Bbeta^*)-\EE\left[f_l(\Bbeta_l)-f_l(\Bbeta^*)\big|\Bbeta_{l} \right]\right]\\
		&\leq \frac{b_0}{\kl}\left(d\sum_{l=t_2}^{t}\prod_{k=l+1}^{t}\left(1-\frac{\ca}{12(k-t_2+\cb)}\right)^2\frac{\ca^2}{(l-t_2+\cb)^2}\right.\\
		&~~~~~~~~~~~~~~~~~~~~~~~~~~~~~~\left.\times \left\|f_l(\Bbeta_l)-f_l(\Bbeta^*)-\EE\left[f_l(\Bbeta_l)-f_l(\Bbeta^*)\big|\Bbeta_{l} \right]\right\|_{\Psi_2}^2\right)^{1/2}\\
		&\leq \bar{\tau}^2\frac{\ku}{\kl^2}\frac{\ku b_0}{\kl b_1}\frac{d}{n}b_0^2\frac{C\ca\sqrt{C^*}}{(t+1-t_2+\cb )^{\frac{ \ca }{12}}}\sqrt{ \sum_{l=t_2}^{t} (l+1-t_2+\cb)^{\frac{\ca}{6}-3}}\\
		&\leq \bar{\tau}^2\frac{\ku}{\kl^2}\frac{\ku b_0}{\kl b_1}\frac{d}{n}b_0^2\frac{C\ca\sqrt{C^*}}{t+1-t_2+\cb }.
	\end{align*}
	Therefore, combining the pieces, equation~\eqref{eq3} can be bounded from above as
	\begin{align*}
		\ltwo{\Bbeta_{t+1}-\Bbeta^*}^2\leq C^*\bar{\tau}^2\frac{\ku}{\kl^2}\frac{d}{n}\frac{\ku^2b_0^2}{\kl^2b_1^2}\frac{1}{t+1-t_2+\cb }b_0^2,
	\end{align*}
	where $C^*>\ca^2$ is some sufficiently large constant.

	\subsubsection{Analysis of Batch Learning with Well-controlled Initial Errors}

	We will elucidate the convergence dynamics under the condition of well-controlled initial error rates, specifically when $\ltwo{\Bbeta_{0}-\Bbeta^*}^2\leq C_1^2\max\{\tau^2 ,(1-\tau)^2\}(\ku/\kl^2)(d/n)\cdot b_0^2$. Employing a similar analysis as applied in the third phase of Theorem~\ref{thm:Batch}, we obtain
	\begin{align*}
		\ltwo{\Bbeta_{t+1}-\Bbeta^*}^2&\leq \ltwo{\Bbeta_{t}-\Bbeta^*}^2-\eta_t\frac{1}{12b_0}\kl\ltwo{\Bbeta_{t}-\Bbeta^*}^2+C_3^2\eta_t^2\bar{\tau}^2\ku\frac{d}{n}\frac{\ku^2b_0^2}{\kl^2b_1^2}\\
		&~~~~~~~~~~~~~~~~~~~-2\eta_t\left[f_t(\Bbeta_t)-f_t(\Bbeta^*)-\EE\left[f_t(\Bbeta_t)-f_t(\Bbeta^*)\big|\Bbeta_{t} \right]\right].
	\end{align*}
	Inserting the stepsize $\eta_t=\frac{\ca}{t+\cb}\frac{b_0}{\kl}$ into the above equation results in
	\begin{align*}
		\ltwo{\Bbeta_{t+1}-\Bbeta^*}^2&\leq\left(1-\frac{\ca}{12(t+\cb )}\right)\ltwo{\Bbeta_{t}-\Bbeta^*}^2+C_3^2\bar{\tau}^2\frac{\ku}{\kl^2}\frac{\ca^2}{(t+\cb)^2}\frac{d}{n}b_0^2\\
		&~~~~~~~~~~~~~~~-2\frac{\ca}{t+\cb }\frac{b_0}{\kl} \left[f_t(\Bbeta_t)-f_t(\Bbeta^*)-\EE\left[f_t(\Bbeta_t)-f_t(\Bbeta^*)\big|\Bbeta_{t} \right]\right].
	\end{align*}
	Moreover, the preceding equation implies
	\begin{align*}
		& \ltwo{\Bbeta_{t+1}-\Bbeta^*}^2 \\
		&\leq\prod_{l=0}^{t}\left(1-\frac{\ca}{12(l+\cb )}\right)\ltwo{\Bbeta_{0}-\Bbeta^*}^2\\
		&~~~+C_3^2\bar{\tau}^2\frac{\ku^2b_0^2}{\kl^2b_1^2}\frac{\ku}{\kl^2}\sum_{l=0}^t \left(1-\frac{\ca}{12(l+1+\cb )}\right)\cdots\left(1-\frac{\ca}{12(t+\cb )}\right)\frac{\ca^2}{(l+\cb )^2}\frac{d}{n}b_0^2\\
		&~~~-\frac{b_0}{\kl}\sum_{l=0}^t \left(1-\frac{\ca}{12(l+1+\cb )}\right)\cdots\left(1-\frac{\ca}{12(t+\cb )}\right)\frac{\ca}{l+\cb }\\
		&~~~~~~~~~~~~~~~~~~~~~~~~~~~\times\left[f_l(\Bbeta_l)-f_l(\Bbeta^*)-\EE\left[f_l(\Bbeta_l)-f_t(\Bbeta^*)\big|\Bbeta_{l} \right]\right].
	\end{align*}
	In what follows, we are going to upper bound each term of the right hand side equation.
	
	\underline{\underline{Case One: large $\ca$.}} With $\ca>12$, the first term on the right-hand side can be bounded as
	\begin{multline*}
		\prod_{l=0}^{t}\left(1-\frac{\ca}{12(l+\cb )}\right)\ltwo{\Bbeta_{0}-\Bbeta^*}^2\\\leq\prod_{l=0}^{t}\left(1-\frac{1}{l+\cb }\right)\ltwo{\Bbeta_{0}-\Bbeta^*}^2\leq\frac{\cb}{t+1+\cb}\ltwo{\Bbeta_{0}-\Bbeta^*}^2.
	\end{multline*}
	The product sequence is bounded by
	\begin{multline*}
		\left(1-\frac{\ca}{12(l+1+\cb )}\right)\cdots\left(1-\frac{\ca}{12(t+\cb )}\right)\leq \exp\left(-\frac{\ca}{12}\sum_{k=l+1}^{t} \frac{1}{k+\cb}\right)\\
		\leq\exp\left(-\frac{\ca}{12}\int_{l+1}^{t+1}\frac{1}{x+\cb}\; dx\right)=\left(\frac{l+1+\cb}{t+1+\cb}\right)^{\frac{\ca}{12}}.
	\end{multline*}
	Subsequently, the second term can be bounded as
	\begin{multline*}
		\sum_{l=0}^t \left(1-\frac{\ca}{12(l+1+\cb )}\right)\cdots\left(1-\frac{\ca}{12(t+\cb )}\right)\frac{\ca^2}{(l+\cb )^2}\\
		\leq\sum_{l=0}^t \left(\frac{l+\cb+1}{t+\cb+1}\right)^{\frac{\ca}{12}}\frac{\ca^2}{(l+\cb )^2}\leq \left(\frac{\cb+1}{\cb}\right)^{\frac{\ca}{12}}\frac{\ca^2}{t+\cb+1}.
	\end{multline*}
	Next, consider the last term. It is noteworthy that under $\cap_{l=0}^t\{\ltwo{\Bbeta_{l}-\Bbeta^*}^2\leq C^*\bar{\tau}^2\frac{\ku}{\kl^2}\frac{1}{l+\cb} \frac{d}{n} b_0^2\}$, we have
	\begin{align*}
		\|f_l(\Bbeta_l)-f_l(\Bbeta^*)-\EE\left[f_l(\Bbeta_l)-f_t(\Bbeta^*)\big|\Bbeta_{l} \right] \|_{\Psi_2}^2\leq C^*\bar{\tau}^4\frac{d}{n^2}\frac{\ku^2}{\kl^2}\frac{1}{l+\cb}b_0^2.
	\end{align*}
	Therefore, in accordance with Theorem 2 in \cite{shamir2011variant}, with probability at least $1-\exp(-cd)$, the last term satisfies
	\begin{align*}
		&~~~\frac{b_0}{\kl}\sum_{l=0}^t \left(1-\frac{\ca}{12(l+1+\cb )}\right)\cdots\left(1-\frac{\ca}{12(t+\cb )}\right)\frac{\ca}{l+\cb }\\
		&~~~~~~~~~~~~~~~~~~~~~~~~~\times\left[f_l(\Bbeta_l)-f_l(\Bbeta^*)-\EE\left[f_l(\Bbeta_l)-f_t(\Bbeta^*)\big|\Bbeta_{l} \right]\right]\\
		&\leq \ca\bar{\tau}^2\frac{\ku}{\kl^2}\frac{\ku b_0}{\kl b_1}\frac{d}{n}b_0^2\sqrt{\sum_{l=0}^{t} \left(\frac{l+1+\cb}{t+1+\cb}\right)^{\frac{\ca}{6}}\frac{C^*}{(l+\cb)^3}  }\\
		&\leq \ca\bar{\tau}^2\frac{\ku}{\kl^2}\frac{\ku b_0}{\kl b_1}\frac{\sqrt{C^*}}{t+1+\cb}\frac{d}{n}b_0^2.
	\end{align*}
	Putting together the pieces, we conclude that 
	$$
	\ltwo{\Bbeta_{t+1}-\Bbeta^*}^2\leq C^*\bar{\tau}^2\frac{\ku}{\kl^2}\frac{\ku^2b_0^2}{\kl^2b_1^2}\frac{1}{t+1+\cb} \frac{d}{n} b_0^2.
	$$
	\underline{\underline{Case Two: small $\ca$}} In this setting, with $\ca<12$, the second term satisfies
	\begin{multline*}
		\sum_{l=0}^t \left(1-\frac{\ca}{12(l+1+\cb )}\right)\cdots\left(1-\frac{\ca}{12(t+\cb )}\right)\frac{\ca^2}{(l+\cb )^2}\\
		\leq\sum_{l=0}^t \left(\frac{l+\cb+1}{t+\cb+1}\right)^{\frac{\ca}{12}}\frac{\ca^2}{(l+\cb )^2}\leq \frac{\ca^2}{(t+\cb+1)^{\frac{\ca}{12}}}.
	\end{multline*}
	Under the event  $\cap_{l=0}^t\{\ltwo{\Bbeta_{l}-\Bbeta^*}^2\leq \frac{c}{(l+\cb)^{\frac{\ca}{12}}} \ltwo{\Bbeta_{0}-\Bbeta^*}^2\}$, we have
	\begin{align*}
		\|f_l(\Bbeta_l)-f_l(\Bbeta^*)-\EE\left[f_l(\Bbeta_l)-f_t(\Bbeta^*)\big|\Bbeta_{l} \right] \|_{\Psi_2}^2\leq C^*\bar{\tau}\frac{\ku}{n}\frac{1}{(l+\cb)^{\frac{\ca}{12}}} \ltwo{\Bbeta_{0}-\Bbeta^*}^2 .
	\end{align*}
	Thus, with probability exceeding $1-\exp(-cd)$, it holds
	\begin{multline*}
		\frac{b_0}{\kl}\sum_{l=0}^t \left(1-\frac{\ca}{12(l+1+\cb )}\right)\cdots\left(1-\frac{\ca}{12(t+\cb )}\right)\frac{\ca}{l+\cb }\\\times\left[f_l(\Bbeta_l)-f_l(\Bbeta^*)-\EE\left[f_l(\Bbeta_l)-f_t(\Bbeta^*)\big|\Bbeta_{l} \right]\right]\\
		\leq \bar{\tau}\ca\frac{\sqrt{d}}{\sqrt{n}}\frac{\sqrt{\ku}}{\kl}b_0\cdot\ltwo{\Bbeta_{0}-\Bbeta^*}\cdot\sqrt{\sum_{l=0}^{t} \left(\frac{l+1+\cb}{t+1+\cb}\right)^{\frac{\ca}{6}}\frac{C^*}{(l+\cb)^{2+\frac{\ca}{12}}}  }\\
		\leq \bar{\tau}\ca\frac{\sqrt{d}}{\sqrt{n}}\frac{\sqrt{\ku}}{\kl}b_0\cdot\ltwo{\Bbeta_{0}-\Bbeta^*}\cdot \frac{\sqrt{C^*}}{(t+1+\cb)^{\frac{\ca}{12}}}.
	\end{multline*}
	To sum up, we establish the upper bound
	\begin{align*}
		\ltwo{\Bbeta_{t+1}-\Bbeta^*}^2&\leq\left(\frac{1+\cb}{t+1+\cb}\right)^{\frac{\ca}{12}}\ltwo{\Bbeta_{0}-\Bbeta^*}^2+C_3^2\bar{\tau}^2\frac{\ku}{\kl^2}\frac{\ku^2b_0^2}{\kl^2b_1^2}\frac{d}{n}b_0^2\cdot \frac{\ca^2}{(t+\cb+1)^{\frac{\ca}{12}}}\\
		&~~~~~~~~~~~~~~~~~~~~~~~~~~~~~~~~~~+\bar{\tau}\frac{\sqrt{d}}{\sqrt{n}}\frac{\sqrt{\ku}}{\kl}b_0\cdot\ltwo{\Bbeta_{0}-\Bbeta^*}\cdot \frac{\ca\sqrt{C^*}}{(t+1+\cb)^{\frac{\ca}{12}}}.
	\end{align*}
	Therefore, by choosing sufficiently small $\ca<(\kl/\ku)(b_1/b_0)(\kl/\sqrt{\ku})\sqrt{n/d}\ltwo{\Bbeta_{0}-\Bbeta^*}/\bar{\tau}b_0$ and $\ca<\cb\leq1$, we obtain
	\begin{align*}
		\ltwo{\Bbeta_{t+1}-\Bbeta^*}^2&\leq\frac{C^*}{(t+1+\cb)^{\frac{\ca}{12}}}\ltwo{\Bbeta_{0}-\Bbeta^*}^2.
	\end{align*}

	\subsubsection{Proof of Theorem~\ref{thm:Batch regret}}
	
	Prior to demonstrating Theorem~\ref{thm:Batch regret}, it is imperative to establish the ensuing lemma.
	\begin{lemma}
		Assume the same conditions and stepsize schemes as in Theorem~\ref{thm:Batch}. 
		In the second phase, $\EE\ltwo{\Bbeta_{t+1}-\Bbeta^*}^2\leq (1-c_1(\kl/\ku)/(b_1^2/b_0^2))^{t+1}\ltwo{\Bbeta_{t_1}-\Bbeta^*}^2$; and in phrase three, $\EE\ltwo{\Bbeta_{t+1}-\Bbeta^*}^2\leq C\frac{\ku}{\kl^2}\frac{1}{t+1-t_2+\cb} \frac{d}{n}b_0^2$.
		\label{lem:regret_expectation_batch}
	\end{lemma}

	\begin{proof}
		Note that the expectation can be expressed as
		\begin{align*}
			\EE\left[\ltwo{\Bbeta_{t+1}-\Bbeta^*}^2\right]&=\EE\left[\EE\left[\ltwo{\Bbeta_{t+1}-\Bbeta^*}^2\big| \Bbeta_{t}\right]\right]\\
			&=\EE\left[\EE\left[\ltwo{\Bbeta_{t}-\Bbeta^*}^2-2\eta_t\inp{\Bbeta_{t}-\Bbeta^*}{\g_t}+\eta_t^2\ltwo{\g_t}^2\big| \Bbeta_{t}\right]\right]
		\end{align*}
		We first consider the inside conditional expectation. According to the sub-gradient definition and considering that $\eta_t$ is independent of $(\X_{t},Y_{t})$, we have
		\begin{equation}
			\begin{split}
				&~~~\EE\left[\ltwo{\Bbeta_{t}-\Bbeta^*}^2-2\eta_t\inp{\Bbeta_{t}-\Bbeta^*}{\g_t}+\eta_t^2\ltwo{\g_t}^2\big| \Bbeta_{t}\right]\\
				&\leq\ltwo{\Bbeta_{t}-\Bbeta^*}^2-2\eta_t\EE\left[f_t(\Bbeta_t)-f_t(\Bbeta^*)\big|\Bbeta_{t}\right]+\eta_t^2 \EE\left[\ltwo{\g_t}^2\big| \Bbeta_{t}\right]
			\end{split}
			\label{eq:2-2}
		\end{equation}
		The sub-gradient value at $0$ won't affect the expectation calculations, so for convenience, we let $\partial \rho_{Q,\tau}(x)|_{x=0}=0$. Then the sub-gradient at $\Bbeta_{t}$ could be written as
		\begin{align*}
			n_t\g_t=\sum_{i=1}^{n_t}\X_i^{(t)}\cdot \left(\tau1_{Y_i^{(t)}<\inp{\X_i^{(t)}}{\Bbeta_{t}}} + (1-\tau)1_{Y_i^{(t)}>\inp{\X_i^{(t)}}{\Bbeta_{t}}}\right).
		\end{align*}
		Its length is bounded with
		\begin{multline*}
			n_t^2\ltwo{\g_t}^2=\sum_{i=1}^{n_t}\ltwo{\X_i^{(t)}}^2\cdot \left(\tau^21_{Y_i^{(t)}<\inp{\X_i^{(t)}}{\Bbeta_{t}}} + (1-\tau)^21_{Y_i^{(t)}>\inp{\X_i^{(t)}}{\Bbeta_{t}}}\right)\\
			+\sum_{i\neq j} \inp{\X_i^{(t)}}{\X_j^{(t)}}\cdot\left(\tau1_{Y_i^{(t)}<\inp{\X_i^{(t)}}{\Bbeta_{t}}} + (1-\tau)1_{Y_i^{(t)}>\inp{\X_i^{(t)}}{\Bbeta_{t}}}\right)\\\times \left(\tau1_{Y_j^{(t)}<\inp{\X_j^{(t)}}{\Bbeta_{t}}} + (1-\tau)1_{Y_j^{(t)}>\inp{\X_j^{(t)}}{\Bbeta_{t}}}\right).
		\end{multline*}
		Taking the $\Bbeta_t$ conditional expectation on each side of the above equation leads to
		\begin{align*}
			&n_t^2\EE\left[\ltwo{\g_t}^2\big| \Bbeta_{t}\right]\leq \max\{\tau^2,(1-\tau)^2\}\ku n_td+\frac{n_t(n_t-1)}{2}\EE\left[\inp{\X_i^{(t)}}{\X_j^{(t)}}\right.\\
			&~~~~~\left.\times\left(H_\xi(\inp{\X_i^{(t)}}{\Bbeta_{t}-\Bbeta^*})-H_\xi(0)\right) \cdot\left(H_\xi(\inp{\X_j^{(t)}}{\Bbeta_{t}-\Bbeta^*})-H_\xi(0)\right)\big|\Bbeta_{t}\right].
		\end{align*}
		Notice that $\X_i^{(t)}$ follows Assumption~\ref{assump:sensing_operators:vec} and then the transformed vector $\Z_i^{(t)}:=\bSigma^{-\frac{1}{2}}\X_i^{(t)}$ is isotropic. Thus we have
		\begin{multline*}
			E:=\EE\left[\inp{\X_i^{(t)}}{\X_j^{(t)}}\cdot\left(H_\xi(\inp{\X_i^{(t)}}{\Bbeta_{t}-\Bbeta^*})-H_\xi(0)\right) \cdot\left(H_\xi(\inp{\X_j^{(t)}}{\Bbeta_{t}-\Bbeta^*})-H_\xi(0)\right)\big|\Bbeta_{t}\right]\\
			=\EE\left[\inp{\bSigma^{\frac{1}{2}}\Z_i^{(t)}}{\bSigma^{\frac{1}{2}}\Z_j^{(t)}}\times\left(H_\xi(\inp{\Z_i^{(t)}}{\bSigma^{\frac{1}{2}}(\Bbeta_{t}-\Bbeta^*)})-H_\xi(0)\right) \right.\\ \left.\times\left(H_\xi(\inp{\Z_j^{(t)}}{\bSigma^{\frac{1}{2}}(\Bbeta_{t}-\Bbeta^*)})-H_\xi(0)\right)\big|\Bbeta_{t}\right].
		\end{multline*}
		There exists a set of unit length orthogonal vectors $\{\e_1:=\frac{\bSigma^{\frac{1}{2}}(\Bbeta_{t}-\Bbeta^*)}{\ltwo{ \bSigma^{\frac{1}{2}}(\Bbeta_{t}-\Bbeta^*)}},\e_2,\dots,\e_d\}$ such that any two of them satisfy $\e_i^{\top}\e_j=0$ with $i\neq j$. Then the vector $\Z_i$ can be decomposed into $d$ independent random vectors
		\begin{align*}
			\Z_i^{(t)}=\frac{\inp{\Z_i^{(t)}}{\bSigma^{\frac{1}{2}}(\Bbeta_{t}-\Bbeta^*)} }{\ltwo{\bSigma^{\frac{1}{2}}(\Bbeta_t-\Bbeta^*)}^2}\bSigma^{\frac{1}{2}}(\Bbeta_t-\Bbeta^*)+\inp{\Z_{i}^{(t)}}{\e_2}\e_2+\dots +\inp{\Z_{i}^{(t)}}{\e_d}\e_d.
		\end{align*}
		And $\Z_j^{(t)}$ could be written in a parallel way. It is worth noting that the term $H_\xi(\inp{\Z_i^{(t)}}{\bSigma^{\frac{1}{2}}(\Bbeta_{t}-\Bbeta^*)})-H_\xi(0)$ is independent of the last $d-1$ components of the decomposition and $\EE\inp{\Z_i^{(t)}}{\e_l}=\boldsymbol{0}$ holds with $l=2,\dots,d$. Thus we have 
		\begin{align*}    
			E
			&=\frac{\ku}{\ltwo{\bSigma^{\frac{1}{2}}(\Bbeta_{t}-\Bbeta^*)}^2}\EE\left[\inp{\Z_i^{(t)}}{\bSigma^{\frac{1}{2}}(\Bbeta_{t}-\Bbeta^*)}\inp{\Z_j^{(t)}}{\bSigma^{\frac{1}{2}}(\Bbeta_{t}-\Bbeta^*)}\right.\\
			&~~~~\left.\times\left(H_\xi(\inp{\Z_i^{(t)}}{\bSigma^{\frac{1}{2}}(\Bbeta_{t}-\Bbeta^*)})-H_\xi(0)\right) \cdot\left(H_\xi(\inp{\Z_j^{(t)}}{\bSigma^{\frac{1}{2}}(\Bbeta_{t}-\Bbeta^*)})-H_\xi(0)\right)\big|\Bbeta_{t}\right].
		\end{align*}
		There are two methods to upper bound the above equation, which are presented as follows.
		\begin{enumerate}[1.]
			\item Bound $H_\xi(\inp{\Z_i^{(t)}}{\bSigma^{\frac{1}{2}}(\Bbeta_{t}-\Bbeta^*)})-H_\xi(0)$ with constant $1$. Then we have
			\begin{align*}
				E&\leq \frac{\ku}{\ltwo{\bSigma^{\frac{1}{2}}(\Bbeta_{t}-\Bbeta^*)}^2}\EE\left[|\inp{\Z_i^{(t)}}{\bSigma^{\frac{1}{2}}(\Bbeta_{t}-\Bbeta^*)}||\inp{\Z_j^{(t)}}{\bSigma^{\frac{1}{2}}(\Bbeta_{t}-\Bbeta^*)}|\big|\Bbeta_{t}\right]\\
				&\leq \ku.
			\end{align*}
			In this way, we have
			\begin{align*}
				\EE\left[\ltwo{\g_t}^2\big|\Bbeta_{t}\right]\leq \max\{\tau^2,(1-\tau)^2\}\ku\frac{d}{n_t}+\ku.
			\end{align*}
			\item Bound $H_\xi(\inp{\Z_i^{(t)}}{\bSigma^{\frac{1}{2}}(\Bbeta_{t}-\Bbeta^*)})-H_\xi(0)$ using conditions in Assumption~\ref{assump:heavy-tailed}, which implies $H_\xi(\inp{\Z_i^{(t)}}{\bSigma^{\frac{1}{2}}(\Bbeta_{t}-\Bbeta^*)})-H_\xi(0)\leq \frac{1}{b_1}|\inp{\Z_{i}^{(t)}}{\bSigma^{\frac{1}{2}}(\Bbeta_{t}-\Bbeta^*)}|$.
			Then we have
			\begin{align*}
				E&\leq \frac{\ku}{b_1^2\ltwo{\bSigma^{\frac{1}{2}}(\Bbeta_{t}-\Bbeta^*)}^2}\EE\left[|\inp{\Z_i^{(t)}}{\bSigma^{\frac{1}{2}}(\Bbeta_{t}-\Bbeta^*)}|^2|\inp{\Z_j^{(t)}}{\bSigma^{\frac{1}{2}}(\Bbeta_{t}-\Bbeta^*)}|^2\big|\Bbeta_{t}\right]\\
				&\leq \frac{1}{b_1^2}\ku^2\ltwo{\Bbeta_{t}-\Bbeta^*}^2.
			\end{align*}
			This method leads to
			\begin{align*}
				\EE\left[\ltwo{\g_t}^2\big|\Bbeta_t\right]\leq \max\{\tau^2,(1-\tau)^2\}\ku\frac{d}{n_t}+\frac{\ku^2}{b_1^2}\ltwo{\Bbeta_{t}-\Bbeta^*}^2.
			\end{align*}
		\end{enumerate}
		Thus, altogether, we have
		\begin{align}
			\EE\left[\ltwo{\g_t}^2\big|\Bbeta_t\right]\leq \max\{\tau^2,(1-\tau)^2\}\ku\frac{d}{n_t}+\min\left\{\ku,\frac{\ku^2}{b_1^2}\ltwo{\Bbeta_{t}-\Bbeta^*}^2\right\}.
			\label{eq:2-3}
		\end{align}
		Then we are ready to analyze the convergence dynamics in each phase.
		
		\medskip
		\noindent
		
		\noindent
		{\sc Second Phase Analysis}. In this region, equation~\eqref{eq:2-3} yields $\EE\left[\ltwo{\g_t}^2\big|\Bbeta_t\right]\leq 2\frac{\ku^2}{b_1^2}\ltwo{\Bbeta_{t}-\Bbeta^*}^2$ and Lemma~\ref{teclem:loss_expectation} provides $\EE[f_t(\Bbeta_{t})-f_t(\Bbeta^*)|\Bbeta_{t}]\geq\frac{\kl}{b_0}\ltwo{\Bbeta_{t}-\Bbeta^*}^2$. Hence, equation~\eqref{eq:2-2} is upper bounded with
		\begin{align*}
			\EE\left[\ltwo{\Bbeta_{t+1}-\Bbeta^*}^2\big| \Bbeta_{t} \right] &\leq \ltwo{\Bbeta_{t}-\Bbeta^*}^2-\eta_t\frac{\kl}{12b_0}\ltwo{\Bbeta_t-\Bbeta^*}^2+\eta_t^2\frac{\ku}{b_1^2}\ltwo{\Bbeta_{t}-\Bbeta^*}^2\\
			&\leq\left(1-\frac{\kl}{\ku}\frac{b_1^2}{b_0^2}\right)\ltwo{\Bbeta_{t}-\Bbeta^*}^2.
		\end{align*}
		Take expectation over $\Bbeta_t$ and then we obtain
		\begin{align*}
			\EE\left[\ltwo{\Bbeta_{t+1}-\Bbeta^*}^2 \right]\leq\left(1-\frac{\kl}{\ku}\frac{b_1^2}{b_0^2}\right)\EE\left[\ltwo{\Bbeta_{t+1}-\Bbeta^*}^2 \right].
		\end{align*}
		
		\medskip
		\noindent
		{\sc Third Phase Analysis}. In this region, \eqref{eq:2-3} guarantees $\EE[\ltwo{\g_t}^2\big|\Bbeta_t]\leq 2(\ku^2/\kl^2)(b_0^2/b_1^2)\bar{\tau}^2\ku\frac{d}{n_t}$ and Lemma~\ref{teclem:loss_expectation} proves $\EE[f_t(\Bbeta_{t})-f_t(\Bbeta^*)|\Bbeta_{t}]\geq\frac{\kl}{b_0}\ltwo{\Bbeta_{t}-\Bbeta^*}^2$. Hence, equation~\eqref{eq:2-2} is upper bounded with
		\begin{align*}
			&~~~\EE\left[\ltwo{\Bbeta_{t+1}-\Bbeta^*}^2\big| \Bbeta_{t} \right] \\
			&\leq \ltwo{\Bbeta_{t}-\Bbeta^*}^2-\eta_t\frac{\kl}{12b_0}\ltwo{\Bbeta_t-\Bbeta^*}^2+\eta_t^2\frac{\ku^2 b_0^2}{\kl^2 b_1^2}\max\{\tau^2,(1-\tau)^2\}\ku\frac{d}{n}\\
			&\leq\left(1-\frac{\ca}{12(t-t_2+\cb)}\right)\ltwo{\Bbeta_{t}-\Bbeta^*}^2+\frac{\ca}{(t-t_2+\cb)^2}\frac{\ku^2 b_0^2}{\kl^2 b_1^2}\max\{\tau^2,(1-\tau)^2\}\frac{\ku}{\kl^2}\frac{d}{n}b_0^2.
		\end{align*}
		Take expectation over $\Bbeta_t$ and then we have
		\begin{multline*}
			\EE\left[\ltwo{\Bbeta_{t+1}-\Bbeta^*}^2 \right]\\\leq\left(1-\frac{\ca}{12(t-t_2+\cb)}\right)\EE\left[\ltwo{\Bbeta_{t}-\Bbeta^*}^2\right]+\frac{\ca}{(t-t_2+\cb)^2}\max\{\tau^2,(1-\tau)^2\}\frac{\ku^2 b_0^2}{\kl^2 b_1^2}\frac{\ku}{\kl^2}\frac{d}{n}b_0^2\\
			\leq C^*\left(1-\frac{\ca}{24(t-t_2+\cb)}\right)\frac{\ku^2 b_0^2}{\kl^2 b_1^2}\frac{\ku}{\kl^2}\frac{1}{t-t_2+\cb}\frac{d}{n}b_0^2\leq C^*\frac{\ku^2 b_0^2}{\kl^2 b_1^2}\frac{\ku}{\kl^2}\frac{1}{t+1-t_2+\cb}\frac{d}{n}b_0^2.
		\end{multline*}
	\end{proof}
	Then we are ready to prove Theorem~\ref{thm:Batch regret}. According to the expected excess risk bound in Lemma~\ref{teclem:loss_expectation}, we have $\Bbeta^*=\arg\max_{\Bbeta}\EE\sum_{t=0}^{T} f_{t}(\Bbeta_{t})-f_t(\Bbeta)$.  In the first phase, it has
	\begin{align*}
		\EE\left[f_t(\Bbeta_{t})-f_t(\Bbeta^*)\right]
		&\leq\max\{\tau,1-\tau\}\sqrt{\ku}\left(\EE\ltwo{\Bbeta_{t}-\Bbeta^*}^2\right)^{1/2}\\
		&\leq\max\{\tau,1-\tau\}\sqrt{\ku}\left(1-c(\kl/\ku)/(\max\{\tau^2,(1-\tau)^2\})\right)^t\ltwo{\Bbeta_{0}-\Bbeta^*},
	\end{align*}
	where the first line is obtained in the same way as Theorem~\ref{thm:one_sample_regret} and the last line follows from Lemma~\ref{lem:one_sampe_expectation}. In the second phase and the third phase, according to Lemma~\ref{teclem:loss_expectation}, we have
	\begin{align*}
		\EE\left[f_t(\Bbeta_{t})-f_t(\Bbeta^*)\right]&=	\EE\left[\EE\left[f_t(\Bbeta_{t})-f_t(\Bbeta^*)\big|\Bbeta_{t}\right]\right]\leq \frac{\ku}{b_1}\EE\ltwo{\Bbeta_t-\Bbeta^*}^2.
	\end{align*}
	As a result, applying Lemma~\ref{lem:regret_expectation_batch}, the regret can ultimately be bounded by
	\begin{align*}
		&\textsf{Regret}_{T}=\EE\sum_{t=0}^{T} f_{t}(\Bbeta_{t})-f_t(\Bbeta^*)\\
		&\leq\sqrt{\ku}\ltwo{\Bbeta_{0}-\Bbeta^*}\sum_{t=0}^{t_1}\left(1-c(\kl/\ku)\right)^t\\
		&~~~~~~~~~~~~~+\frac{\ku}{b_1}\sum_{t=t_1}^{t_2} \left(1-c\frac{b_1^2}{b_0^2}\frac{\kl^2}{\ku^2}\right)^{t-t_1}\ltwo{\Bbeta_{t_1}-\Bbeta^*}^2+C\frac{\ku^3}{\kl^3}\frac{b_0^3}{b_1^3}\frac{d}{n} b_0\sum_{t=t_1}^{T}\frac{1}{t-t_1+\cb }\\
		&\leq C\frac{\ku}{\kl}\sqrt{\ku} \ltwo{\Bbeta_{0}-\Bbeta^*}+C\frac{ku^3}{\kl^3}\frac{b_0^2}{b_1^2}\gamma^2/b_1+C\frac{\ku^3}{\kl^3}\frac{b_0^3}{b_1^3}\frac{d}{n}b_0\log\left(\frac{T+1-t_1+\cb }{\cb }\right)\\
		&\leq C\frac{\ku^3}{\kl^3}\frac{b_0^2}{b_1^2}\max\{\sqrt{\ku} \ltwo{\Bbeta_{0}-\Bbeta^*}, \gamma^2/b_1\}+C\frac{\ku^3}{\kl^3}\frac{b_0^3}{b_1^3} \frac{d}{n}b_0\log\left(\frac{T+1-t_1+\cb }{\cb }\right).
	\end{align*}
	The proof is then complete.

	\subsection{Proof of Sequential Learning with Infinite Storage}
	Same as the previous settings, here the update at $t$ can be characterized as the follows,
	\begin{align}
		\ltwo{\Bbeta_{t+1}-\Bbeta^*}^2&=\ltwo{\Bbeta_{t}-\Bbeta^*}^2-2\eta_t\inp{\Bbeta_{t}-\Bbeta^*}{\g_t}+\eta_t^2\ltwo{\g_t}^2.
		\label{eq:update_infinite_storage}
	\end{align}
	We begin by establishing the regularity properties in this setting, which will be instrumental in proving convergence. According to Lemma~10 in \cite{shen2023computationally} and Lemma~\ref{teclem:sub-gradient}, under the event
	\begin{align*}
		\bcalE_t&:=\left\{\sup_{\Bbeta_1,\Bbeta_2\in\RR^d} \left|f_t(\Bbeta_1)-f_t(\Bbeta_2)-\EE\left[ f_t(\Bbeta_1)-f_t(\Bbeta_2)\right]\right|\cdot\ltwo{\Bbeta_1-\Bbeta_2}^{-1}\right.\\
		&~~~~~~~~~~~~~~~~~~~~~~~~~~~~~~~~~~~~~~~~~~~~~~~~~~~~~~~~~~~~~~~~~~~~~~~~~\left.\leq C\max\{\tau,1-\tau\}\sqrt{\ku d/\sum_{l=0}^tn_l}\right\},
	\end{align*}
	the following properties hold for any $\Bbeta\in\RR^d$ and corresponding sub-gradient $\g\in\partial f_t(\Bbeta)$:
	\begin{align}
		f_t(\Bbeta)-f_t(\Bbeta^*)\geq\frac{1}{4}\sqrt{\kl}\ltwo{\Bbeta-\Bbeta^*}- \gamma,\quad 	\ltwo{\g}\leq \sqrt{\ku},
		\label{eq:4-1}
	\end{align} 
	and simultaneously,
	\begin{equation}
		\begin{split}
			f_t(\Bbeta)-f_t(\Bbeta^*)&\geq\frac{1}{12b_0}\kl\ltwo{\Bbeta-\Bbeta^*}^2-C_1\sqrt{\ku d/\sum_{l=0}^{t}n_l}\ltwo{\Bbeta-\Bbeta^*},\\	\ltwo{\g}&\leq 2\frac{1}{b_1}\ku\ltwo{\Bbeta-\Bbeta^*}+ C_2\sqrt{\ku d/\sum_{l=0}^t n_l},
		\end{split}
		\label{eq:4-2}
	\end{equation}
	where $C>0$ is a universal constant that is independent of both the ambient dimension and the sample size. Specifically, Proposition~\ref{tecprop:empirical} shows that $$\PP(\bcalE_t)\geq 1-\exp(-C_3d)-\exp\left(-\sqrt{\sum_{l=0}^t n_l/\log\left(\sum_{l=0}^t n_l \right)}\right).$$
	\subsubsection{First Phase}
	We shall also prove the convergence dynamics via induction. The desired inequality at $\Bbeta_0$ is obvious. Then assume we already have the dynamics at $t$, namely, $\ltwo{\Bbeta_{t}-\Bbeta^*}\leq \left(1-\frac{1}{100}\frac{\kl}{\ku}\right)^t\cdot D_0$ and we are going to prove the concentration at $t+1$. For convenience, denote $$D_t:= \left(1-\frac{1}{100}\frac{\kl}{\ku}\right)^t\cdot D_0,\quad D_{t+1}:=\left(1-\frac{1}{100}\frac{\kl}{\ku}\right)^{t+1}\cdot D_0.$$ Definition of sub-gradient and event $\bcalE_t$ together infer that the intermediate term of \eqref{eq:update_infinite_storage} could be lower bounded with
	\begin{align*}
		\inp{\Bbeta_{t}-\Bbeta^*}{\g_t}&\geq f_t(\Bbeta_t)-f_t(\Bbeta^*)\\
		&\geq \EE\left[f_t(\Bbeta_t)-f_t(\Bbeta^*)\big| \Bbeta_{t} \right]-C\max\{\tau,1-\tau\}\sqrt{\ku d/\sum_{l=0}^t n_l}\ltwo{\Bbeta_t-\Bbeta^*}\\
		&\geq \frac{1}{6}\sqrt{\kl}\ltwo{\Bbeta_t-\Bbeta^*},
	\end{align*}
	where the last line uses expectation calculations, which can be found in the batch learning setting and $n_0\geq C\max\{\tau^2,(1-\tau)^2\} (\ku/\kl) d$. Moreover, Lemma~\ref{teclem:sub-gradient} proves that under $\bcalE_t$, $\ltwo{\g_t}\leq\sqrt{\ku}$. Then \eqref{eq:update_infinite_storage} could be upper bounded with
	\begin{align*}
		\ltwo{\Bbeta_{t+1}-\Bbeta^*}^2\leq \ltwo{\Bbeta_{t}-\Bbeta^*}^2-\frac{1}{3}\eta_t\sqrt{\kl}\ltwo{\Bbeta_{t}-\Bbeta^*}+\eta_t^2\ku.
	\end{align*}
	Then by the induction $\ltwo{\Bbeta_{t}-\Bbeta^*}\leq D_t$ and stepsize definition $\eta_t\in\frac{\sqrt{\kl}}{\ku} D_t\cdot[1/8,5/24]$, and by regarding the right hand side as quadratic function of $\ltwo{\Bbeta_t-\Bbeta^*}$, we have
	\begin{align*}
		\ltwo{\Bbeta_{t+1}-\Bbeta^*}^2&\leq D_t^2-\frac{1}{3}\eta_t\sqrt{\kl} D_t + \eta_t^2\ku\leq \left(1-\frac{1}{50}\frac{\kl}{\ku}\right)^{t+1}\cdot D_0^2,
	\end{align*}
	which yields $\ltwo{\Bbeta_{t+1}-\Bbeta^*}\leq \left(1-\frac{1}{100}\frac{\kl}{\ku}\right)^{t+1}\cdot D_0$.
	\subsubsection{Second Phase}
	In the second phase, the expectation of the loss function could be lower bounded with
	\begin{align*}
		\EE\left[f_t(\Bbeta_{t})-f_t(\Bbeta^*)\big| \Bbeta_t\right]\geq\frac{1}{12b_0}\kl\ltwo{\Bbeta_{t}-\Bbeta^*}^2.
	\end{align*}
	Furthermore, the definition of sub-gradient and event $\bcalE_t$ imply that
	\begin{align*}
		f_t(\Bbeta_{t})-f_t(\Bbeta^*)&\geq \frac{1}{12b_0}\kl\ltwo{\Bbeta_t-\Bbeta^*}^2-C\max\{\tau,1-\tau\}\sqrt{\ku d/\sum_{l=0}^{t}n_l}\cdot\ltwo{\Bbeta_{t}-\Bbeta^*}\\
		&\geq \frac{1}{24b_0}\kl\ltwo{\Bbeta_t-\Bbeta^*}^2- C_1b_0\frac{\ku}{\kl}\cdot\frac{d}{\sum_{l=0}^t n_l},
	\end{align*}
	where Cauchy-Schwarz inequality is used in the last line that $$\frac{1}{24b_0}\kl\ltwo{\Bbeta_t-\Bbeta^*}^2+ Cb_0\bar{\tau}^2\frac{\ku}{\kl}\frac{d}{\sum_{l=0}^t n_l}\geq\tilde{C}_1\bar{\tau}\sqrt{\ku d/\sum_{l=0}^{t}n_l}\cdot\ltwo{\Bbeta_{t}-\Bbeta^*},$$ Lemma~\ref{teclem:sub-gradient} shows that
	\begin{align*}
		\ltwo{\g_t}^2\leq 4\frac{1}{b_1^2}\ku^2\ltwo{\Bbeta_t-\Bbeta^*}^2+C_2\bar{\tau}^2\ku \cdot\frac{d}{\sum_{l=0}^t n_l}.
	\end{align*}
	Then equation \eqref{eq:update_infinite_storage} could be upper bounded with
	\begin{align*}
		\ltwo{\Bbeta_{t+1}-\Bbeta^*}^2&\leq \ltwo{\Bbeta_{t}-\Bbeta^*}^2-\frac{1}{12b_0}\eta_t\kl n_l\ltwo{\Bbeta_{t}-\Bbeta^*}^2+\eta_t^2\frac{4}{b_1^2}\ku^2\ltwo{\Bbeta_{t}-\Bbeta^*}^2\\
		&~~~~~~~~~~~~~~~~~~~~~~~~~~~~~~~~~~~~~~~~~~~+2\eta_t C_1 b_0 \frac{\ku}{\kl}\frac{d}{\sum_{l=0}^t n_l} + C_2\eta_t^2 \ku \frac{d}{\sum_{l=0}^t n_l}.
	\end{align*}
	Plug the stepsize $\eta_t\in[c_1,c_2]\cdot\frac{\kl}{\ku^2}\cdot \frac{b_1^2}{b_0}$ into the above equation and then we have
	\begin{align*}
		\ltwo{\Bbeta_{t+1}-\Bbeta^*}^2\leq \left(1-c\frac{b_1^2}{b_0^2}\cdot\frac{\kl^2}{\ku^2}\right)\cdot\ltwo{\Bbeta_{t}-\Bbeta^*}^2 + C\frac{d}{\ku\sum_{l=0}^t n_l}\cdot b_1^2,
	\end{align*}
	where $c,C$ are some universal constants. In all, we have
	\begin{align*}
		\ltwo{\Bbeta_{t+1}-\Bbeta^*}^2\leq \left(1-c\frac{b_1^2}{b_0^2}\cdot\frac{\kl^2}{\ku^2}\right)^{t+1-t_1}\cdot \ltwo{\Bbeta_{t_1}-\Bbeta^*}^2 + C\frac{d}{\ku}\cdot b_1^2\cdot\sum_{s=t_1}^{t}\frac{\left(1-c\frac{b_1^2}{b_0^2}\cdot\frac{\kl^2}{\ku^2}\right)^{t-s} }{\sum_{l=0}^{s} n_l}.
	\end{align*}
	If the number of arriving samples remains to be $n$, namely, $n_l=n$ for all $l\geq 1$, then equipped with Lemma~\ref{teclem:sequence}, the estimation error rate is upper bounded with
	\begin{multline*}
		\ltwo{\Bbeta_{t+1}-\Bbeta^*}^2\leq \left(1-2c\frac{b_1^2}{b_0^2}\cdot\frac{\kl^2}{\ku^2}\right)^{t+1-t_1}\cdot \ltwo{\Bbeta_{t_1}-\Bbeta^*}^2\\ + \tilde{C}\frac{\ku}{\kl^2}\cdot \frac{d}{n_0+tn}\cdot \frac{b_0^2}{|\log(1-c\frac{b_1^2}{b_0^2}\cdot\frac{\kl^2}{\ku^2})|},
	\end{multline*}
	where $\tilde{C}$ is some constant irrelevant with $t$. It is worth noting that when $0<x<1$, $|\log(1-x)|\geq x$ holds. Thus the last term of the above equation could be further simplified by
	\begin{align*}
		\ltwo{\Bbeta_{t+1}-\Bbeta^*}^2\leq \left(1-2c\frac{b_1^2}{b_0^2}\cdot\frac{\kl^2}{\ku^2}\right)^{t+1-t_1}\cdot \ltwo{\Bbeta_{t_1}-\Bbeta^*}^2 + \tilde{C}_1\frac{\ku}{\kl^2}\cdot \frac{d}{n_0+tn}\cdot b_0^2\cdot\left(\frac{b_0}{b_1}\cdot\frac{\ku}{\kl}\right)^2,
	\end{align*}
	with some constant $\tilde{C}_1$.
	
	\section{Technical Lemmas}
	\noindent The following lemma is derivable through integral calculations and can be found the appendix of \cite{shen2023computationally}.
	
	\begin{lemma}
		Suppose the predictors $\{\X_i\}_{i=1}^{n}$ satisfy Assumption~\ref{assump:sensing_operators:vec}. Then for any fixed vector $\Bbeta\in\RR^d$, we have \begin{align*}
			\sqrt{\frac{\kl}{2\pi}} \ltwo{\Bbeta}\leq\EE\rho_{Q,\tau}(\inp{\Bbeta}{\X_i})\leq \sqrt{\frac{\ku}{2\pi}} \ltwo{\Bbeta}.
		\end{align*}
		\label{teclem:gaussian_expecation}
	\end{lemma}
	
	\begin{lemma}
		Suppose the predictors and noise term satisfy Assumption~\ref{assump:sensing_operators:vec} and Assumption~\ref{assump:heavy-tailed} (a), respectively. And the loss function is based on one observation, $f(\Bbeta)=\rho_{Q,\tau}(Y-\X^{\top}\Bbeta)$. Then for any fixed $\Bbeta$, we have
		\begin{align*}
			\EE\left\{f(\Bbeta)-f(\Bbeta^*)\right\}\geq\sqrt{\frac{\kl}{2\pi}} \ltwo{\Bbeta-\Bbeta^*}-\gamma,
		\end{align*}
		and
		\begin{align*}
			\EE\left\{f(\Bbeta)-f(\Bbeta^*)\right\}\geq \frac{\kl}{12b_0}\ltwo{\Bbeta-\Bbeta^*}^2,\quad\text{for all }\ltwo{\Bbeta-\Bbeta^*}\leq\sqrt{\ku^{-1}}\gamma.
		\end{align*}
		\label{teclem:lowerboundofESR}
	\end{lemma}
	\begin{proof}
		Regarding the expectation, we consider the expectation conditional on $\X$ first and then consider the expectation with respect to $\xi$,
		\begin{align*}
			\EE\left\{f(\Bbeta)-f(\Bbeta^*) \right\}=\EE\left\{\EE\left\{ \rho_{Q,\tau}(\xi+\X^{\top}(\Bbeta^*-\Bbeta))-\rho_{Q,\tau}(\xi)\big|\X\right\}\right\}.
		\end{align*}
		Specifically, it's
		\begin{align*}
			\EE\left\{ \rho_{Q,\tau}(\xi+\X^{\top}(\Bbeta^*-\Bbeta))-\rho_{Q,\tau}(\xi)\big|\X\right\}=\int_{0}^{\inp{\X}{\Bbeta-\Bbeta^*}}(H_{\xi}(t)-H_{\xi}(0))\, dt
		\end{align*}
		Denote $\sigma:=((\Bbeta-\Bbeta^*)^{\top}\bSigma(\Bbeta-\Bbeta^*))^{1/2}$ and denote $g(\cdot)$ as the density $\inp{\X}{\Bbeta-\Bbeta^*}$. First, consider the case when $\sigma\leq\gamma$,
		\begin{align*}
			\EE\left\{f(\Bbeta)-f(\Bbeta^*) \right\}&=\int_{-\infty}^{+\infty}\int_{0}^{z} g(z)\cdot(H_{\xi}(t)-H_{\xi}(0))\, dt\, dz\\
			&\geq \int_{-\sigma}^{+\sigma}\int_{0}^{z} g(z)\cdot(H_{\xi}(t)-H_{\xi}(0))\, dt\, dz\geq \int_{-\sigma}^{+\sigma} g(z)\cdot\frac{1}{2b_0}z^2\, dz\\
			&=\frac{1}{2b_0}\int_{-\sigma}^{+\sigma}\frac{1}{\sqrt{2\pi}\sigma}\exp\left(-\frac{z^2}{2\sigma^2}\right)\cdot z^2\; dz \geq \frac{\sigma^2}{12 b_0}.
		\end{align*}
		Then consider another way to find the lower bound, following the triangle inequality and the expectation calculations in Lemma~\ref{teclem:gaussian_expecation},
		\begin{align*}
			\EE \{ f(\Bbeta) -f(\Bbeta^*)  \} &\geq\EE\left\{ \rho_{Q,\tau}(\X^{\top}(\Bbeta^*-\Bbeta))-\rho_{Q,\tau}(\xi)-\rho_{Q,\tau}(-\xi)\right\}\\
			&\geq \sqrt{\frac{\kl}{2\pi}}\|\Bbeta-\Bbeta^*\|-\gamma,
		\end{align*}
		which completes the proof.
	\end{proof}
	\begin{proposition}[Proposition 3 in \cite{shen2023computationally}]
		Suppose the loss function is given by the following equation
		\begin{align*}
			f(\Bbeta):=\frac{1}{n}\sum_{i=1}^{n}\rho_{Q,\tau}(Y_i-\inp{\X_i}{\Bbeta}),
		\end{align*}
		where $\{\X_i\}_{i=1}^n$ follow Assumption~\ref{assump:sensing_operators:vec} and $\{(\X_i,Y_i)\}_{i=1}^n$ are independent observations. Then there exist $C_1,C_2$ such that with probability exceeding $1-\exp\left(-C_1d\right)-\exp\left(-\sqrt{n/\log n}\right)$, the following holds for all $\Bbeta_1,\Bbeta_2\in\RR^d$,
		\begin{align*}
			\left|f(\Bbeta_1)-f(\Bbeta_2)-\EE\left(f(\Bbeta_1)-f(\Bbeta_2) \right)\right|\leq C_2\max\{\tau,1-\tau\}\sqrt{\ku\frac{d}{n}}\ltwo{\Bbeta_1-\Bbeta_2}.
		\end{align*}
		\label{tecprop:empirical}
	\end{proposition}
	
	\begin{proposition}[Concentration with Independence]
		Suppose the setting is same as Proposition~\ref{tecprop:empirical}. Then for any fixed $\Bbeta_1,\Bbeta_2$, with probability exceeding $1-2\exp\left(-C\frac{ns^2}{\ku\ltwo{\Bbeta_1-\Bbeta_2}^2}\right)$ where $C$ is some constant, it has
		\begin{align*}
			\left|f(\Bbeta_1)-f(\Bbeta_2)-\EE\left[f(\Bbeta_1)-f(\Bbeta_2) \right]\right|\leq s.
		\end{align*}
		Similarly, with probability exceeding $1-2\exp\left(-c_1\frac{ns^2}{\ku\max\{\tau^2,(1-\tau)^2\}\ltwo{\Bbeta_1-\Bbeta^*}^2}\right)$, it has
		\begin{align*}
			\frac{1}{n}\left|\sum_{i=1}^{n} \inp{\Bbeta_1-\Bbeta^*}{\X_i}\cdot g_i-\EE \inp{\Bbeta_1-\Bbeta^*}{\X_i}\cdot g_i\right|\leq s,
		\end{align*}
		where $g_i:= -\tau\mathbb{I}_{(\xi_i> \inp{\X_i}{\Bbeta_{1}-\Bbeta^*})} + (1-\tau) \mathbb{I}_{(\xi_i< \inp{\X_i}{\Bbeta_{1}-\Bbeta^*})}+\delta\mathbb{I}_{(\xi_i= \inp{\X_i}{\Bbeta_{1}-\Bbeta^*})}$ is the sub-gradient with any $\delta\in[-\tau,1-\tau]$.
		\label{tecprop:indpdt}
	\end{proposition}
	\begin{proof}
		Note that \begin{align*}
			\left\| \big| |Y_i-\inp{\X_i}{\Bbeta_1}| - |Y_i-\inp{\X_i}{\Bbeta_2}| \big| \right\|_{\Psi_2} \leq \| |\inp{\X_i}{\Bbeta_1-\Bbeta_2}| \|_{\Psi_2}\leq C\sqrt{\ku}\ltwo{\Bbeta_1-\Bbeta_2}.
		\end{align*}
		Then by sum of independent sub-Gaussians, we obtain the desired inequality. The second consequence could be obtained in a similar way,
		\begin{multline*}
			\left\|  \inp{\Bbeta_1-\Bbeta^*}{\X_i}\cdot g_i\right\|_{\Psi_2}\leq \max\{\tau,1-\tau\}\left\|  |\inp{\Bbeta_1-\Bbeta^*}{\X_i} |\right\|_{\Psi_2}\\\leq C\max\{\tau,1-\tau\}\sqrt{\ku}\ltwo{\Bbeta_1-\Bbeta_2}.
		\end{multline*}
	\end{proof}

	\noindent The following lemma provides two types of upper bound for $\ltwo{\g}$ and in the convergence analyses we would choose to use the sharpest one accordingly.
	\begin{lemma}[Upper Bound of Sub-gradient]
		Suppose the predictors and noise term satisfy Assumption~\ref{assump:sensing_operators:vec} and Assumption~\ref{assump:heavy-tailed}, respectively. And the loss function is based on $n$ independnet observations $f(\Bbeta)=\frac{1}{n}\sum_{i=1}^n\rho_{Q,\tau}(Y_i-\X_i^{\top}\Bbeta)$ with $n\geq Cd$. Suppose 
		$$
		\bcalE:= \bigg\{ \sup_{\Bbeta_1,\Bbeta_2\in\RR^d} \left|f(\Bbeta_1)-f(\Bbeta_2)-\EE\left[ f(\Bbeta_1)-f(\Bbeta_2)\right]\right|\cdot\ltwo{\Bbeta_1-\Bbeta_2}^{-1}\leq C_1\max\{\tau,1-\tau\}\sqrt{ \frac{\ku d}{n}} \bigg\}
		$$ holds. Then for all $\Bbeta$ and any sub-gradient $\g\in\partial f(\Bbeta)$, we have
		\begin{align*}
			\ltwo{\g}\leq\left\{
			\begin{array}{l}
				\sqrt{\ku}\\
				1.5\frac{\ku}{b_1}\ltwo{\Bbeta-\Bbeta^*}+C_1\max\{\tau,1-\tau\}\sqrt{\frac{\ku d}{ n}}
			\end{array}
			\right. .
		\end{align*}
		\label{teclem:sub-gradient}
	\end{lemma}
	\begin{proof}
		Under the event $\bcalE$ and with sub-gradient definition, for any $\Delta\Bbeta\in\RR^d$, we have
		\begin{align}
			\inp{\Delta\Bbeta}{\g}\leq f(\Bbeta+\Delta\Bbeta)-f(\Bbeta)&\leq\EE f(\Bbeta+\Delta\Bbeta)-\EE f(\Bbeta)+C_1\max\{\delta,1-\delta\}\sqrt{\ku\frac{d}{n}}\ltwo{\Delta\Bbeta}.
			\label{eq9}
		\end{align}
		We shall provide the two types of upper bounds respectively.
		\begin{enumerate}[1.]
			\item Note that the quantile function $\rho_{Q,\tau}(\cdot)$ satisfies $\rho_{Q,\tau}(x_1+x_2)\leq \rho_{Q,\tau}(x_1)+\rho_{Q,\tau}(x_2)$ and then for any fixed $\Bbeta,\Delta\Bbeta$, we have
			\begin{multline*}
				\EE f(\Bbeta+\Delta\Bbeta)-\EE f(\Bbeta)\\= \frac{1}{n}\sum_{i=1}^n \EE \left[ \rho_{Q,\tau}(\xi_i-\inp{\X_i}{\Bbeta+\Delta\Bbeta-\Bbeta^*})-\rho_{Q,\tau}(\xi_i-\inp{\X_i}{\Bbeta-\Bbeta^*}) \right]\\
				\leq\frac{1}{n}\sum_{i=1}^{n} \EE \rho_{Q,\tau}(-\inp{\X_i}{\Delta\Bbeta}).
			\end{multline*}
			Also notice that $-\inp{\X_i}{\Delta\Bbeta}$ is a mean zero Gaussian variable and then Assumption~\ref{assump:sensing_operators:vec} implies that $\EE \rho_{Q,\tau}(-\inp{\X_i}{\Delta\Bbeta})\leq \sqrt{\frac{2}{\pi}}\sqrt{\ku}\ltwo{\Bbeta-\Bbeta^*}$. Thus together with $n\geq C_3 \max\{\tau^2,(1-\tau)^2\} d$ we obtain $\EE f(\Bbeta+\Delta\Bbeta)-\EE f(\Bbeta)\leq\sqrt{\ku}\ltwo{\Delta\Bbeta}$. Then \eqref{eq9} becomes $$\inp{\Delta\Bbeta}{\g}\leq \sqrt{\ku}\ltwo{\Delta\Bbeta},$$
			holds for all $\Bbeta,\Delta\Bbeta$. Then insert $\Delta\Bbeta=\g$ into the above equation and we acquire $$\ltwo{\g}\leq\sqrt{\ku}.$$
			\item According to the first two equations in the proof of Lemma 11 in \cite{shen2023computationally}, it holds that
			\begin{align*}
				\EE f(\Bbeta+\Delta\Bbeta)-\EE f(\Bbeta)\leq\frac{1}{2b_1}\ku\ltwo{\Delta\Bbeta}^2+\frac{1}{b_1}\ku\ltwo{\Delta\Bbeta}\ltwo{\Bbeta-\Bbeta^*}.
			\end{align*}
			Inserting it into \eqref{eq9}, and taking $\Delta\Bbeta=\frac{b_1}{4\ku}\g$, we obtain
			\begin{align*}
				\ltwo{\g}\leq 1.5\frac{\ku}{b_1}\ltwo{\Bbeta-\Bbeta^*}+C_1\max\{\delta,1-\delta\}\sqrt{\frac{\ku d}{n}}.
			\end{align*}
			
		\end{enumerate}
	\end{proof}

	\begin{lemma}[Lemma 6 in \cite{shen2023computationally}]
		Suppose the predictors and noise term satisfy Assumption~\ref{assump:sensing_operators:vec} and Assumption~\ref{assump:heavy-tailed}(b), respectively. Let $f(\Bbeta)=(1/n)\sum_{i=1}^n\rho_{Q,\tau}(Y_i-\X_i^{\top}\Bbeta)$ for $n\geq 1$. Then, for any fixed $\Bbeta$, 
		\begin{align*}
			\EE\left[f(\Bbeta)-f(\Bbeta^*)\right]\leq \frac{\ku}{b_1}\ltwo{\Bbeta-\Bbeta^*}^2.
		\end{align*}
		\label{teclem:loss_expectation}
	\end{lemma}
	
	\begin{lemma}
		Under Assumptions~\ref{assump:sensing_operators:vec} and \ref{assump:heavy-tailed}(b), for any $\Bbeta\in\RR^d$ we have
		\begin{align*}
			\EE\inp{\Bbeta-\Bbeta^*}{\X}\cdot\partial \rho_{Q,\tau}(\xi-\inp{\X}{\Bbeta-\Bbeta^*})&\leq \frac{1}{b_1}\ku\ltwo{\Bbeta-\Bbeta^*}^2,
		\end{align*}
		where $\partial \rho_{Q,\tau}(u)=-\tau\mathbb{I}_{(u \geq 0)} + (1-\tau) \mathbb{I}_{(u< 0)}+\delta\mathbb{I}_{(u= 0)}$ is the sub-gradient with $\delta\in[-\tau,1-\tau]$.
		\label{teclem:intermediate_expectation}
	\end{lemma}
	\begin{proof}
		First consider the conditional expectation,
		\begin{align*}
			\EE\left[\inp{\Bbeta-\Bbeta^*}{\X}\cdot\partial \rho_{Q,\tau}(\xi-\inp{\X}{\Bbeta-\Bbeta^*})\big| \X \right]=\inp{\Bbeta-\Bbeta^*}{\X}\cdot\EE\left[\partial \rho_{Q,\tau}(\xi-\inp{\X}{\Bbeta-\Bbeta^*})\big|\X\right].
		\end{align*}
		Then the following equation can be obtained via some trivial calculations,
		\begin{align*}
			&~~~\EE\left[\partial \rho_{Q,\tau}(\xi-\inp{\X}{\Bbeta-\Bbeta^*})\big|\X\right]=\int_{0}^{\inp{\X}{\Bbeta-\Bbeta^*}} h_\xi(x)\; dx.
		\end{align*}
		Denote $z:=\inp{\X}{\Bbeta-\Bbeta^*}$, which follows mean zero Gaussian distributions with variance $(\Bbeta-\Bbeta^*)^{\top}\bSigma(\Bbeta-\Bbeta^*)\leq\ku\ltwo{\Bbeta-\Bbeta^*}^2$. Let $f_z(\cdot)$ be its density. Thus we obtain
		\begin{align*}
			&~~~\EE\inp{\Bbeta-\Bbeta^*}{\X}\cdot\partial \rho_{Q,\tau}(\xi-\inp{\X}{\Bbeta-\Bbeta^*})=\EE\left[\inp{\Bbeta-\Bbeta^*}{\X}\cdot\EE\left[\partial \rho_{Q,\tau}(\xi-\inp{\X}{\Bbeta-\Bbeta^*})\big|\X \right]\right]\\
			&= \int_{-\infty}^{+\infty}\int_{0}^y y f_z(y) h_\xi(x)\; dx\; dy\leq \frac{1}{b_1}\ku\ltwo{\Bbeta-\Bbeta^*}^2,
		\end{align*}
		where the last equation uses the upper bound of the density characterized in Assumption~\ref{assump:heavy-tailed}.
	\end{proof}
	\begin{lemma}
		Suppose $0<a<1$ is some fixed constant. Define the sequence with $a_n:=\sum_{k=1}^{n}\frac{a^{n-k}}{k+m}$, where $m>0$ is some integer. Then the following holds for each $n$,
		\begin{align*}
			(m+n)a_n\leq \frac{9}{1-a}\frac{1}{|\log a|}.
		\end{align*}
		\label{teclem:sequence}
	\end{lemma}
	\begin{proof}
		First consider the function $h_a(x):=xa^{x}$ with the domain $x>0$. Its first-order derivative is
		\begin{align*}
			h_a^{'}(x)=a^{x}-xa^x|\log a|=(1-|\log a|x)a^x.
		\end{align*}
		It infers that $h_a(x)$ achieves the maximum at $x=\frac{1}{|\log a|}$ and $h_a(x)$ is monotone decreasing in the region $x\in(\frac{1}{|\log a|},+\infty)$. And the $\max_{x>0} h_a(x)=\frac{1}{|\log a|}\exp(\frac{1}{|\log a|}\log a)\leq\frac{3}{|\log a|}$.
		In this way we have
		\begin{align}
			na^{[\frac{n}{2}]}\leq \frac{6}{|\log a|},\quad\text{for all } n\in\ZZ_+.
			\label{eq:9-1}
		\end{align}
		Then consider the sequence $(m+n)a_n$,
		\begin{align*}
			(m+n)a_n&=1+\frac{m+n}{m+n-1}a+\cdots+\frac{m+n}{m+1}a^{n-1}\\
			&= \underbrace{1+\frac{m+n}{m+n-1}a+\cdots+\frac{m+n}{m+n-[\frac{n}{2}]}a^{[\frac{n}{2}]}}_{A_1}  \\
			& ~~~~~ + \underbrace{\frac{m+n}{m+n-[\frac{n}{2}]-1}a^{[\frac{n}{2}]+1}+\cdots +\frac{m+n}{m+1}a^{n-1}}_{A_2}.
		\end{align*}
		As for term $A_1$, when $k\leq[\frac{n}{2}]$ it has $\frac{m+n}{m+n-k}\leq 2$ and then we obtain
		$$A_1\leq 2\left(1+a+\cdots a^{[\frac{n}{2}]}\right)\leq \frac{2}{1-a}.$$ And $A_2$ could be bounded with
		\begin{align*}
			A_2&\leq \frac{m+n}{m}a^{[\frac{n}{2}]+1}+\cdots+\frac{m+n}{m}a^{n-1}= \frac{m+n}{m}\left(a^{[\frac{n}{2}]+1}+\cdots+a^{n-1}\right)\\
			&\leq \frac{1}{1-a}\left(1+\frac{n}{m}\right) a^{[\frac{n}{2}]+1}\leq 7\frac{1}{1-a}\frac{1}{|\log a|},
		\end{align*}
		where equation~\eqref{eq:9-1} is used. Thus altogether, we have $(m+n)a_n\leq \frac{9}{1-a}\frac{1}{|\log a|}$.
	\end{proof}
	The following lemma provides the concentration of Gaussian vectors, an immediate result of Bernstein inequality (Theorem 2.8.1 of \cite{vershynin2018high}).
	\begin{lemma}
		Suppose the random vector $\X \in\RR^d$ satisfies Assumption~\ref{assump:sensing_operators:vec}. Then with probability exceeding $1-\exp(-cd)$, its $\ell_2$-norm is bounded by $\ltwo{\X}^2\leq 2\ku d$.
		\label{teclem:gaussian_vector}
	\end{lemma}
\end{document}